\documentclass{article}

\usepackage[english]{babel}
\usepackage[utf8]{inputenc}
\usepackage[T1]{fontenc}

\pdfoutput=1

\usepackage[bookmarks,bookmarksdepth=2, colorlinks=true, linkcolor=blue,citecolor=red, urlcolor=blue]{hyperref}

\usepackage{fullpage}

\usepackage{mystyle}
\usepackage{wrapfig}
\usepackage{tikz-cd}
\usepackage{multicol}
\usepackage{graphicx}
\usepackage[numbers]{natbib}
\usepackage{enumitem}

\begin{document}
\centerline{\bf EQUIVARIANT COX RING}

\bigskip

\centerline{\small ANTOINE VEZIER}

\bigskip

\begin{abstract}
We define the equivariant Cox ring of a normal variety with algebraic group action. We study algebraic and geometric aspects of this object and show how it is related to the ordinary Cox ring. Then, we specialize to the case of normal rational varieties of complexity one under the action of a connected reductive group $G$. We show that the $G$-equivariant Cox ring is then a finitely generated integral normal $G$-algebra. Under a mild additional condition, we give a presentation by generators and relations of its subalgebra of $U$-invariants, where $U$ is the unipotent part of a Borel subgroup of $G$. The ordinary Cox ring is also finitely generated and canonically isomorphic to the $U$-equivariant Cox ring, so that it inherits a canonical structure of $U$-algebra. Relying on a work of Hausen and Herppich, we prove that the subalgebra of $U$-invariants of the Cox ring is a finitely generated Cox ring of a variety of complexity one under the action of a torus. This yields in particular that this latter algebra is a complete intersection.

Characterizing the log terminality of singularities in a finitely generated Cox ring is an interesting question, particularly since the work of Gongyo, Okawa, Sannai and Takagi characterizing varieties of Fano type via singularities of Cox rings (\cite{Gongyo}). We provide a criterion of combinatorial nature for the Cox ring of an almost homogeneous $G$-variety of complexity one to have log terminal singularities.

Iteration of Cox rings has been introduced by Arzhantsev, Braun, Hausen and Wrobel in \cite{HausenIteration}. For log terminal quasicones with a torus action of complexity one, they proved that the iteration sequence is finite with a finitely generated factorial master Cox ring. We prove that the iteration sequence is finite for equivariant and ordinary Cox rings of normal rational $G$-varieties of complexity one satisfying a mild additional condition (e.g. complete varieties or almost homogeneous varieties). In the almost homogeneous case, we prove that the equivariant and ordinary master Cox rings are finitely generated and factorial.
\end{abstract}

\section{Introduction}

The \textit{Cox ring} of a normal algebraic variety $X$, denoted $\cox(X)$, is an important invariant that encodes a lot of geometric information, see \cite{coxrings} for a comprehensive reference on this rich subject. It is the ring of global sections of the \textit{Cox sheaf}, which is, roughly speaking, the direct sum 
\begin{center}
$\Rr_X:=\bigoplus_{[\Ff]\in\clg(X)}\Ff$
\end{center}
indexed by elements of the \textit{class group} of $X$, i.e. the group of isomorphism classes of divisorial sheaves on $X$. Under some additional assumptions, $\Rr_X$ can be endowed with a structure of quasi-coherent $\clg(X)$-graded $\Oo_X$-algebra, whence a structure of $\clg(X)$-graded ring on $\cox(X)$. When $\Rr_X$ is of finite type as an $\Oo_X$-algebra, its relative spectrum $\hat{X}$ is a normal algebraic variety over $X$ called the \textit{characteristic space} of $X$. The characteristic space is endowed with an action of the diagonalizable algebraic group $\Gamma_{\clg(X)}$ whose group of characters is $\clg(X)$, and the structural morphism $\hat{X}\rightarrow X$ is a good quotient for this action which induces a $\Gamma_{\clg(X)}$-torsor on an open subvariety whose complement is of codimension $\geq 2$.
Also, when the Cox ring is a finitely generated $k$-algebra, its spectrum $\tilde{X}$ is called the \textit{total coordinate space} of $X$, and the affinization morphism $\hat{X}\rightarrow \tilde{X}$ is a $\Gamma_{\clg(X)}$-equivariant open immersion whose image has a complement of codimension $\geq 2$ in $\tilde{X}$. When $X$ is smooth, the characteristic space is a $\Gamma_{\pic(X)}$-torsor over $X$ called the \textit{universal torsor}.

Given a variety with well defined Cox ring, it is a very challenging (and important) problem to find an explicit presentation of this ring by generators and relations. When the variety has a lot of symmetries, one can hope to answer this problem by taking advantage of this additional information. This can already be seen in the case of \textit{toric varieties} which is at the origin of the theory (\cite[2.2]{coxrings}). In this case, the Cox ring is a polynomial $k$-algebra in the canonical sections associated to the prime divisors lying in the complement of the torus. Hence, a toric variety admits a description as a quotient by a diagonalizable group of an open subvariety in an ambient affine space. This provides a vast generalization of the classical construction of the projective space. Toric varieties belong to the wider class of \textit{spherical varieties}, i.e. normal varieties endowed with an action of a connected reductive group such that a Borel subgroup acts with a dense orbit. These are exactly the varieties of \textit{complexity} zero, that is, the  minimal codimension of an orbit by a Borel subgroup is zero. The Cox ring of a spherical variety is also finitely generated with a known explicit description. However, this description is already much more complicated than in the case of toric varieties (\cite{Brion2007}).

More generally, when $X$ is endowed with an action of an algebraic group $G$, the general idea is to try to keep track of the $G$-action in $\cox(X)$ in order to obtain a graded $G$-algebra and use methods from representation theory to approach its structure. Geometrically, this means lifting equivariantly the $G$-action to the characteristic space. Unfortunately, such a lift does not exist in general, and even when it exists, it is in general not unique and depends on a choice (\cite[4.2.3]{coxrings}). An alternative approach which has proven useful (\cite[4.2]{Brion2007}) is to define an equivariant analogue of the Cox ring and to relate it to the usual Cox ring. Building on Brion's ideas in loc. cit., we give an intrinsic construction of the \textit{$G$-equivariant Cox sheaf}. Roughly, this is a direct sum 
\begin{center}
$\Rr^G_X:=\bigoplus_{[\Ff]\in\clg^G(X)}\Ff$
\end{center}
indexed by elements of the \textit{equivariant class group} $\clg^G(X)$, i.e. the group of isomorphism classes of $G$-linearized divisorial sheaves on $X$. The main difficulty is then to endow this sheaf with a structure of quasi-coherent $\clg^G(X)$-graded $\Oo_X$-algebra. This requires in particular to revisit the construction of the usual Cox sheaf itself (Section \ref{SecCoxRing}), and to compare it with the classical construction \cite[1.4]{coxrings}. The \textit{$G$-equivariant Cox ring}, denoted $\cox^G(X)$, is the ring of global sections of $\Rr^G_X$. When $\Rr^G_X$ is of finite type as an $\Oo_X$-algebra, its relative spectrum $\hat{X}^G$ is a normal $G\times\Gamma_{\clg^G(X)}$-variety called the \textit{$G$-equivariant characteristic space} of $X$. When moreover $\cox^G(X)$ is a finitely generated $k$-algebra, its spectrum $\tilde{X}^G$ is a normal $G\times\Gamma_{\clg^G(X)}$-variety called the \textit{$G$-equivariant total coordinate space}. These varieties sit in the diagram
\begin{center}
\begin{tikzcd}
		\hat{X}^G \arrow[d,"q", swap] \arrow[r,hook,""] &\tilde{X}^{G}\\
		X &\,\,\,,
\end{tikzcd}
\end{center}
where the affinization morphism $\hat{X}^G\xhookrightarrow{}\tilde{X}^G$ is a $G\times\Gamma_{\clg^G(X)}$-equivariant open immersion whose image has a complement of codimension $\geq 2$ in $\tilde{X}^G$. Moreover, $q$ is a $G$-equivariant good quotient of $\hat{X}^G$ by $\Gamma_{\clg^G(X)}$ and induces a $\Gamma_{\clg^G(X)}$-torsor in codimension $\geq 2$. More generally, $G$-equivariant good quotients by diagonalizable groups satisfying this last property are called \textit{$G$-equivariant diagonalizable almost principal bundles}. In Sections \ref{Sec_EqDiagTorsors} to \ref{Sec_RecognizingAnEquivariantCharacteristicSpace}, we study these bundles and derive properties of $\hat{X}^G$ and $\tilde{X}^G$ all along. 

Assuming that $G$ is connected, we establish a localization exact sequence involving the natural morphism $\clg^G(X)\rightarrow\clg^G(X_0)$ associated to the restriction to a $G$-stable open subvariety $X_0\subset X$. This exact sequence (obtained in Section \ref{SecEqClassGroup}) is the basic tool to relate $\cox^G(X)$ and $\cox^G(X_0)$, a link which is established in Section \ref{Sec_EqCoxRing_GStableOpenSubvariety}. As an application, we consider the case where $X$ is \textit{almost homogeneous}, that is, admits an open $G$-orbit $X_0\simeq G/H$. The equivariant Cox ring of $G/H$ always exists, and is the coordinate algebra of a homogeneous space $G/K$ which can be described explicitly (\ref{Ex_EqCox_HomSpace}). Still assuming that $X$ is almost homogeneous, there is a good quotient $\tilde{X}^G\rightarrow\AAA^n_k$ by $G$, where $n$ is the number of $G$-stable prime divisors in $X$ (\ref{PropCoxGFreeModuleCoxG}), and the general fibers of this morphism are isomorphic to the equivariant total coordinate space of the open orbit $X_0$ (\ref{RelationCoxEqToGInvSubset}). Furthermore, the open $G\times\Gamma_{\clg^G(X)}$-orbit $q^{-1}(X_0)$ is isomorphic to the fiber bundle
\begin{center}
$G/K\times^{\Gamma_{\hat{H}}}\Gamma_{\clg^G(X)}\rightarrow \GG_m^n$
\end{center}
associated with the $\Gamma_{\hat{H}}$-torsor $\Gamma_{\clg^G(X)}\rightarrow\GG_m^n$ (\ref{Cor_StructOpenOrbit}). This description is useful for doing computations in $\cox^G(X)$, because it can be embedded as a graded subalgebra of a coordinate algebra admitting an explicit description.

Still assuming $G$ connected, we uncover the link between $\cox^G(X)$ and $\cox(X)$ in Section \ref{Sec_RelationCoxG_To_Cox}. It turns out that we have a commutative diagram
\begin{center}
\begin{tikzcd}
		\hat{X}^G  \arrow[r,hook,""] & \tilde{X}^{G}\\
		\hat{X} \arrow[u,"",] \arrow[r,hook,""]& \tilde{X} \arrow[u,"",] \,\,,
\end{tikzcd}
\end{center}
where horizontal arrows are affinization morphisms, and vertical arrows are finite morphisms. Moreover, these finite morphisms are equivariant with respect to the morphism $\Gamma_{\clg(X)}\rightarrow \Gamma_{\clg^G(X)}$ dually defined by the forgetful morphism $\clg^G(X)\rightarrow\clg(X)$. When this last morphism is surjective (which can always be achieved up to replacing $G$ with a finite cover), the vertical arrows are closed immersions. In this case, we determine explicitly the ideal of $\cox^G(X)$ defining the subvariety $\tilde{X}$ of $\tilde{X}^G$. 

Starting from Section \ref{Sec_GeneralitiesComplexityOne}, we specialize to the case of a normal variety $X$ of complexity one under the action of a connected reductive group $G$. We show that finite generation of $\cox^G(X)$ (resp. $\cox(X)$) is equivalent to the rationality of $X$. Hence, assuming that $X$ is a rational variety, $\cox^G(X)$ is a finitely generated normal integral $\clg^G(X)$-graded $k$-algebra admitting $G$ as a group of graded automorphisms. 

A classical approach in invariant theory is to study the algebra of $U$-invariants $\cox^G(X)^U$ where $U$ is the unipotent part of a Borel subgroup of $G$. Indeed, the algebras $\cox^G(X)$ and $\cox^G(X)^U$ share many properties, and the structure of this last algebra is more accessible. Under a mild additional condition, we give in Section \ref{SecCoxU} a presentation by generators and relations of $\cox^G(X)^U$, which generalizes a result of Ponomareva (\cite[Thm 4]{Ponomareva}). When $G$ is a torus $\TT$, we use the description of $\cox(X)$ as a quotient of $\cox^\TT(X)$ in order to obtain a presentation by generators and relations of $\cox(X)$. This generalizes a result of Hausen and Süss (\cite[Thm 4.4.1.6]{coxrings}), and already demonstrates the power of equivariant Cox rings as a tool to study usual Cox rings. 

There is a canonical isomorphism between the $U$-equivariant Cox ring and the ordinary Cox ring, whence a canonical structure of $U$-algebra on $\cox(X)$. When $X$ is spherical, it follows from \cite[4.2.3]{Brion2007} and \ref{CorCoxGToCoxIII} that $\cox(X)^U$ is a finitely generated polynomial $k$-algebra. When $X$ is rational of complexity one, we show in Section \ref{Sex_CoxXU_II} that the $k$-algebra $\cox(X)^U$ is a finitely generated Cox ring of a normal variety of complexity one under the action of a torus. To prove this property, we rely on a work of Hausen and Herppich (\cite[3.4.2 and 3.4.3]{coxrings}). This yields in particular that $\cox(X)^U$ is a complete intersection.

In Section \ref{Sec_LogTermCoxRing}, we turn to the singularities of the total coordinate space of an almost homogeneous $G$-variety of complexity $\leq 1$. It is well known that Cox rings are normal integral domains. However, they can be quite singular in general. They are not necessarily Cohen-Macaulay although it has been recently shown by Braun in \cite{LBraun} that they are \textit{$1$-Gorenstein} (i.e. the canonical sheaf over the total coordinate space is invertible). Our goal is to provide a condition of combinatorial nature for the total coordinate space to have log terminal singularities. This is an interesting question, for example a $\QQ$-factorial normal projective variety $X$ is of \textit{Fano type} (i.e. there exists an effective $\QQ$-divisor $\Delta$ such that $-(K_X+\Delta)$ is ample and the pair $(X,\Delta)$ is Kawamata log terminal) if and only if its Cox ring is finitely generated with log terminal singularities (\cite[Thm 1.1]{Gongyo}). It turns out that the total coordinate space of a spherical variety is always log terminal (\ref{Prop_SphericalCoxLogTerm}). In complexity one, we obtain the following criterion (\ref{Thm_MainTheorem})
\begin{thm*}
Let $G$ be a connected reductive group, and $X$ be an almost homogeneous $G$-variety of complexity one with only constant invertible regular functions. Then,
\begin{align*}
  \tilde{X} \textrm{ has log terminal singularities} & \iff \cox(X)^U \textrm{ is a Platonic ring.}
\end{align*}
\end{thm*}
\noindent The above condition on $\cox(X)^U$ is of combinatorial nature, and has been introduced in \cite{HausenIteration}. It translates into a condition on a geometric data attached to $X$ (see \ref{Rem_Inputdata}, \ref{Rem_Platonic}).

In Section \ref{Sec_IterationCoxRings}, we study the iteration of Cox rings for normal rational varieties of complexity one. A general remark is that there is no reason a priori for a total coordinate space $\tilde{X}$ to have a trivial class group. Iteration of Cox rings roughly consists in studying the sequence of total coordinate spaces
\begin{center}
$...\rightarrow\tilde{X}^{(n)}\rightarrow\tilde{X}^{(n-1)}\rightarrow...\rightarrow\tilde{X}^{(2)}\rightarrow\tilde{X}$,
\end{center}
where $\tilde{X}^{(n)}$ denotes the total coordinate space of $\tilde{X}^{(n-1)}$. A basic question is whether this sequence is finite, that is, there exists an integer $n$ for which $\tilde{X}^{(n)}$ has either trivial class group, or a non-finitely generated Cox ring, or a Cox ring which is not well defined (i.e. $\Oo(\tilde{X}^{(n)})^*\not\simeq k^*$ and $\clg(\tilde{X}^{(n)})$ has a non-trivial torsion subgroup). If the sequence is finite, $X$ is said to have \textit{finite iteration of Cox rings}, and the last obtained Cox ring is called \textit{the master Cox ring}. This question has been studied in \cite{HausenIteration} and \cite{Wrobel1} for varieties of complexity one under a torus action. Finite iteration of Cox rings is proved in this setting, though with possibly a non-finitely generated master Cox ring. Further advances on this problem are works of Gagliardi in the spherical case (\cite{Gagliardi}), and of Braun for Fano type varieties (\cite{LBraun}). In this direction we obtain (\ref{Thm_FiniteIterationCoxRings}):
\begin{thm*}
Let $G$ be a connected reductive group, and $X$ be a normal rational $G$-variety of complexity one with only constant invertible regular functions. Suppose moreover that $X$ is complete, or that $X$ is almost homogeneous. Then, $X$ admits finite iteration of Cox rings with well defined iterated Cox rings. In the almost homogeneous case, the master Cox ring is factorial and finitely generated.
\end{thm*}

To finish, we mention that many results in this text find applications in a further article (\cite{avezier_SL2}) dedicated to almost homogeneous $\SL_2$-threefolds, for which ordinary and equivariant Cox rings are canonically isomorphic. 

\begin{ack*}
I am very grateful to Michel Brion for his help, and a very careful reading of the present work. I also thank Jürgen Hausen for interesting discussions during my stay in Tübingen, and Lukas Braun for enlightening email exchanges on a preliminary version of this paper. 
\end{ack*}

\begin{conv*}
Let $k$ denote an algebraically closed base field of characteristic zero. This last assumption is not necessary for the construction of the equivariant Cox ring, but it will simplify the exposition, and it will be needed starting from Section \ref{Sec_GeneralitiesComplexityOne}. 

In this text, we work in the category of \textit{algebraic schemes}, that is, separated $k$-schemes of finite type. Hence, a morphism $(f,f^\sharp):X \rightarrow Y$ of algebraic schemes is meant to be a $k$-morphism.  A \textit{variety} is an integral algebraic scheme. Without further precision, a \textit{point} of a variety is meant to be a closed point. The sheaf of \textit{regular functions} (or structure sheaf) on an algebraic scheme $X$ is denoted $\Oo_X$. A \textit{subvariety} of a variety $X$ is a locally closed subset equipped with its reduced scheme structure. The sheaf of units associated to $\Oo_X$ is denoted $\Oo_X^*$. The dual $\Hh om(\Ff,\Oo_X)$ of a quasi-coherent $\Oo_X$-module $\Ff$ is denoted $\Ff^\vee$. The isomorphism class of an object $A$ in a category is denoted $[A]$.
 
In this text, an \textit{algebraic group} $G$ is an affine algebraic group scheme. The multiplication morphism is denoted $m_G:G\times G\rightarrow G$, and the neutral element $e_G:\spec k\rightarrow G$.  

The \textit{character group} of an algebraic group $G$ is denoted $\hat{G}$. Given a finitely generated abelian group $M$, we let $\Gamma_M$ denote the \textit{diagonalizable (algebraic) group} $\spec(k[M])$ with character group $M$. Let $\varphi:\Gamma_1\rightarrow\Gamma_2$ be a morphism of diagonalizable groups, we let $\varphi^\sharp$ denote the dual morphism between character groups. Without further precision, the letter $\TT$ denotes an arbitrary torus.

A \textit{$G$-algebra} is a $k$-algebra with a structure of rational $G$-module such that $G$ acts by automorphisms of algebra.

An \textit{almost homogeneous variety} is a normal variety on which a connected algebraic group acts with a dense orbit. 

A \textit{quasicone} is an affine $\GG_m$-variety such that the orbit closures meet in a common point.

Let $G$ be an algebraic group, $X$ an algebraic $G$-scheme, and $q:X\rightarrow  Y$ a $G$-invariant morphism. We say that $q$ is a $G$-\textit{torsor} (or \textit{principal $G$-bundle}) over $Y$ if $q$ is faithfully flat, and the natural morphism $G\times X\rightarrow X\times_Y X$ is an isomorphism. As algebraic groups are smooth in characteristic zero, $q$ is faithfully flat if and only if it is smooth and surjective.

Let $G$ be an algebraic group and $X$ an affine algebraic $G$-scheme. If $\Oo(X)^G$ is a finitely generated $k$-algebra, then the affine algebraic scheme $X//G:=\spec(\Oo(X)^G)$ is called the \textit{categorical quotient} of $X$ by $G$. It is the universal object in the category of $G$-invariant morphisms from $X$ to affine algebraic schemes.

Our convention is that a \textit{reductive group} is a \textit{linearly reductive group}, that is, every finite dimensional representation of such group is semisimple. In particular, a reductive group is not necessarily connected. Let $G$ be a reductive group and $X$ an algebraic $G$-scheme. A \textit{good quotient} of $X$ by $G$ is an affine $G$-invariant morphism $q:X\rightarrow Y$ such that $q^\sharp$ induces an isomorphism $\Oo_Y\rightarrow (q_*\Oo_X)^G$. It is a universal object in the category of $G$-invariant morphisms from $X$ to algebraic schemes.

For a normal variety $X$, the \textit{group of Weil divisors} is denoted $\wdiv(X)$. Every Weil divisor defines a coherent sheaf $\Oo_X(D)$ whose non-zero sections over an open subset $U$ are rational functions $f\in k(X)^*$ such that $\divi(f)+D$ defines an effective divisor on $U$.

The \textit{canonical sheaf} $\omega_X$ on a normal variety $X$ is the pushforward on $X$ of the sheaf of differental forms of maximal degree on the smooth locus $X_{sm}$. It is a divisorial sheaf, and any Weil divisor $K_X$ such that $\Oo_X(K_X)\simeq\omega_X$ is a \textit{canonical divisor}. A \textit{$\QQ$-Gorenstein variety} is a normal variety such that some non-zero power of the canonical sheaf is invertible. 
If the canonical sheaf is invertible, the variety is said to be \textit{$\QQ$-Gorenstein of index one}, or simply \textit{$1$-Gorenstein}.
A \textit{Gorenstein variety} is a normal Cohen-Macaulay variety whose canonical sheaf is invertible. 

Let $X$ be a normal variety. We say that $X$ has \textit{rational singularities} if there exists a proper birational morphism
\begin{center}
$\varphi:Z\rightarrow X$,
\end{center}
where $Z$ is a smooth variety (a resolution of singularities), and such that $R^i\varphi_*\Oo_Z=0$, $\forall i>0$. This last property doesn't depend on the choice of a resolution. We say that $X$ has \textit{log terminal singularities} if the following conditions are satisfied:
\begin{itemize}[label=$\bullet$]
\item $X$ is $\QQ$-Gorenstein.
\item There is a resolution of singularities $\varphi:Z\rightarrow X$ such that
\begin{center}
$K_{Z}=\varphi^*K_X+\sum\alpha_i E_i$, $\alpha_i>-1$,
\end{center}
where the sum runs over the exceptional divisors $E_i$ of $\varphi$.
\end{itemize}
\end{conv*}

\section{Equivariant Cox ring}

\subsection{Cox ring of a pointed variety}
\label{SecCoxRing}
Considering spaces of sections of line bundles over a $G$-variety $X$, one doesn't necessarily want to see these as spaces of rational functions satisfying vanishing conditions with respect to some Cartier divisor on $X$. A prominent example of this situation is given by spaces of sections of line bundles over a flag variety, which naturally occur as simple $G$-submodules of the left regular representation. This motivates a definition of the Cox ring which is quite different from the usual one in the litterature. In \cite[1.4.2.3]{coxrings}, a canonical construction for the Cox ring of a pointed normal variety is presented. We push this idea one step further by considering rigidified sheaves at the base point (\cite[9.2.8]{Kleinman}). This leads to an intrinsic construction of the Cox ring which will be well suited for the construction of the equivariant analogue.

\begin{defn}
A \textit{pointed normal variety} is a pair $(X,x)$ consisting of a normal variety $X$, and a smooth point $x\in X$.
\end{defn}
Let $(X,x\xhookrightarrow{i}X)$ be a pointed normal variety such that $\Oo(X)^*\simeq k^*$. It will be clear from the construction below that this last hypothesis is required in order to equip the Cox ring of $(X,x)$ with a canonical structure of algebra. Typical examples of varieties satisfying this condition are normal complete varieties and almost homogeneous varieties under a semisimple algebraic group. A \textit{divisorial sheaf} on $X$ is a coherent reflexive $\Oo_X$-module of rank one. On a smooth variety, a divisorial sheaf is simply an invertible sheaf. Moreover, every divisorial sheaf on $X$ is isomorphic to the pushforward on $X$ of an invertible sheaf on the smooth locus $X_{\rm sm}$ (\cite[1.12]{HarsthorneGDivisors}). Given two divisorial sheaves $\Ff_1,\Ff_2$, the double dual of their tensor product 
\begin{center}
$\Ff_1\star\Ff_2:=(\Ff_1\otimes_{\Oo_X}\Ff_2)^{\vee\vee}$
\end{center}
is a divisorial sheaf. The operation $\star$ is compatible with taking isomorphism classes and endows the set $\clg(X)$ of isomorphism classes of divisorial sheaves with a structure of abelian group called the \textit{class group}, the inverse map being $[\Ff]\mapsto[\Ff^\vee]$. The class group identifies with the group of Weil divisors modulo linear equivalence through the natural isomorphism
\begin{center}
$[D]\mapsto [\Oo_X(D)]$  \quad(\cite[2.8]{HarsthorneGDivisors}).
\end{center}
By \cite[1.12]{HarsthorneGDivisors} again, the pullback functor associated to the open immersion of the smoooth locus $X_{\rm sm}\xhookrightarrow{} X$ induces an isomorphism of $\clg(X)$ with the \textit{Picard group} of $X_{sm}$, denoted $\pic(X_{\rm sm})$. In particular, the smoothness of $x$ yields an isomorphism of one dimensional $k$-vector spaces $i^*\Ff\rightarrow k$.

\begin{defn}
Let $(X,x\xhookrightarrow{i}X)$ be a pointed normal variety. A \textit{rigidified divisorial} sheaf is a pair $(\Ff,f)$ where $\Ff$ is a divisorial sheaf on $X$, and $f:i^*\Ff\rightarrow k$ is an isomorphism of $k$-vector spaces. The isomorphism $f$ is a \textit{rigidification} of $\Ff$. A morphism $v:(\Ff,f)\rightarrow(\Gg,g)$ of rigidified divisorial sheaves is an $\Oo_X$-module morphism preserving the rigidification. Precisely, this means that the diagram
\begin{center}
\begin{tikzcd}
		i^*\Ff \arrow[d,"f"] \arrow[r,"i^*v"] & i^*\Gg \arrow[ld,"g"]\\
		k &
\end{tikzcd} commutes.
\end{center}
\end{defn}
The set $E$ of isomorphism classes of rigidified divisorial sheaves carries a natural structure of abelian group. Indeed, the double dual of the tensor product of two rigidified divisorial sheaves is naturally equipped with a rigidification, and this operation is compatible with taking isomorphism classes. In fact, the next proposition says that we can identify the groups $E$ and $\clg(X)$.  
\begin{prop}\cite[9.2.9]{Kleinman}\label{PropIsoClgGRigididfied}
The group morphism $\rho:E\rightarrow \clg(X), [(\Ff,f)]\mapsto [\Ff]$ is an isomorphism.
\end{prop}
Also, an important feature of these rigidified divisorial sheaves is that they have no non-trivial automorphisms. Indeed, similarly as for an invertible sheaf, an automorphism of a divisorial sheaf $\Ff$ is defined by a invertible regular function on $X$. Thus, $\Aut(\Ff)\simeq k^*$ by hypothesis, and it follows that the identity is the only automorphism preserving the additional data of a rigidification. Hence, if an isomorphism exists between two rigidified divisorial sheaves, it is unique. In particular, this allows us to canonically identify sections over an open set $U\subset X$ of rigidified divisorial sheaves in a given class $[(\Ff,f)]\in E$. Formally, we have for each open subset $U\subset X$ an equivalence relation
\begin{center}
$s_1\sim_U s_2\iff s_2=v(U)(s_1)$, where $s_1\in \Ff_1(U), s_2\in \Ff_2(U)$, and $v:(\Ff_1,f_1)\xrightarrow{\simeq} (\Ff_2,f_2)$.
\end{center}
This relation is compatible with the $\Oo(U)$-module structure, and with the restrictions to open subsets. Performing the quotients with respect to the various relations $\sim_U$, we obtain a well-defined divisorial sheaf $\Ff^x$ lying in the class $[\Ff]\in\clg(X)$. Indeed, let's denote $\theta_U$ the canonical projection associated to the equivalence relation $\sim_U$. Then the isomorphisms
\begin{center}
$\Ff(U)\rightarrow \Ff^x(U), s\mapsto \theta_U(s)$
\end{center}
for the various open subsets induce an isomorphism of sheaves $\Ff\xrightarrow{\simeq}\Ff^x$. In other words, for any element $[\Ff]\in\clg(X)$ we have constructed a canonical representative $\Ff^x$, which we call the \textit{rigidified sheaf associated to} $[\Ff]$.

\begin{prop}\label{PropDefCoxSheaf}
Let $(X,x)$ a pointed normal variety such that $\Oo(X)^*\simeq k^*$. Then, the sheaf
\begin{center}
$\Rr_X^x:=\bigoplus_{[\Ff]\in \clg(X)} \Ff^x$
\end{center}
has a natural structure of quasi-coherent $\clg(X)$-graded $\Oo_X$-algebra.
\end{prop}
\begin{proof}
By construction, $\Rr^{x}_X$ is a $\clg(X)$-graded quasi-coherent $\Oo_X$-module. The multiplication between homogeneous sections is defined by the natural morphism
\begin{center}
$\Ff^x(U)\otimes_{\Oo(U)}\Gg^x(U)\rightarrow (\Ff^x\star\Gg^x)^x(U)$
\end{center}
for every open subset $U\subset X$. This multiplication endows $\Rr_X^x$ with a structure of quasi-coherent $\clg(X)$-graded $\Oo_X$-algebra.
\end{proof}

\begin{defn}\cite[1.6.1]{coxrings}\label{DefCoxSheaf}
With the assumptions of the last proposition, $\Rr_X^x$ is the \textit{Cox sheaf} of the pointed variety $(X,x)$. The \textit{Cox ring} $\cox^x(X)$ is the ring of global sections of $\Rr_X^x$. When $\Rr_X^x$ is of finite type as an $\Oo_X$-algebra, the relative spectrum $\hat{X}^x$ of the Cox sheaf is the \textit{characteristic space} of $(X,x)$. When $\cox^x(X)$ is a finitely generated $k$-algebra, its spectrum $\tilde{X}^x$ is the \textit{total coordinate space} of $(X,x)$. When no ambiguity occurs, we often drop the reference to the base point in the notation.
\end{defn}

\begin{rem}\label{RemFinitenessClg}
It is customary to suppose that $\clg(X)$ is finitely generated in the above definition. This is justified by the hope to obtain a Cox sheaf of finite type. A sufficient condition is that the Cox ring be a finitely generated $k$-algebra (\cite[1.6.1.1]{coxrings}). Normal varieties with finitely generated Cox ring are called \textit{Mori Dream Spaces (MDS)}, due to their good behaviour with respect to Mori's minimal model program (\cite{HuKeel}).
\end{rem}

Now we compare the construction above with the construction of the Cox sheaf of a pointed normal variety as defined in \cite[1.4.2.3]{coxrings}. In particular, we produce an isomorphism of graded $\Oo_X$-algebras between both constructions. We begin by recalling this last construction.
Regard $\clg(X)$ as the group of Weil divisors on $X$ modulo linear equivalence and suppose it is finitely generated. Consider the subgroup $K^x\subset\wdiv(X)$ whose elements are Weil divisors on $X$ whose support doesn't contain $x$. The restriction $K^x\rightarrow\clg(X)$ of the natural projection is still surjective. Choose a finitely generated subgroup $K\subset K^x$ which again projects surjectively on $\clg(X)$, and denote $\pi:K\rightarrow \clg(X)$ this projection. For each $E\in\ker \pi$, there exists by assumption a unique $f_E\in k(X)$ defined on a neighborhood of $x$ such that $E=\divi(f_E)$ and $f_E(x)=1$. This allows to define a group morphism $\chi: \ker\pi \rightarrow k(X)^*$ such that $\divi(\chi(E))=E$, for all $E\in \ker\pi$. The Cox sheaf of $(X,x)$ is then defined as the quotient of the $\Oo_X$-algebra $\Ss:=\bigoplus_{D\in K} \Oo_X(D)$ by the sheaf of ideals $\Ii$ generated by the elements $(1-\chi(E))_{E\in\ker \pi}$, where $1$ and $\chi(E)$ are respectively global sections of $\Oo_X$ and $\Oo_X(-E)$.

Now choose a $\ZZ$-basis $(D_1,...,D_r)$ of $K$ and rational functions $f_1,...f_r$ such that $f_i(x)=1$ and $D_{i|U}=\divi(f_i)_{|U}$ on a neighborhood $U$ of $x$. These functions exist by hypothesis on $K$, and this choice amounts to choosing rigidifications for the divisorial sheaves $\Oo_X(D_i)$, $i=1,...,r$. This induces rigidifications of the divisorial sheaves $\Oo_X(\sum_i n_i D_i)$, $(n_i)\in\ZZ^r$ in the same way. We have a canonical surjective morphism of graded $\Oo_X$-algebras
\begin{center}
$\varphi:\Ss=\bigoplus_{(n_i)\in\ZZ^r}\Oo_X(\sum_i n_i D_i)\rightarrow\Rr_X^x$, $s\in\Gamma(U,\Oo_X(\sum_i n_i D_i))\mapsto \theta_U(s)$
\end{center}
such that the induced morphism between grading groups is $\pi$. Indeed,  the equivalence relations $\sim_U$ between sections of rigidified sheaves are compatible with the algebra structure on $\Ss$. Notice that two homogeneous sections $s_1, s_2$ over $U$ have the same projection $\theta_U(s_1)=\theta_U(s_2)$ if and only if there exists $E\in\ker\pi$ such that $s_2=\chi(E)s_1$. This implies that $\ker\varphi=\Ii$, whence the sought isomorphism.

\subsection{Equivariant class group}
\label{SecEqClassGroup}

Let $G$ an algebraic group, $\sigma: G\times X\rightarrow X$ an action of $G$ on a normal variety $X$, and $p_X: G\times X\rightarrow X$ the second projection. Following \cite[Exp. I, 6.5]{SGAIII}, a $G$-linearization of a quasi-coherent sheaf $\Ff$ on $X$ is an isomorphism $\Phi:\sigma^*\Ff\xrightarrow{\simeq}p_X^*\Ff$ verifying a certain cocycle condition. The geometric incarnation of $\Ff$ is the vector fibration $\VVV(\Ff):=\spec_X(\sym_{\Oo_X}(\Ff))$ over $X$ (\cite[II.1.7.8]{EGA}). A $G$-linearization of $\Ff$ is then equivalent to an action of $G$ on $\VVV(\Ff)$ that lifts the $G$-action on $X$ and commutes with the natural $\GG_m$-action on $\VVV(\Ff)$. Because $\sigma$ and $p_X$ are flat morphisms, the associated pullbacks of coherent sheaves commute with the operation of taking duals (\cite[III.12.3.5]{EGA}). It follows that given two $G$-linearized divisorial sheaves $\Ff_1$ and $\Ff_2$ on $X$, there is a canonical $G$-linearization induced on $\Ff_1\star\Ff_2$ and $\Ff_1^{\vee}$. Also, there is an obvious notion of isomorphism of $G$-linearized sheaves, and the operation $\star$ is compatible with this notion. The induced operation on the set of isomorphism classes of $G$-linearized divisorial sheaves on $X$ yields an abelian group structure. This defines the \textit{equivariant class group}, denoted $\clg^G(X)$. This group contains the \textit{equivariant Picard group} $\pic^G(X)$, that is, the group of isomorphism classes of $G$-linearized invertible sheaves on $X$. 

\begin{prop}\label{EqIsomorphismsClg_Codim2}
Let $j:X_0\rightarrow X$ a $G$-equivariant open immersion whose image has a complement of codimension $\geq 2$ in $X$. Then, the pullback functor $j^*$ on coherent sheaves induces isomorphisms
\begin{center}
$\clg(X)\simeq \clg(X_0)$ and $\clg^G(X)\simeq\clg^G(X_0)$.
\end{center}
In particular, we have isomorphisms
\begin{center}
$\clg(X)\simeq \pic(X_{\rm sm})$ and $\clg^G(X)\simeq\pic^G(X_{\rm sm})$.
\end{center}
\end{prop}
\begin{proof}
The isomorphism $\clg(X)\simeq \clg(X_0)$ is a direct consequence of \cite[1.12]{HarsthorneGDivisors}. Furthermore, the pullback functor $j^*$ on coherent sheaves induces a well defined morphism $\clg^G(X)\rightarrow\clg^G(X_0)$ because $j$ is an equivariant flat morphism between normal $G$-varieties. Also, we have an open immersion $\Id_G\times j:G\times X_0\rightarrow G\times X$ of normal varieties whose image has a complement of codimension $\geq 2$ in $G\times X$. By \cite[1.12]{HarsthorneGDivisors} again, the pullback functor $(\Id_G\times j)^*$ induces for each divisorial sheaf $\Ff$ on $X$ a bijection between the set of $G$-linearizations of $\Ff$ and the set of $G$-linearizations of $j^*\Ff$. Thus, the morphism $\clg^G(X)\rightarrow\clg^G(X_0)$ is bijective.
\end{proof}

Assuming from now that $G$ is connected, we have the important exact sequence
\begin{equation}\label{EqSEClGroupEquivariantOubli}
0\rightarrow \Oo(X)^{*G}\rightarrow \Oo(X)^*\xrightarrow{\chi}\hat{G}\xrightarrow{\gamma} \clg^G(X)\xrightarrow{\phi} \clg(X)
\end{equation}
relating the equivariant class group to the usual one through the forgetful morphism $\phi$ that "forgets" the linearizations. This sequence is obtained from the analogous exact sequence for $\pic(X_{\rm sm})$ (\cite[4.2.2]{LinearizationGBrion}), by means of the above isomorphisms and $\Oo(X)^*\simeq \Oo(X_{\rm sm})^*$. The morphism $\gamma$ maps a character $\lambda\in\hat{G}$ to the associated $G$-linearization $\Oo_X(\lambda)$ of the structure sheaf (\cite[4.1.7]{LinearizationGBrion}), and $\chi$ maps an invertible regular function $f$ to its weight $\chi_f$ relative to the $G$-action (\cite[4.1.6]{LinearizationGBrion}).

\begin{rem}\label{RemFiniteIndexForgetGLin}
By \cite[5.2.1]{LinearizationGBrion} and \ref{EqIsomorphismsClg_Codim2}, the cokernel of $\phi$ has finite exponent bounded by the order of the Picard group of $G$ (which is finite, see \cite[5.1.3]{LinearizationGBrion}). It follows that finite generation of $\clg^G(X)$ implies finite generation of $\clg(X)$. By the above exact sequence, the converse is true. 

On several occasions, we will have to suppose that the forgetful morphism is surjective. It is apparent from the above that this condition is satisfied if we suppose that $\pic(G)=0$. This latter assumption is not a serious restriction as there exists a central isogeny $\tilde{G}\rightarrow G$ of connected algebraic groups such that $\pic(\tilde{G})=0$ (\cite[18.22]{Milne}).
\end{rem}

Consider a non-empty open subset $X_0\xhookrightarrow{j} X$, and denote $(D_i)_{i=1,..,r}$ the possibly empty family of prime divisors lying in the complement. A very useful tool for the study of class groups of normal varieties is the localization exact sequence 
\begin{equation}\label{EqSEClGroup}
0\rightarrow \Oo(X)^*\rightarrow \Oo(X_0)^*\xrightarrow{\divi} \bigoplus_{i=1}^r \ZZ D_i \xrightarrow{D\mapsto [D]}\clg(X)\xrightarrow{j^*}\clg(X_0)\rightarrow 0.
\end{equation}
Suppose that $X_0$ is $G$-stable. Then, it is natural to look for an analogous exact sequence involving the restriction morphism $\clg^G(X)\xrightarrow{j^*}\clg^G(X_0)$. Again, we can suppose that $X$ is smooth and consider (equivariant) Picard groups instead of (equivariant) class groups. The main idea is then to use a construction from Edidin and Graham in order to realize $\pic^G(X)$ as the Picard group of a variety built from $X$ in a natural way. Following \cite[Lemma 9]{EdidinGraham}, there exists a finite dimensional $G$-module $E$ and a non-empty $G$-stable open subset $E_0\subset E$ whose complement in $E$ is of codimension $\geq 2$, and such that we have a $G$-torsor $E_0\rightarrow E_0/G$.
\begin{prop}
The associated fiber bundle $f:E_0\times X\rightarrow E_0\times^G X$ exists in the category of varieties and we have an isomorphism
\begin{center}
$\pic(E_0\times^G X)\xrightarrow{\simeq}\pic^G(X)$
\end{center}
obtained from the isomorphisms $\pic^G(X)\xrightarrow{p_X^*}\pic^G(E_0\times X)$ and $\pic(E_0\times^G X)\xrightarrow{f^*}\pic^G(E_0\times X)$.
\end{prop}
\begin{proof}
That the associated fiber bundle $f$ exists in the category of varieties is a consequence of a theorem of Sumihiro (\cite[5.3.4]{LinearizationGBrion}). The isomorphism $\pic^G(X)\xrightarrow{p_X^*}\pic^G(E_0\times X)$ follows from \cite[Lemma 2]{EdidinGraham}, and the isomorphism $\pic(E_0\times^G X)\xrightarrow{f^*}\pic^G(E_0\times X)$ is a consequence of a descent result from Grothendieck (\cite[3.3.1]{LinearizationGBrion}). Considering  $(p_X^*)^{-1}\circ f^*$ yields the desired isomorphism.
\end{proof}

\begin{cor}
Let $G$ a connected algebraic group, $X$ a normal $G$-variety, $X_0\xhookrightarrow{j} X$ a non-empty $G$-stable open subvariety, and $(D_i)_{i=1,..,r}$ the family of prime divisors lying in the complement of $X_0$. Then the sequence
\begin{equation}\label{EqSEClGroupEquivariant}
0\rightarrow \Oo(X)^{*G}\rightarrow \Oo(X_0)^{*G}\xrightarrow{\divi} \bigoplus_{i=1}^r \ZZ D_i \xrightarrow{D\mapsto [D]}\clg^G(X)\xrightarrow{j^*}\clg^G(X_0)\rightarrow 0
\end{equation}
is exact.
\end{cor}
\begin{proof}
Notice already that the arrow $\bigoplus_{i=1}^r \ZZ D_i\rightarrow\clg^G(X)$ is well-defined because each prime divisor $D_i$ is $G$-invariant, hence the sheaves $\Oo_X(D_i)$ are canonically linearized. Again, we can replace $X$ by its smooth locus and consider (equivariant) Picard groups. Consider the diagram
\begin{center}
\begin{tikzcd}[column sep=small]
		0 \arrow[r, ""]  & \Oo(E_0\times^G X)^* \arrow[d,""] \arrow[r,""] & \Oo(E_0\times^G X_0)^* \arrow[d,""] \arrow[r,""] & \bigoplus_{i=1}^r\ZZ (E_0\times^G D_i) \arrow[d,"E_0\times^G D_i\mapsto D_i"] \arrow[r, ""] & \textrm{Pic}(E_0\times^G X) \arrow[d,""] \arrow[r, ""] & \textrm{Pic}(E_0\times^G  X_0) \arrow[r, ""] \arrow[d, ""] & 0 \\
		0 \arrow[r, ""]  & \Oo(X)^{*G} \arrow[r,""]  & \Oo(X_0)^{*G} \arrow[r,""]                       & \bigoplus_{i=1}^r \ZZ D_i \arrow[r,""] & \textrm{Pic}^G(X) \arrow[r,"j^*"]          & \textrm{Pic}^G(X_0) \arrow[r, ""]& 0
\end{tikzcd}
\end{center}
where the upper line is the exact sequence (\ref{EqSEClGroup}) applied to $E_0\times^G X_0\xhookrightarrow{} E_0\times^G X$, and the bottom line is the sequence (\ref{EqSEClGroupEquivariant}). The fourth and fifth vertical arrows are isomorphisms as described in the last proposition, and the first and second vertical arrows are defined as follow: notice that because $f$ is a torsor, there is an isomorphism $f^\sharp:\Oo_{E_0\times^G X}\rightarrow (f_*\Oo_{E_0\times X})^G$, hence the ring morphism $f^\sharp(E_0\times^G X)$ induces an isomorphism $\Oo(E_0\times^G X)^*\simeq \Oo(E_0\times X)^{*G}$. Also, because $E_0$ has a complement in $E$ of codimension $\geq 2$ and $E$ is an affine space, the projection $p_X:E_0\times X\rightarrow X$ induces an isomorphism
\begin{center}
$p^\sharp_X(X):\Oo(X)^{*G}\rightarrow \Oo(E_0\times X)^{*G}$\end{center}
between groups of $G$-invariant units. The first vertical arrow is then the isomorphism
\begin{center}
$(p_X^\sharp(X))^{-1}\circ f^\sharp(E_0\times^G X)$.
\end{center}
The second vertical arrow is the isomorphism obtained the same way but replacing $X$ by $X_0$. The third arrow is obviously an isomorphism between these two free abelian groups of rank $r$. The diagram is commutative by construction, and the second arrow is exact because all vertical arrows are isomorphisms.
\end{proof}

\subsection{Construction of the equivariant Cox ring}
\label{SecCoxRingEq}

Let $G$ an algebraic group, $\sigma: G\times X\rightarrow X$ an action of $G$ on a pointed normal variety $(X,x\xhookrightarrow{i}X)$ such that $\Oo(X)^{*G}\simeq k^*$, and $p_X: G\times X\rightarrow X$ the second projection. Natural examples of such varieties are almost homogeneous varieties, and normal complete $G$-varieties. Below, we follow the lines of the construction presented in Section \ref{SecCoxRing}, but considering $G$-linearized divisorial sheaves on $X$. In particular, this will motivate the assumption $\Oo(X)^{*G}\simeq k^*$.

Similarly as before, consider the category of rigidified $G$-linearized divisorial sheaves on $X$ whose morphisms preserve linearizations and rigidifications. Let $E^G$ be the set of isomorphism classes associated to this category, which again carries a natural structure of abelian group. 
\begin{prop}\label{PropIsoClgGRigididfiedEq}
The group morphism $\rho:E^G\rightarrow \clg^G(X), [(\Ff,f)]\mapsto [\Ff]$ is an isomorphism.
\end{prop}
\begin{proof}
We adapt the proof \cite[9.2.9]{Kleinman} of Proposition \ref{PropIsoClgGRigididfied}. The map $\rho$ obviously defines a surjective group morphism. To check injectivity, consider an element $[(\Ff,f)]\in \ker \rho$. Then, there is an isomorphism of $G$-linearized divisorial sheaves $v:\Ff\rightarrow\Oo_X$, where $\Oo_X$ is endowed with the trivial linearization. Consider the automorphism $w:=i^*v\circ f^{-1}$ of $k$ viewed as a $k$-vector space. Let $\varphi\rightarrow \spec k$ the structural morphism of $X$, the pullback $\varphi^*w:\Oo_X\rightarrow\Oo_X$ of $w$ is the automorphism of $\Oo_X$ (as a $G$-linearized sheaf) defined by the multiplication by the constant function $w(1_k)$. By construction, we have $(\Ff, f)\simeq (\Oo_X, w^{-1})$ through $v$, and $(\Oo_X, \Id_k)\simeq (\Oo_X, w^{-1})$ through $\varphi^*w$. This proves $[(\Ff,f)]=[(\Oo_X,\Id_k)]$.
\end{proof}
Given a $G$-linearized divisorial sheaf $\Ff$ on $X$, we have $\Aut(\Ff)\simeq \Oo(X)^{*G}$. Because $\Oo(X)^{*G}\simeq k^*$ by assumption, we can  identify sections of rigidified $G$-linearized divisorial sheaves lying in a given class $[(\Ff,f)]\in E^G$. Similarly as in \ref{SecCoxRing}, this identification is done using the equivalence relations $\sim^G_U$ indexed by open subsets $U\subset X$. This builds up a canonical representative $\Ff^x$ of the corresponding class $[\Ff]\in\clg^G(X)$ which we call the rigidified $G$-linearized sheaf associated to $[\Ff]$.

\begin{prop}\label{PropDefCoxSheafGeq}
The sheaf
\begin{center}
$\Rr_X^{G,x}:=\bigoplus_{[\Ff]\in \clg^G(X)} \Ff^x$
\end{center}
has a natural structure of quasi-coherent $\clg^G(X)$-graded $\Oo_X$-algebra. If this $\Oo_X$-algebra is of finite type, its relative spectrum $\hat{X}^{G,x}:=\spec_X(\Rr_X^{G,x})$ is an algebraic scheme endowed with a $G$-action that commutes with the natural $\Gamma_{\clg^G(X)}$-action, and that lifts the $G$-action on $X$.
\end{prop}
\begin{proof}
By construction, $\Rr^{G,x}_X$ is a $\clg^G(X)$-graded quasi-coherent $\Oo_X$-module. The multiplication between homogeneous sections is defined by the natural morphism
\begin{center}
$\Ff^x(U)\otimes_{\Oo(U)}\Gg^x(U)\rightarrow (\Ff^x\star\Gg^x)^x(U)$
\end{center}
for every open subset $U\subset X$. This multiplication equips $\Rr^{G,x}_X$ with a structure of quasi-coherent $\clg^G(X)$-graded $\Oo_X$-algebra. 

Suppose that $\Rr_X^{G,x}$ is of finite type. Then, the relative spectrum $\hat{X}^{G,x}$ is by definition an algebraic scheme which comes with a morphism $q:\hat{X}^{G,x}\rightarrow X$. Moreover, the grading group $\clg^G(X)$ has to be finitely generated, thus $\Gamma_{\clg^G(X)}$ is a diagonalizable group  having a natural action on $\hat{X}^{G,x}$. Now, consider the direct sum 
\begin{center}
$\Phi:=\bigoplus_{[\Ff]\in \clg^G(X)}\Phi_{\Ff^x}$
\end{center}
of the isomorphisms $\Phi_{\Ff^x}:\sigma^*\Ff^x\rightarrow p_X^*\Ff^x$ given by $G$-linearizations of the divisorial sheaves $\Ff^x$. By construction, $\Phi$ defines an isomorphism of  $\clg^G(X)$-graded $\Oo_{G\times X}$-algebra $\sigma^*\Rr^{G,x}\rightarrow p_X^*\Rr^{G,x}$. By virtue of \cite[II.1.5.5]{EGA}, we can identify the morphism $\spec_{G\times X}(\sigma^*{\Rr^{G,x}})\rightarrow G\times X$ with the base change of $q$ along the morphism $G\times X\xrightarrow{\sigma} X$. In the same way, we identify $\spec_{G\times X}(p_X^*{\Rr^{G,x}})$ with $G\times \hat{X}^{G,x}$ as schemes over $G\times X$. Hence, we obtain a cartesian square
\begin{center}
\begin{tikzcd}
		G\times\hat{X}^{G,x} \arrow[d,"", swap] \arrow[r,""] \arrow[dr, phantom, "\square"] &\hat{X}^{G,x} \arrow[d,"q"]\\
		G\times X \arrow[r,"\sigma", swap] & X\,.
\end{tikzcd}
\end{center}
The $G$-linearizations of the homogeneous components yield a cocycle condition for the isomorphism $\Phi$, hence a $G$-linearization of $\Rr^{G,x}$. It follows that the morphism $G\times\hat{X}^{G,x}\rightarrow \hat{X}^{G,x}$ satisfies the associativity axiom of a group action. By the diagram, this action lifts the $G$-action on the base $X$. Finally, this action commutes with the $\Gamma_{\clg^G(X)}$-action because the automorphism of $\clg^G(X)$ induced by $\Phi$ through its action on the grading group is the identity.
\end{proof}

\begin{defn}\label{DefCoxSheafEq}
The \textit{equivariant Cox sheaf} of $(X,x)$ is the $\clg^G(X)$-graded $\Oo_X$-algebra $\Rr_X^{G,x}$. The \textit{equivariant Cox ring} $\cox^{G,x}(X)$ is the ring of global sections of $\Rr_X^{G,x}$. When $\Rr_X^{G,x}$ is of finite type as an $\Oo_X$-algebra, the relative spectrum $\hat{X}^{G,x}$ over $X$ of the Cox sheaf is the \textit{equivariant characteristic space} of $(X,x)$. When $\cox^{G,x}(X)$ is a finitely generated $k$-algebra, its spectrum $\tilde{X}^{G,x}$ is the \textit{equivariant total coordinate space} of $(X,x)$. When no ambiguity occurs, we often drop the reference to the base point in the notation.
\end{defn}

\begin{rem}\label{Rem_CoxEqIsomCox_SemisimpleSimplyConnected}
By the exact sequence (\ref{EqSEClGroupEquivariantOubli}) and \ref{RemFiniteIndexForgetGLin}, the equivariant Cox ring and the usual Cox ring are canonically isomorphic if $G$ is a semisimple and simply connected algebraic group.
\end{rem}

\subsection{Equivariant diagonalizable torsors}
\label{Sec_EqDiagTorsors}

Let $G$ be an algebraic group, $\Gamma$ a diagonalizable group, and $(X,x)$ a pointed normal $G$-variety. 

\begin{defn}
A \textit{diagonalizable torsor} is a torsor (or principal bundle) under a diagonalizable group. Let $Y$ be a $G\times\Gamma$-variety, and $q:Y\rightarrow X$ a diagonalizable $\Gamma$-torsor. Then $q$ is a \textit{$G$-equivariant diagonalizable $\Gamma$-torsor} over $X$ if the morphism $q$ is $G$-equivariant.
\end{defn}

The $G$-equivariant diagonalizable $\Gamma$-torsors over $X$ naturally form a category, taking for arrows the $G\times\Gamma$-equivariant morphisms over $X$. In this Section, we classify $G$-equivariant diagonalizable torsors over $X$ up to isomorphism, assuming $X$ admits only constant invariant invertible regular functions. The next proposition gives the general form of such a torsor. Recall that a torsor by an (affine) algebraic group is necessarily an affine invariant morphism. 

\begin{prop}\label{Prop_GenForm_DiagTorsor}
Let $Y$ be an algebraic $\Gamma$-scheme, and $q:Y\rightarrow X$ be a $\Gamma$-invariant affine morphism. Then, $Y$ is the relative spectrum over $X$ of a quasi-coherent $\hat{\Gamma}$-graded $\Oo_X$-algebra $\Aa^G$ of finite type. Moreover, $q$ defines a $G$-equivariant $\Gamma$-torsor over $X$ if and only if
\begin{enumerate}
\item $\forall \chi\in \hat{\Gamma}$, $\Aa^G_\chi$ is a $G$-linearized invertible sheaf.
\item $\forall \chi,\lambda\in\hat{\Gamma}$, the natural morphism $\Aa^G_\chi\otimes_{\Oo_X}\Aa^G_\lambda \rightarrow \Aa^G_{\chi+\lambda}$ is an isomorphism of $G$-linearized invertible sheaves.
\end{enumerate}
\end{prop}
\begin{proof}
By \cite[Exp. 8, 4.1]{SGAIII}, the morphism $q$ is a $\Gamma$-torsor if and only if the homogeneous components of $\Aa^G$ are invertible sheaves and the natural morphisms $\Aa^G_\chi\otimes_{\Oo_X}\Aa^G_\lambda \rightarrow \Aa^G_{\chi+\lambda}$ are isomorphisms of invertible sheaves. Proceeding as in the proof of \ref{PropDefCoxSheafGeq}, we obtain that the conditions on $G$-linearizations are necessary and sufficient to obtain a $G$-action on $Y$ with the desired properties.
\end{proof}

Denote $H^1_G(X,\Gamma)$ the set of isomorphism classes of $G$-equivariant $\Gamma$-torsor over $X$. The assignment 
\begin{center}
$\Gamma\mapsto H^1_G(X,\Gamma)$
\end{center}
naturally defines a functor from the category of diagonalizable groups to the category of pointed sets, where origins are taken to be the classes of trivial $G$-equivariant diagonalizable torsors . Indeed, consider a morphism of diagonalizable groups $\varphi_1:\Gamma_1\rightarrow \Gamma_2$, and define $H^1_G(X,\varphi_1)$ as follows: Let $q:Y\rightarrow X$ a $G$-equivariant $\Gamma_1$-torsor over $X$, and $Y\times^{\Gamma_1}\Gamma_2$ the associated fiber bundle over $X$, where $\Gamma_1$ acts on $\Gamma_2$ via $\varphi_1$. This fiber bundle exists as a (normal) algebraic scheme (\cite[3.3.3]{LinearizationGBrion}), and sits in the cartesian square
\begin{center}
	\begin{tikzcd}
		Y\times \Gamma_2 \arrow[r,"p_Y"] \arrow[d,  "/\Gamma_1"] & Y \arrow[d, "q", swap]\\
		Y\times^{\Gamma_1}\Gamma_2 \arrow[r, "\psi", swap] & X\,,
	\end{tikzcd}
\end{center}
in which all arrows are $G$-equivariant. Moreover, $\psi$ defines a $\Gamma_2$-torsor as it pullbacks to the trivial torsor along the faithfully flat morphism $q$. This construction is compatible with taking isomorphism classes, thus it defines a map
\begin{center}
$H^1_G(X,\varphi_1):H^1_G(X,\Gamma_1)\rightarrow H^1_G(X,\Gamma_2)$.
\end{center}
Moreover, considering a second morphism $\varphi_2:\Gamma_2\rightarrow\Gamma_3$ of diagonalizable groups, we have
\begin{center}
$H^1_G(X,\varphi_2\circ\varphi_1)=H^1_G(X,\varphi_2)\circ H^1_G(X,\varphi_1)$.
\end{center}
The next proposition says that the functor $H^1_G(X,.)$ is a group functor.

\begin{prop}
Let $\Gamma,\Gamma_1,\Gamma_2$ be diagonalizable  groups. We have a canonical bijection
\begin{center}
$H^1_G(X,\Gamma_1)\times H^1_G(X,\Gamma_2)\rightarrow H^1_G(X,\Gamma_1\times\Gamma_2)$, $([Y_1],[Y_2])\mapsto [Y_1\times_X Y_2]$,
\end{center}
where $\Gamma_1\times\Gamma_2$ acts on $Y_1\times_X Y_2$ componentwise. Moreover the canonical map
\begin{center}
$H^1_G(X,\Gamma)\times H^1_G(X,\Gamma)\rightarrow H^1_G(X,\Gamma\times\Gamma)\xrightarrow{H^1_G(X,m_\Gamma)} H^1_G(X,\Gamma)$
\end{center}
endow $H^1_G(X,\Gamma)$ with the structure of an abelian group for which the neutral element is the class of the trivial $G$-equivariant $\Gamma$-torsor over $X$. Finally, considering a morphism of diagonalizable groups $\varphi:\Gamma_1\rightarrow \Gamma_2$, the map $H^1_G(X,\varphi)$ is a group morphism.
\end{prop}
\begin{proof}
The first map is well defined and has an inverse. Indeed, denote $p_i$ the two projections $\Gamma_1\times\Gamma_2\rightarrow\Gamma_i$. Then, the map which associates to a class $[Y]$ of $G$-equivariant $\Gamma_1\times\Gamma_2$-torsor over $X$, the pair 
\begin{center}
$(H^1_G(X, p_1)([Y]),H^1_G(X, p_2)([Y]))$ 
\end{center}
is the sought inverse. Hence, $H^1_G(X, .)$ commutes with finite products. As a consequence, this functor sends groups to groups. But diagonalizable groups are group objects in their category, whence the claimed structure of abelian group on $H^1_G(X,\Gamma)$. The last assertion is directly checked.
\end{proof}

In order to determine the group $H^1_G(X,\Gamma)$, the last proposition allows to reduce to the situations where $\Gamma=\GG_m$ and $\Gamma=\mu_n$. As a consequence of \ref{Prop_GenForm_DiagTorsor}, we have an isomorphism
\begin{equation}\label{Eq_SE_TorsorEqClassificationI}
\pic^G(X)\xrightarrow{\simeq}H^1_G(X,\GG_m), [\Ll]\mapsto [\spec_X(\bigoplus_{k\in\ZZ}\Ll^{\otimes k})].
\end{equation}
On the other hand, the exact sequence
\begin{center}
$1\rightarrow \mu_n\xrightarrow{i} \GG_m\xrightarrow{t\mapsto t^n} \GG_m\rightarrow 1$
\end{center}
gives rise to an exact sequence of abstract groups
\begin{center}
$1\rightarrow \Oo(X)^*/(\Oo(X)^*)^n\rightarrow H^1(X,\mu_n)\rightarrow \pic(X)[n]\rightarrow 1$,
\end{center}
where $H^1(X,\mu_n)$ denotes the group of isomorphism classes of  $\mu_n$-torsors over $X$, and $\pic(X)[n]$ denotes the $n$-torsion subgroup of $\pic(X)$ (\cite[Exp. XI, 6.4]{SGAI}). It follows that a $\mu_n$-torsor $Y$ over $X$ is defined by choosing an $n$-torsion element $[\Ll]$ in $\pic(X)$, an isomorphism $f:\Ll^{\otimes n}\rightarrow \Oo_X$, and setting
\begin{center}
$Y:=\spec_X\bigoplus_{i=0}^{n-1}\Ll^{\otimes i}$,
\end{center}
where the $\Oo_X$-algebra structure is defined by the choice of $f$. From this fact and \ref{Prop_GenForm_DiagTorsor}, we deduce the exact sequence
\begin{center}
$1\rightarrow \Oo(X)^{*G}/(\Oo(X)^{*G})^n\rightarrow H_G^1(X,\mu_n)\rightarrow \pic^G(X)[n]\rightarrow 1$.
\end{center}
Assuming $\Oo(X)^{*G}\simeq k^*$, this exact sequence yields an isomorphism 
\begin{equation}\label{Eq_SE_TorsorEqClassificationII}
\pic^G(X)[n]\xrightarrow{\simeq}H^1_G(X,\mu_n), [\Ll]\mapsto [\spec_X(\bigoplus_{i=0}^{n-1} (\Ll^{\otimes i})^x],
\end{equation}
where the algebra structure is given by the multiplication between sections of rigidified $G$-linearized invertible sheaves. Using ($\ref{Eq_SE_TorsorEqClassificationI}$) and ($\ref{Eq_SE_TorsorEqClassificationII}$), we give a refined version of \ref{Prop_GenForm_DiagTorsor}.

\begin{prop}\label{Prop_GenForm_DiagTorsorII}
Suppose that $\Oo(X)^{*G}\simeq k^*$, and fix an isomorphism $\hat{\Gamma}\simeq \ZZ^r\times\ZZ/n_1\ZZ\times...\times\ZZ/n_s\ZZ$. Then a $G$-equivariant $\Gamma$-torsor $Y$ over $X$ is defined up to isomorphism by choosing 
\begin{itemize}
\item a $r$-tuple $([\Ll_1],...,[\Ll_r])\in\pic^G(X)\times...\times\pic^G(X)$,
\item a $s$-tuple $([\Ll_{r+1}],...,[\Ll_{r+s}])\in \pic^G(X)[n_1]\times...\times\pic^G(X)[n_s]$,
\end{itemize}
and setting  $Y:=\spec_X\Aa^{G,x}$ where $\Aa^{G,x}$ is the $\hat{\Gamma}$-graded $\Oo_X$-algebra
\begin{center}
$\bigoplus_{(k_i)\in\ZZ^r\times\ZZ/n_1\ZZ\times ... \times\ZZ/n_s\ZZ}\,\,\, (\Ll_1^{\otimes k_1})^x\otimes...\otimes(\Ll_{r+s}^{\otimes k_{r+s}})^x$,
\end{center}
the algebra structure being given by the multiplication between sections of rigidified $G$-linearized invertible sheaves.
\end{prop}

\begin{cor}\label{Cor_PicG_RepresentH1}
Suppose that $\Oo(X)^{*G}\simeq k^*$. Then, there is a natural isomorphism of abelian groups
\begin{center}
$H^1_G(X,\Gamma)\simeq \Hom(\hat{\Gamma},\pic^G(X))$.
\end{center}
\end{cor}

\subsection{Equivariant almost principal bundles}
\label{Sec_ClassificationAlmostPrincipalBundles}
Let $G$ an algebraic group, $\Gamma$ a diagonalizable group, and $(X,x)$ a pointed normal $G$-variety. In \cite[0.4]{Hashimoto}, Hashimoto introduced the notion of an \textit{almost principal bundle} under an algebraic group. We recall this notion and define a $G$-equivariant analogue.

\begin{defn}\label{DefAlmostTorsor}
Let $H$ an algebraic group, $Y$ an algebraic $H$-scheme, and $q:Y\rightarrow X$ an $H$-invariant morphism. We say that $q$ is an \textit{almost principal $H$-bundle} over $X$ if there exists ($H$-stable) open subschemes $Y_0\subset Y$, $X_0\subset X$ whose respective complements are of codimension $\geq 2$ and such that $q$ induces an $H$-torsor $Y_0\rightarrow X_0$. If moreover, $Y$ is a $G\times H$-scheme and $q$ is $G$-equivariant, we say that $q$ is a \textit{$G$-equivariant almost principal $H$-bundle over $X$}.
\end{defn}

\begin{prop}\label{Prop_ClgG_CommutativeSquare}
Let $H$ an algebraic group, and $q:Y\rightarrow X$ a $G$-equivariant almost principal $H$-bundle. Suppose that $G$ is connected, and $Y$ is a normal variety. Then there is a commutative diagram 
\begin{center}
	\begin{tikzcd}
		\textrm{ }\clg^{G\times H}(Y) \arrow[r,"\phi_1"]  & \textrm{ }\clg^{H}(Y)   \\
		\textrm{ }\clg^{G}(X) \arrow[u,  "q^*"] \arrow[r, "\phi_2"] &\textrm{ } \clg(X) \arrow[u, "q^*"] \,\,,
	\end{tikzcd}
\end{center}
where vertical arrows are induced by the pullback $q^*$ of (equivariant) divisorial sheaves, and horizontal arrows are given by forgetting linearizations. Moreover, vertical arrows are isomorphisms, with inverse morphisms $q_*^H$ explicitly given by 
\begin{center}
$[\Ff]\mapsto [(q_*\Ff)^H]$.
\end{center}
\end{prop}
\begin{proof}
Using \ref{EqIsomorphismsClg_Codim2}, we can suppose that $q$ is a $H$-torsor, that $X,Y$ are smooth, and consider (equivariant) Picard groups instead of (equivariant) class groups. Notice that because $q$ is $H$-invariant and $G$-equivariant, the two vertical arrows are well defined. Also, the commutativity of the diagram is directly checked. The isomorphism $q^*:\pic(Y)\rightarrow\pic^H(X)$ is a consequence of a descent result from Grothendieck (\cite[3.3.1]{LinearizationGBrion}). Now, we show that the left vertical arrow is an isomorphism. Let $\Ll\in\ker q^*$, then we have that $\Ll\simeq\Oo_X$ as invertible sheaves. Because of the exact sequence (\ref{EqSEClGroupEquivariantOubli}), we have $\Ll\simeq\Oo_X(\lambda)$ as $G$-linearized invertible sheaves, for a certain character $\lambda\in\hat{G}$. But as $q^*\Oo_X(\lambda)\simeq\Oo_Y(\lambda)$, it follows that $\lambda=0$ which proves injectivity. For surjectivity it suffices, by the exact sequence (\ref{EqSEClGroupEquivariantOubli}) again, to check that all the elements $[\Oo_Y(\lambda)]\in\pic^{G\times H}(Y)$ with $\lambda\in \hat{G}$ are reached by $q^*$. This is obvious as $q^*\Oo_X(\lambda)\simeq\Oo_Y(\lambda)$. The last claim is a direct consequence of the fact that $q$ is a torsor and that invertible sheaves are locally trivial.
\end{proof}

In the framework of Cox rings, we consider almost principal bundles under diagonalizable groups. In fact, we need a stronger notion in order to obtain a generalization of \ref{Prop_GenForm_DiagTorsorII}. 

\begin{defn}\label{DefAlmostTorsorDiag}
A \textit{diagonalizable almost principal bundle over $X$} is a good quotient $q:Y\rightarrow X$ of a normal algebraic scheme $Y$ by a diagonalizable group such that $q$ is an almost principal bundle.
\end{defn}

\begin{prop}\label{Prop_Forme_Gen_AlmostBundle}
Consider a $G$-equivariant diagonalizable almost principal $\Gamma$-bundle $q:Y\rightarrow X$. Then, $Y$ is the relative spectrum over $X$ of a $\hat{\Gamma}$-graded quasi-coherent $\Oo_X$-algebra $\Aa^G$ satisfying
\begin{itemize}
\item $\forall \chi\in \hat{\Gamma}$, $\Aa^G_\chi$ is a $G$-linearized divisorial sheaf.
\item $\forall \chi,\lambda\in\hat{\Gamma}$, the natural morphism $\Aa^G_\chi\otimes_{\Oo_X}\Aa^G_\lambda \rightarrow \Aa^G_{\chi+\lambda}$ induces an isomorphism of $G$-linearized sheaves over the smooth locus of $X$.
\end{itemize}
Conversely, given a quasi-coherent $\hat{\Gamma}$-graded $\Oo_X$-algebra $\Aa^G$ of finite type satisfying the two conditions above, the relative spectrum over $X$ of $\Aa^G$ defines a $G$-equivariant diagonalizable almost principal $\Gamma$-bundle.
\end{prop}
\begin{proof}
Because $q$ is a good quotient by $\Gamma$, we have an isomorphism $Y\simeq \spec_X\Aa^G$, where $\Aa^G:=q_*\Oo_Y$ is a $\hat{\Gamma}$-graded quasi-coherent $\Oo_X$-algebra. By \ref{Prop_GenForm_DiagTorsor}, the homogeneous components of $\Aa^G$ define invertible sheaves over the locus $X_0\xhookrightarrow{i} X$ where $q$ is a torsor, and the natural morphisms $\Aa^G_\chi\otimes_{\Oo_X}\Aa^G_\lambda \rightarrow \Aa^G_{\chi+\lambda}$ induce isomorphisms of invertible sheaves over $X_0$. Also, because $Y$ is a normal algebraic scheme, we have isomorphisms $i_*i^*\Aa^G_\lambda\simeq\Aa^G_\lambda$ for each homogeneous component. It follows that
\begin{center}
$(\Aa^G_\lambda)^{\vee\vee}\simeq i_*i^*(\Aa^G_\lambda)^{\vee\vee}\simeq i_*((i^*\Aa^G_\lambda)^{\vee\vee})\simeq  i_*i^*\Aa^G_\lambda\simeq\Aa^G_\lambda$,
\end{center}
where the first isomorphism comes from the fact that $(\Aa^G_\lambda)^{\vee\vee}$ is reflexive (\cite[1.11]{HarsthorneGDivisors}), the second holds because the pullback by a flat morphism commutes with taking duals, and the third holds because $i^*\Aa^G_\lambda$ is invertible, hence reflexive. Thus, the homogeneous components of $\Aa^G$ are divisorial sheaves on $X$. Also, we see that $X_0$ can be taken to be the smooth locus of $X$. Indeed, the divisorial sheaves $\Aa^G_\lambda$ are invertible over $X_{\rm sm}$ and the natural morphisms $\Aa^G_\chi\otimes_{\Oo_X}\Aa^G_\lambda \rightarrow \Aa^G_{\chi+\lambda}$ induce isomorphisms of $G$-linearized invertible sheaves over $X_{\rm sm}$. Finally, the conditions on $G$-linearizations are necessary and sufficient to have a $G$-action on $Y$ with the desired properties.

Conversely, consider a quasi-coherent $\hat{\Gamma}$-graded $\Oo_X$-algebra $\Aa^G$ of finite type satisfying the two conditions of the statement, and denote $q:Y\rightarrow X$ the affine morphism defined by its relative spectrum over $X$. By \ref{Prop_GenForm_DiagTorsor}, the restriction of $q$ to $q^{-1}(X_{\rm sm})$ defines a $G$-equivariant $\Gamma$-torsor over $X_{\rm sm}$. This yields an isomorphism $\Oo_X\simeq\Aa_0$, thus $q$ is a good quotient. To prove that we have a $G\times\Gamma$-action on $Y$ such that $q$ is $G$-equivariant, we proceed exactly as in the proof of \ref{PropDefCoxSheafGeq}. Now, we show that $Y$ is a normal algebraic scheme. For this, we verify that $\Aa^G$ is a sheaf of normal algebras, thus we can suppose that $X$ is smooth. Then, $q$ is smooth as well because it is a torsor under a smooth algebraic group. As a consequence, the stalks of $\Aa^G$ are regular local rings, hence are integrally closed. It remains to prove that the complement of $q^{-1}(X_{\rm sm})$ in $Y$ is of complement $\geq 2$. Because $q$ is affine, we can suppose that $X$ and $Y$ are affine. Also, as $Y$ is normal, it is a disjoint union $Y=\cup_{i=1}^r Y_i$ of normal affine varieties, and we have canonical isomorphisms
\begin{align*}
  \Oo_Y(q^{-1}(X_{\rm sm})) & \simeq \prod_{i=1}^r\Oo_{Y_i}(q^{-1}(X_{\rm sm})\cap Y_i)\\
  		& \simeq \Aa^G(X_{\rm sm})\\
  		& \simeq \Aa^G(X)\\
  		& \simeq \Oo_Y(Y)\\
    	& \simeq \prod_{i=1}^r\Oo_{Y_i}(Y_i).
\end{align*}
It follows from these isomorphisms and the next lemma that $q^{-1}(X_{\rm sm})\cap Y_i$ has a complement of codimension $\geq 2$ in $Y_i$, $i=1,...,r$. Hence, $q^{-1}(X_{\rm sm})$ has a complement of codimension $\geq 2$ in $Y$, and $q$ is a $G$-equivariant almost principal $\Gamma$-bundle over $X$.
\end{proof}

\begin{lem}\label{Lem_LocalCohomology}
Let $X$ be an affine variety and $Z\subset X$ a closed subvariety such that the restriction morphism $\Oo(X)\rightarrow\Oo(X\setminus Z)$ is an isomorphism. Then, $Z$ is of codimension $\geq 2$ in $X$.
\end{lem}
\begin{proof}
We consider the local cohomology groups $H^i_Z(X,\Oo_X)$ associated with the stucture sheaf $\Oo_X$ and the closed subset $Z\subset X$. By \cite[Ex. III.2.3]{Hartshorne}, we have a long exact sequence of local cohomology 
\begin{center}
$0\rightarrow H^0_Z(X,\Oo_X)\rightarrow H^0(X,\Oo_X)\rightarrow H^0(X\setminus Z,\Oo_X)\rightarrow H^1_Z(X,\Oo_X)\rightarrow H^1(X,\Oo_X)\rightarrow ...$
\end{center}
By our assumptions, the groups $H^0_Z(X,\Oo_X)$ and $H^1_Z(X,\Oo_X)$ are trivial. On the other hand, recall the cohomological interpretation of the depth of an ideal $I\subset \Oo(X)$ (see \cite[Ex. III.3.4]{Hartshorne}):
\begin{center}
$\depth_I(\Oo(X))\geq n \iff H^i_{\Vv(I)}(X, \Oo_X)=0, \forall i<n$,
\end{center}
where $\Vv(I)$ denotes the zero set of $I$ in $X$. Applying this to the ideal $I$ defining $Z$, we obtain the result by virtue of the inequalities
\begin{center}
$\codime_X(Z)\geq\depth_I(\Oo(X))\geq 2$.
\end{center}
\end{proof}

\begin{rem}\label{Rem_AlmostPrinc_CanTakeSmoothLocus}
The proof of the last proposition shows that given a $G$-equivariant diagonalizable almost principal $\Gamma$-bundle $q:Y\rightarrow X$, the restriction $q^{-1}(X_{\rm sm})\rightarrow X_{\rm sm}$ is a $G$-equivariant $\Gamma$-torsor, and $q^{-1}(X_{\rm sm})$ has a complement of codimension $\geq 2$ in $Y$.
\end{rem}

\begin{defn}\label{Def_TypeDiagTorsor}
Let $q:Y\rightarrow X$ be a $G$-equivariant diagonalizable almost principal $\Gamma$-bundle. Then, the natural morphism 
\begin{center}
$\type_G(Y):\hat{\Gamma}\rightarrow \clg^G(X)$, $\chi\mapsto [(q_*\Oo_Y)_\chi]$,
\end{center}
is called the \textit{$G$-equivariant type} of $Y$. The morphism
\begin{center}
$\type(Y):=\phi\circ\type_G(Y):\hat{\Gamma}\rightarrow \clg(X)$,
\end{center}
where $\phi:\clg^G(X)\rightarrow\clg(X)$ is the forgetful morphism, is the \textit{type} of $Y$.
\end{defn}

Proposition \ref{Prop_Forme_Gen_AlmostBundle} gives the general form of a $G$-equivariant diagonalizable almost principal $\Gamma$-bundle over $X$. These objects naturally form a category, taking for arrows the $G\times\Gamma$-equivariant morphisms over $X$. Assume that $\Oo(X)^{*G}\simeq k^*$, and let $q:Y\rightarrow X$ be a $G$-equivariant diagonalizable almost principal $\Gamma$-bundle. Using \ref{EqIsomorphismsClg_Codim2}, \ref{Prop_GenForm_DiagTorsorII}, and \ref{Prop_Forme_Gen_AlmostBundle}, we see that the $G$-equivariant type of $Y$ determines its isomorphism class. However, we don't know a priori if conversely any morphism
\begin{center}
$\hat{\Gamma}\rightarrow\clg^G(X)$
\end{center}
defines a $G$-equivariant diagonalizable almost principal $\Gamma$-bundle over $X$. Indeed, such a morphism defines a quasi-coherent $\hat{\Gamma}$-graded $\Oo_X$-algebra $\Aa^G$ satisfying the two conditions of \ref{Prop_Forme_Gen_AlmostBundle}, but it might not be of finite type. The next proposition gives a sufficient condition to ensure that the algebra $\Aa^G$ defined by any morphism as above is of finite type. Moreover, its proof shows how to construct any $G$-equivariant diagonalizable almost principal bundle over $X$ from the equivariant characteristic space of $(X,x)$. 

\begin{prop}\label{Prop_AlmostPrinc_FiniteType}
Suppose that $\Oo(X)^{*G}\simeq k^*$, and that the equivariant Cox sheaf of $(X,x)$ is of finite type. Consider a quasi-coherent $\hat{\Gamma}$-graded $\Oo_X$-algebra $\Aa^G$ satisfying the two conditions of Proposition \ref{Prop_Forme_Gen_AlmostBundle}. Then, $\Aa^G$ is of finite type as an $\Oo_X$-algebra.
\end{prop}
\begin{proof}
We are going to realize $\Aa^G$ as a subalgebra of invariants by a diagonalizable group. Consider the morphism of diagonalizable groups $\varphi:\Gamma_{\clg^G(X)}\rightarrow\Gamma$ dually defined by the morphism 
\begin{center}
$\type_G(Y):\hat{\Gamma}\rightarrow \clg^G(X)$,
\end{center}
and the $\Gamma_{\clg^G(X)}\times\Gamma$-action on $\hat{X}^G\times\Gamma$ defined on schematic points by $(g,h_1).(x,h_2):=(g.x,\varphi(g)^{-1}h_1h_2)$. Then, we have a good quotient $f:\hat{X}^G\times\Gamma\rightarrow X$ by this action. Indeed, consider the morphisms
\begin{center}
$\hat{X}^G\times\Gamma\xrightarrow{p_{\hat{X}^G}}\hat{X}^G\xrightarrow{q_1}X$
\end{center}
where $p_{\hat{X}^G}$ is the projection, and $q_1$ is the structural morphism of the equivariant characteristic space $\hat{X}^G$ over $X$. Both morphisms are good quotients, and the composition is a good quotient by $\Gamma_{\clg^G(X)}\times\Gamma$, hence $f=q_1\circ p_{\hat{X}^G}$. As a consequence of this factorization, $f$ is in addition a $G$-equivariant diagonalizable almost principal $\Gamma_{\clg^G(X)}\times\Gamma$-bundle over $X$. Performing first the quotient by $\Gamma_{\clg^G(X)}$, we obtain another factorization
\begin{center}
$\hat{X}^G\times\Gamma\xrightarrow{q_2} Y\xrightarrow{q_3}X$
\end{center}
of $f$ into $G$-equivariant diagonalizable almost principal bundles. The $\Oo_X$-algebra $q_{3*}\Oo_Y$ is obtained by computing $\Gamma_{\clg^G(X)}$-invariants: 
\begin{center}
$(f_*\Oo_{\hat{X}^G\times\Gamma})^{\Gamma_{\clg^G(X)}}=(\bigoplus_{(\lambda,[\Ff])\in \hat{\Gamma}\times\clg^G(X)}\Ff^x\lambda)^{\Gamma_{\clg^G(X)}}=\bigoplus_{\lambda\in\hat{\Gamma}}\Ff_\lambda^x$,
\end{center}
where $[\Ff_\lambda]=\type_G(Y)(\lambda)$, $\forall\lambda\in\hat{\Gamma}$. Hence, we obtain
\begin{center}
$q_{3*}\Oo_Y=\Aa^G$,
\end{center}
which proves that this $\Oo_X$-algebra is of finite type.
\end{proof}

Assume from now that $\Oo(X)^{*G}\simeq k^*$, and that the equivariant Cox sheaf of $(X,x)$ is of finite type. By the last proposition and the above discussion, we obtain a one-to-one correspondence between $G$-equivariant diagonalizable almost principal bundles over $X$ and $G$-equivariant diagonalizable torsors over $X_{\rm sm}$. This yields in particular the

\begin{prop}\label{Prop_FormGen_AlmostPrincII}
Suppose that $\Oo(X)^{*G}\simeq k^*$, that the equivariant Cox sheaf of $(X,x)$ is of finite type, and fix an isomorphism $\hat{\Gamma}\simeq \ZZ^r\times\ZZ/n_1\ZZ\times...\times\ZZ/n_s\ZZ$. Then a $G$-equivariant diagonalizable almost principal $\Gamma$-bundle $Y$ over $X$ is defined up to isomorphism by choosing 
\begin{itemize}
\item a $r$-tuple $([\Ff_1],...,[\Ff_r])\in\clg^G(X)\times...\times\clg^G(X)$,
\item a $s$-tuple $([\Ff_{r+1}],...,[\Ff_{r+s}])\in \clg^G(X)[n_1]\times...\times\clg^G(X)[n_s]$,
\end{itemize}
and setting  $Y:=\spec_X\Aa^{G,x}$ where $\Aa^{G,x}$ is the $\hat{\Gamma}$-graded $\Oo_X$-algebra
\begin{center}
$\bigoplus_{(k_i)\in\ZZ^r\times\ZZ/n_1\ZZ\times ... \times\ZZ/n_s\ZZ}\,\,\, (\Ff_1^{\star k_1})^x\star...\star(\Ff_{r+s}^{\star k_{r+s}})^x$.
\end{center}
\end{prop}

\begin{rem}\label{Rem_CanonicalMorphismFromEqCharSpace}
With the assumptions of the last proposition, let $Y=\spec_X(\Aa^{G,x})$ be a $G$-equivariant diagonalizable almost principal $\Gamma$-bundle over $X$. There is a canonical morphism of graded $\Oo_X$-algebras
\begin{center}
$\theta^G:\Aa^{G,x}=\bigoplus_{\lambda\in\hat{\Gamma}}\Ff_\lambda^x\rightarrow\Rr^{G,x}$
\end{center}
whose associated morphism between grading groups is $\type_G(Y)$. Indeed, with the notation of Section \ref{SecCoxRingEq}, consider the projection $\theta^G_U$ associated to the equivalence relation $\sim^G_U$ for an open subset $U\subset X$. The equivalence relations $\sim^G_U$ are compatible with the algebra structures on $\Aa^{G,x}(U)$, and with  restrictions to open subsets. It follows that the family of projections $(\theta^G_U)_{U\subset X}$ defines a family of graded algebras morphisms $\Aa^{G,x}(U)\rightarrow\Rr_X^{G,x}(U)$. These morphisms define all together a morphism $\theta^G$ of graded $\Oo_X$-algebras, the induced morphism between grading groups being $\type_G(Y)$ by construction. This morphism corresponds to an equivariant morphism 
\begin{center}
$\hat{X}^{G,x}\rightarrow Y$
\end{center}
over $X$. In particular, the structural morphism of the $G$-equivariant characteristic space over $X$ factors through any $G$-equivariant diagonalizable almost principal bundle over $X$.
\end{rem}

As a consequence of the last proposition, $\Rr_X^{G,x}$ doesn't depend on the choice of the point $x$ up to isomorphism, so that we can speak of the the equivariant Cox sheaf of $X$ (resp. the equivariant characteristic space of $X$). In fact, we add some flexibility by introducing the following definition.

\begin{defn}
A \textit{$G$-equivariant characteristic space} over $X$ is an algebraic scheme $Z$ over $X$ endowed with an action of a diagonalizable group $\Gamma$ such that there exists a pair $(f,\varphi)$ consisting of
\begin{itemize}
\item an isomorphism $f:Z\rightarrow \hat{X}^G$ over $X$,
\item an isomorphism $\varphi:\Gamma\rightarrow\Gamma_{\clg(X)}$,
\end{itemize}
such that the square 
\begin{center}
	\begin{tikzcd}
		\Gamma\times Z \arrow[r,""] \arrow[d,  "\varphi\times f"] & Z \arrow[d, "f", swap]\\
		\Gamma_{\clg(X)}\times \hat{X}^G \arrow[r, ""] & \hat{X}^G\,
	\end{tikzcd}
\end{center}
commutes, where the horizontal arrows denote the respective actions. A \textit{$G$-equivariant total coordinate space} of $X$ is the affinization of a $G$-equivariant charateristic space. A\textit{ $G$-equivariant Cox ring} of $X$ is the (graded) coordinate ring of a $G$-equivariant total coordinate space.
\end{defn}

\begin{rem}
Equivalently, a characteristic space over $X$ is a $G$-equivariant diagonalizable almost principal bundle over $X$ whose $G$-equivariant type is an isomorphism onto $\clg^G(X)$.  Also, a $G$-equivariant Cox ring of $X$ is unique up to graded automorphism.
\end{rem}

\begin{ex}\label{Ex_EqCox_HomSpace}
Let $G$ be a connected algebraic group, and $G/H$ be a homogeneous space. Consider the faithfully flat morphism $\varphi:H\rightarrow \Gamma_{\hat{H}}$ of algebraic groups induced by the inclusion $k[\hat{H}]\subset \Oo(H)$. This yields an exact sequence of algebraic groups
\begin{center}
$1\rightarrow K:=\ker \varphi\rightarrow H \xrightarrow{\varphi} \Gamma_{\hat{H}}\rightarrow 1$,
\end{center}
where the normal subgroup $K$ of $H$ is the intersection of characters of $H$. Hence $H/K\simeq\Gamma_{\hat{H}}$ is the diagonalizable group whose character group is $\hat{H}$. It follows that we have a factorization
\begin{center}
$G\xrightarrow{q_1} G/K\xrightarrow{q_2} G/H$
\end{center}
of the canonical projection $\pi:G\rightarrow G/H$, and that $q_2:G/K\rightarrow G/H$ is a $G$-equivariant diagonalizable $\Gamma_{\hat{H}}$-torsor. 

On the other hand, there is a natural isomorphism $\hat{H}\rightarrow \pic^G(G/H)$, $\lambda\mapsto [\Ll_\lambda]$, where $\Ll_\lambda$ denotes the $G$-linearized invertible sheaf of sections of the $G$-line bundle $G\times^\lambda \AAA^1$ over $G/H$. Recall that for an open subset $V\subset G/H$ we have $\Ll_\lambda(V)\simeq\Oo_G(\pi^{-1}(V))^{(H)}_\lambda$. This yields the equalities
\begin{align*}
  q_{2*}\Oo_{G/K}(V) & = \Oo_{G/K}(q_2^{-1}(V))\\
  		& = (q_{1*}\Oo_G(q_2^{-1}(V)))^{K}\\
  		& = \Oo_G(\pi^{-1}(V))^{K}\\
  		& = \bigoplus_{\lambda\in\hat{H}} \Oo_G(\pi^{-1}(V))^{(H)}_\lambda\\
    	& \simeq\bigoplus_{\lambda\in\hat{H}}\Ll_\lambda(V),
\end{align*}
where the last equality follows from the fact that every $H/K$-module is a direct sum of eigenspaces. Thus, we have obtained
\begin{center}
$q_{2*}\Oo_{G/K}=\bigoplus_{\lambda\in\hat{H}}\Ll_\lambda$.
\end{center}
Hence, the $G$-equivariant type of $G/K$ defines an isomorphism with $\pic^G(G/H)$, whence $G/K$ is a $G$-equivariant characteristic space of $G/H$. When $G$ is semisimple and simply connected, the equivariant Cox sheaf is canonically isomorphic to the usual one, and we retrieve \cite[Thm 4.5.1.8]{coxrings}.
\end{ex}

\subsection{Equivariant quotient presentations}

When a normal variety $X$ admits a Cox sheaf of finite type, then the characteristic space $\hat{X}$ is naturally a diagonalizable almost principal $\Gamma_{\clg(X)}$-bundle over $X$ such that the regular invertible $\Gamma_{\clg(X)}$-homogeneous functions are constant. This leads to the notion of a \textit{quotient presentation} of $X$ (\cite[4.2.1.1]{coxrings}). We recall this notion and introduce a $G$-equivariant analogue.

\begin{defn}
Let $\Gamma$ a diagonalizable group, and $X$ a normal variety such that $\Oo(X)^*\simeq k^*$. Then, a \textit{quotient presentation} of $X$ is a diagonalizable almost principal $\Gamma$-bundle $q:Y\rightarrow X$ such that the invertible $\Gamma$-homogeneous regular functions on $Y$ are constant. 
\end{defn}

\begin{rem}
In fact, the definition \cite[4.2.1.1]{coxrings} of a quotient presentation is different but equivalent to the one given above. Indeed, a quotient presentation in loc. cit. is a normal variety $Y$ together with a morphism $q:Y\rightarrow X$ verifying
\begin{enumerate}
\item $q$ is a good quotient by a diagonalizable group $\Gamma$.
\item The $\Gamma$-homogeneous invertible regular functions on $Y$ are constant.
\item There exists a $\Gamma$-stable open $Y_0\subset Y$ whose complement is of codimension $\geq 2$, and on which $\Gamma$ acts freely with closed orbits.
\end{enumerate}
Suppose that $q:Y\rightarrow X$ verifies the above conditions, then $q(Y_0)$ is open with a complement of codimension $\geq 2$ in $X$. Also, as $\Gamma$ acts with closed orbits on $Y_0$ and $q$ is a good quotient, the fibers of $q_{\mid Y_0}:Y_0\rightarrow q(Y_0)$ are orbits. Moreover, $Y_0$ has the quotient topology with respect to $q_{\mid Y_0}$. Indeed, consider an open subset of the form $q^{-1}(V)$, where $V$ is a subset of $q(Y_0)$. Because $q$ is a good quotient, we have that $q(Y_0\setminus q^{-1}(V))=q(Y_0)\setminus V$ is closed in $q(Y_0)$, thus $V$ is open in $q(Y_0)$. It follows that $q$ induces a geometric quotient $Y_0\rightarrow q(Y_0)$. Thus, it induces a $\Gamma$-torsor $Y_0\rightarrow q(Y_0)$ because $\Gamma$ acts freely on $Y_0$ (\cite[Chap. 0, 4.0.9]{MumfordGIT}). Hence, $q:Y\rightarrow X$ is a diagonalizable almost principal $\Gamma$-bundle with only constant $\Gamma$-homogeneous invertible regular functions.

Conversely, consider a diagonalizable almost principal $\Gamma$-bundle $q:Y\rightarrow X$ with only constant $\Gamma$-homogeneous invertible regular functions. Then it satisfies all the conditions above by definition, except that $Y$ is not a priori a variety. However, this will be verified in Proposition \ref{Prop_EqPres_Variety} below.
\end{rem}

\begin{defn}
Let $\Gamma$ a diagonalizable group, and $X$ a normal $G$-variety such that $\Oo(X)^{*G}\simeq k^*$. Then, a \textit{$G$-equivariant quotient presentation} of $X$ is a $G$-equivariant diagonalizable almost principal $\Gamma$-bundle $q:Y\rightarrow X$ such that the invertible $G$-invariant $\Gamma$-homogeneous regular functions on $Y$ are constant.
\end{defn}

\begin{prop}\label{PropLASSTypInjFoncInvEq}
Let $\Gamma$ a diagonalizable group, $X$ a normal $G$-variety such that $\Oo(X)^{*G}\simeq k^*$, and $q:Y\rightarrow X$ a $G$-equivariant diagonalizable almost principal $\Gamma$-bundle. Then $q$ is a $G$-equivariant quotient presentation if and only if the $G$-equivariant type of $Y$ is injective.
\end{prop}
\begin{proof}
Indeed, $\type_G(Y)$ is injective if and only if its kernel is trivial. But this kernel contains a non-trivial element if and only if $q_*\Oo_Y$ admits an homogeneous component isomorphic to
$\Oo_X$ endowed with the trivial $G$-linearization. The latter means that this homogeneous component admits an invertible $G$-invariant global section.
\end{proof}

Let $G$ an algebraic group, and $(X,x)$ a pointed normal $G$-variety such that $\Oo(X)^{*G}\simeq k^*$. We define the category of quotient presentations of $X$ by defining a morphism between two quotient presentations $q_i:Y_i\xrightarrow{//\Gamma_i} X$ to be a pair $(f,\varphi)$ consisting of
\begin{itemize}
\item a morphism $f:Y_1\rightarrow Y_2$ over $X$,
\item a morphism $\varphi:\Gamma_1\rightarrow\Gamma_2$,
\end{itemize}
such that the square
\begin{center}
	\begin{tikzcd}
		\Gamma_1\times Y_1 \arrow[r,""] \arrow[d,  "\varphi\times f"] & Y_1 \arrow[d, "f", swap]\\
		\Gamma_2\times Y_2 \arrow[r, ""] & Y_2\,
	\end{tikzcd}
\end{center}
commutes, where the horizontal arrows denote the respective actions. By the preceding section, a $G$-equivariant quotient presentation of $X$ is defined up to isomorphism by a finitely generated subgroup of $\clg^G(X)$. Indeed, the injectivity of the $G$-equivariant type says that a quotient presentation is up to isomorphism the relative spectrum $\hat{X}^G(M)$ of the $M$-graded $\Oo_X$-algebra (of finite type)
\begin{center}
$\Aa^G(M):=\bigoplus_{[\Ff]\in M}\Ff^x$,
\end{center}
where $M$ is a finitely generated subgroup of $\clg^G(X)$. The following Proposition is a generalization of \cite[4.2.1.4]{coxrings}.

\begin{prop}\label{PropFoncteurSubGroupPicG}
Suppose that the equivariant Cox sheaf of $X$ is of finite type, and consider the category of subgroups of $\clg^G(X)$ with morphisms being the inclusions arrows. Then, the association $M\mapsto \hat{X}^G(M)$ defines an essentially surjective contravariant functor to the category of $G$-equivariant quotient presentations of $X$.
\end{prop}
\begin{proof}
We have only to check the functoriality. Consider two finitely generated subgroups $M'\subset M\subset\clg^G(X)$. This gives an exact sequence of diagonalizable groups
\begin{center}
$1\rightarrow \Gamma_{M/M'}\rightarrow\Gamma_M\rightarrow\Gamma_{M'}\rightarrow 1$,
\end{center}
which translates into the factorization of the $G$-equivariant diagonalizable almost principal $\Gamma_M$-bundle $\hat{X}^G(M)\rightarrow X$ in $G$-equivariant diagonalizable almost principal bundles
\begin{center}
$\hat{X}^G(M)\xrightarrow{//\Gamma_{M/M'}}\hat{X}^G(M')\xrightarrow{//\Gamma_{M'}} X$.
\end{center}
The morphism $\hat{X}^G(M)\xrightarrow{//\Gamma_{M/M'}}\hat{X}^G(M')$ is equivariant with respect to the morphism $\Gamma_M\rightarrow\Gamma_{M'}$, and with respect to the $G$-action, whence the claim.
\end{proof}

\begin{prop}\label{Prop_EqPres_Variety}
Consider a $G$-equivariant quotient presentation  $q:Y\rightarrow X$ of $X$. Then $Y$ is a normal variety.
\end{prop}
\begin{proof}
Because $Y$ is a normal algebraic scheme, it is a variety if and only if it is connected. This is verified if and only if $q_*\Oo_Y$ is a sheaf of integral algebras, so that we can suppose that $X$ (hence $Y$) is smooth. By the last proposition $Y\simeq \hat{X}^G(M)$ for a finitely generated subgroup $M\subset\pic^G(X)$ that we can suppose non-trivial. To verify that $Y$ is connected, we adapt the arguments from  \cite[6.3]{HausenHomogeneous}.  Let $Y_1\cup...\cup Y_r$ be the decomposition of $Y$ as a union of its connected components and suppose by contradiction that $r>1$. These components are transitively permuted by $\Gamma_M$ because $X$ is connected. Also by $\Gamma_M$-invariance, we have $q(Y_1)=X$. Let $\Gamma$ be the stabilizer of $Y_1$ in $\Gamma_M$, which is by assumption a proper subgroup. We let
\begin{center}
$j:\Gamma\xhookrightarrow{}\Gamma_M$
\end{center}
be the inclusion. Consider the $\Gamma_M$-equivariant morphism $\psi:Y\rightarrow \Gamma_M/\Gamma$ which maps a point $y\in Y$ to the element of $\Gamma_M/\Gamma$ corresponding to the connected component of $Y$ containing $y$. It follows that we have a cartesian square 
\begin{center}
\begin{tikzcd}
		Y_1\times \Gamma_M \arrow[r, "p_2"] \arrow[d, "/\Gamma"] &  \Gamma_M \arrow[d, "\pi"] \\
		 Y \arrow[r, "\psi"]  & \Gamma_M/\Gamma\,\,,
\end{tikzcd}
\end{center}
where vertical arrows are $\Gamma$-torsors. Also, restricting $q:Y\rightarrow X$ to $Y_1$ yields a $\Gamma$-torsor $Y_1\rightarrow X$. Together with the square above, we obtain a cartesian square
\begin{center}
\begin{tikzcd}
		Y_1\times \Gamma_M \arrow[r, "p_1"] \arrow[d, "/\Gamma"] &  Y_1 \arrow[d, "/\Gamma"] \\
		 Y \arrow[r, "q"]  & X\,\,,
\end{tikzcd}
\end{center}
where all arrows are torsors. Proceeding as in the proof of \ref{Prop_AlmostPrinc_FiniteType}, we obtain the equality
\begin{center}
$\type_G(Y)=\type_G(Y_1)\circ j^\sharp$.
\end{center}
But $\type_G(Y)$ is injective by assumption, thus $j^\sharp$ is necessarily injective. Now $j^\sharp$ is surjective, being the restriction of characters. It follows that $\Gamma=\Gamma_M$, a contradiction. 
\end{proof}

\begin{cor}\label{Prop_CoxG_NormalVar}
The $G$-equivariant Cox ring is a normal integral $\clg^G(X)$-graded $G$-algebra.
\end{cor}

Suppose that $\Oo(X)^*\simeq k^*$. We now give an insight in the structure of the group of units of $\cox^G(X)$. 
As $\cox^G(X)$ is a $\clg^G(X)$-graded algebra, it is also interesting to consider the subgroup of units which are homogeneous. For example, every unit is homogeneous if e.g. $\clg^G(X)$ is torsion-free. Because of the assumption $\Oo(X)^*\simeq k^*$ and the exact sequence (\ref{EqSEClGroupEquivariantOubli}), we have the $\hat{G}$-graded $k$-subalgebra
\begin{center}
$A^G(\hat{G}):=\bigoplus_{\lambda\in \hat{G}} \Gamma(X,\Oo_X(\lambda)^x)$
\end{center}
of $\cox^G(X)$. The group of units of this algebra is obviously isomorphic to $\hat{G}\times k^*$ via the multiplication map.

\begin{prop}\label{PropUnitsCoxEq}
Suppose that $\Oo(X)^*\simeq k^*$. Then
\begin{enumerate}
\item The group of homogeneous units of $\cox^G(X)$ is isomorphic to $\hat{G}\times k^*$ via the multiplication map.
\item If $\Oo(X)\simeq k$, then every unit is homogeneous.
\end{enumerate}
\end{prop}
\begin{proof}
If a homogeneous component $\Ff^x$ of $\Rr_X^G$ admits an invertible global section $f$, then $\Ff^x$ is isomorphic to $\Oo_X$ as a divisorial sheaf. It follows that $f\in A^G(\hat{G})$, whence the first assertion. Now assume $\Oo(X)\simeq k$ and let $f\in\cox^G(X)^*$. Consider the $\Im\phi$-grading on $\cox^G(X)$ induced by the forgetful morphism $\phi$. Then proceeding as in \cite[1.5.2.5]{coxrings}, we obtain that $f$ is $\Im\phi$-homogeneous of degree zero. But the $\Im\phi$-homogeneous component of degree zero is  the $\hat{G}$-graded $k$-subalgebra $\bigoplus_{\lambda\in \hat{G}} \Gamma(X,\Oo_X(\lambda)^x)$ so we conclude as before.
\end{proof}

\subsection{Recognizing an equivariant characteristic space}
\label{Sec_RecognizingAnEquivariantCharacteristicSpace}

Let $G$ an algebraic group and $(X,x)$ a pointed normal $G$-variety such that $\Oo(X)^{*G}\simeq k^*$. Suppose moreover that the $G$-equivariant Cox sheaf is of finite type. In this section, we give a criterion to recognize a $G$-equivariant characteristic space $Y\rightarrow X$, under some mild assumptions on $X$. As defined in Section \ref{Sec_ClassificationAlmostPrincipalBundles}, a $G$-equivariant characteristic space is characterized by its equivariant type being an isomorphism onto $\clg^G(X)$. The equivariant type of $Y$ is injective if and only if $Y$ is a $G$-equivariant quotient presentation of $X$. We now look for an interpretation of the surjectivity. This turns out to be related to the factoriality of the grading on $\Oo(Y)$, we start by recalling this notion.

\begin{defn}\citep[1.5.3]{coxrings}
Let $M$ be a finitely generated abelian group, and $A$ be an $M$-graded integral $k$-algebra.
\begin{itemize}
\item A non-zero non-invertible element $f\in A$ is \textit{$M$-prime} if it is homogeneous and $(f\mid gh)$ with $g,h$ homogeneous implies $f\mid g$ or $f\mid h$.
\item $A$ is \textit{factorially $M$-graded} if every non-zero non-invertible homogeneous element $f$ is a product of $M$-primes.
\item A normal $\Gamma_M$-variety $X$ is \textit{$\Gamma_M$-factorial} if every $\Gamma_M$-invariant Weil divisor on $X$ is principal. 
\end{itemize}
\end{defn}

Let $\Gamma$ a diagonalizable group, and $Y$ a normal quasi-affine $\Gamma$-variety. By \cite[1.5.3.3]{coxrings}, $Y$ is $\Gamma$-factorial if and only if $\Oo(Y)$ is $\hat{\Gamma}$-factorial. In the next proposition, we give further characterizations of $\hat{\Gamma}$-factoriality.

\begin{prop}\label{Prop_LASS_TypeSurj}
Let $\Gamma$ a diagonalizable group, and $q:Y\rightarrow X$ a $G$-equivariant diagonalizable almost principal $\Gamma$-bundle. The following assertions are equivalent
\begin{enumerate}
\item $Y$ is $\Gamma$-factorial.
\item $\type(Y)$ is surjective.
\item The morphism $\gamma:\hat{\Gamma}\rightarrow\clg^\Gamma(Y)$, $\lambda\mapsto[\Oo_Y(\lambda)]$ is surjective.
\end{enumerate}
Moreover, if one of this conditions is satisfied, then $Y$ is quasi-affine.
\end{prop}
\begin{proof}
The assertions $2$ and $3$ are equivalent because of the lemma below. We suppose that the condition $1$ is verified and show $2$. For this, we can suppose that $X$ is smooth, and that $q$ is a $\Gamma$-torsor. With this assumption, the pullback morphism $q^*:\wdiv(X)\rightarrow\wdiv(Y)$ induces an isomorphism onto the subgroup $\wdiv(Y)^{\Gamma}$ which consists of $\Gamma$-invariant Weil divisors on $Y$. By abuse of notation, we also denote $q_*^\Gamma$ the inverse of this isomorphism. Let's consider a $\Gamma$-homogeneous rational function $f$ on $Y$, in other words a rational section of a divisorial sheaf on $X$. Then, we have $q^*(\divi_X(f))=\divi(f)$ and thus $\divi_X(f)=q_*^\Gamma(\divi(f))$. By this last remark we can write explicitly
\begin{center}
$\type(Y):\hat{\Gamma}\rightarrow \pic(X)$, $\lambda\mapsto [\Oo_X(q_*^\Gamma(\divi(f)))]$,
\end{center}
where $f$ is an arbitrary $\Gamma$-homogeneous rational function of degree $\lambda$ on $Y$. But this morphism is surjective because every Weil divisor on $X$ is the direct image by $q_*^\Gamma$ of a $\Gamma$-invariant Weil divisor on $Y$, which is principal by hypothesis. Now we suppose $3$ and show $1$. Consider a $\Gamma$-invariant Weil divisor $D$ on $Y$.  The divisorial sheaf $\Oo_Y(D)$ is naturally $\Gamma$-linearized. Thus, its class in $\clg(Y)$ is trivial by hypothesis. Hence, $D$ is the divisor of a rational function which is necessarily $\Gamma$-homogeneous. We have shown that $Y$ is $\Gamma$-factorial. It remains to show that $Y$ is quasi-affine. As $\type(Y)$ is surjective, we can cover $X$ by affine open subsets of the form $X_f:=X\setminus \supp(\divi_X(f))$, where $f$ is a  $\Gamma$-homogeneous regular function on $Y$. Indeed, the complement in $X$ of any affine open subvariety is the zero set of a global section of some divisorial sheaf. Because of the surjectivity of $\type(Y)$, this section can be viewed as a homogeneous element of $\Oo(Y)$. It follows that $Y$ is covered by the affine open subsets $Y_f=q^{-1}(X_f)$, whence the claim by \cite[II.5.1.2]{EGA}.
\end{proof}

\begin{lem}
Let $\Gamma$ be a diagonalizable group, and $q:Y\rightarrow X$ a $G$-equivariant diagonalizable almost principal $\Gamma$-bundle. Consider the morphism $\gamma:\hat{\Gamma}\rightarrow\clg^{\Gamma}(Y), \lambda\mapsto [\Oo_Y(\lambda)]$. Then, we have the identity
\begin{center}
$\type(Y)=q_*^{\Gamma}\circ\gamma$,
\end{center}
where $q_*^{\Gamma}:\clg^{\Gamma}(Y)\rightarrow\clg(X)$ is the isomorphism provided by \ref{Prop_ClgG_CommutativeSquare}.
\end{lem}
\begin{proof}
To prove this, we can suppose that $X$ is smooth and consider (equivariant) Picard groups instead of (equivariant) class groups. Now, for each character $\lambda\in\hat{\Gamma}$, consider the cartesian square
\begin{center}
\begin{tikzcd}
		Y\times \GG_m \arrow[r, "p_1"] \arrow[d, ""] &  Y \arrow[d, "q"] \\
	    Z:=Y\times^\lambda \GG_m \arrow[r, ""]  & X\,\,,
\end{tikzcd}
\end{center}
defined by the fiber bundle associated with $\GG_m$ endowed with the $\Gamma$-action defined by $\lambda$. Then, $Z\rightarrow X$ is a $\GG_m$-torsor which by construction corresponds to $q_*^{\Gamma}\circ\gamma(\lambda)\in\pic(X)$. But this is also $\type(Y)(\lambda)$, by definition of the type of $Y$.
\end{proof}

\begin{prop}
Suppose that $\cox^G(X)$ is a finitely generated $k$-algebra. Then, the affinization morphism $\hat{X}^G\rightarrow\tilde{X}^G$ is an open immersion and its image has a complement of codimension $\geq 2$ in $\tilde{X}^G$.
\end{prop}
\begin{proof}
Using that the cokernel of $\type(\hat{X}^G)$ has finite exponent (\ref{RemFiniteIndexForgetGLin}), and that $X_f=X_{f^n}$ for any power of a homogeneous section $f\in\Rr_X(X)$, we can proceed as in the end of the proof of \ref{Prop_LASS_TypeSurj} to show that $\hat{X}^G$ is quasi-affine. The claim on codimension is obtained by applying \ref{Lem_LocalCohomology}.
\end{proof}

\begin{thm}\label{Thm_Recognize_EqCharacteristic_Space}
Let $G$ be a connected algebraic group with trivial Picard group, $(X,x)$ be a pointed normal $G$-variety with only constant invertible regular functions, $\Gamma$ be a diagonalizable group, and $q:Y\rightarrow X$ be a $G$-equivariant diagonalizable almost principal $\Gamma$-bundle. The conditions
\begin{enumerate}
\item The group of $\Gamma$-homogeneous invertible regular functions on $Y$ is isomorphic to $\hat{G}\times k^*$.
\item $Y$ is $\Gamma$-factorial.
\end{enumerate}
are necessary and sufficient for $Y$ to be a $G$-equivariant characteristic space of $X$.
\end{thm}
\begin{proof}
Suppose that the above conditions are fulfilled. By the first one, $q$ is a quotient presentation of $X$, hence $\type_G(Y)$ is injective. Moreover, with the assumptions of the statement the exact sequence (\ref{EqSEClGroupEquivariantOubli}) can be written as
\begin{center}
$0\rightarrow\hat{G}\xrightarrow{\gamma} \clg^G(X)\xrightarrow{\phi}\clg(X)\rightarrow 0$.
\end{center}
By the second condition, $\type(Y)$ is surjective (\ref{Prop_LASS_TypeSurj}). Also, the degree zero part of $\Oo(Y)$ with respect to the induced $\clg(X)$-grading is
\begin{center}
$A^G(M):=\bigoplus_{\lambda\in \hat{G}} \Gamma(X,\Oo_X(\lambda)^x)$,
\end{center}
for a certain subgroup $M$ of $\hat{G}$. Proceeding as in the proof of \ref{PropUnitsCoxEq}, we obtain that the group of homogeneous units of $\Oo(Y)$ is isomorphic with $M\times k^*$ via multiplication map. Hence, we have in fact $M=\hat{G}$ by the first condition, so that $\hat{G}\simeq\ker\phi\subset\Im(\type_G(Y))$. In view of the equality $\type(Y)=\phi\circ\type_G(Y)$, we obtain the surjectivity of $\type_G(Y)$. Conversely, if $Y\rightarrow X$ is a $G$-equivariant characteristic space, then the first condition is satisfied by \ref{PropUnitsCoxEq}. The second is also satisfied because $\type_G(Y)$ and $\phi$ are surjective by hypothesis (\ref{Prop_LASS_TypeSurj}).
\end{proof}

\subsection{Equivariant Cox ring of an open $G$-stable subvariety}
\label{Sec_EqCoxRing_GStableOpenSubvariety}
Let $G$ be a connected algebraic group, $(X,x)$ a pointed normal $G$-variety, and $X_0\xhookrightarrow{j} X$ a $G$-stable open subset containing $x$. Suppose moreover that $\Oo_X(X_0)^{*G}\simeq k^*$, so that both equivariant Cox rings of $X$ and $X_0$ are well-defined. Let $(D_i)_{i=1,..,n}$ the possibly empty family of prime divisors lying in the complement of $X_0$, and $(r_i)_{i=1,..,n}$ the associated family of respective canonical sections of the $G$-linearized divisorial sheaves $\Oo_X(D_i)^x$, $i=1,..,n$.

\begin{prop}\label{RelationCoxEqToGInvSubset}
The $k$-subalgebra of $\cox^G(X)$ generated by $r_1,...,r_n$ is polynomial, and there is an isomorphism of graded $k$-algebras
\begin{center}
$\cox^G(X_0)\simeq\cox^G(X)/((r_i-1)_{1\leq i\leq n})$,
\end{center}
where the right-hand side algebra is endowed with the $\clg^G(X_0)$-grading induced by the surjective morphism $j^*:\clg^G(X)\rightarrow\clg^G(X_0)$.
Moreover, $(r_i-1)_{i=1,...,n}$ is a regular sequence in $\cox^G(X)$.
\end{prop}
\begin{proof}
With these assumptions, the exact sequence (\ref{EqSEClGroupEquivariant}) becomes
\begin{center}
$0\rightarrow \bigoplus_{i=1}^n \ZZ D_i \xrightarrow{D\mapsto [D]}\clg^G(X)\xrightarrow{j^*}\clg^G(X_0)\rightarrow 0$.
\end{center}
For all $[\Gg]\in \clg^G(X_0)$, choose $[\Ff]\in \clg^G(X)$ such that $j^*([\Ff])=[\Gg]$. By the above exact sequence, we can write
\begin{center}
$\Rr_X^{G}=\bigoplus_{[\Gg] \in \clg^G(X_0)}\Mm_{[\Gg]}$, where $\Mm_{[\Gg]}:=\bigoplus_{(n_i)\in \ZZ^n}(\Ff\otimes \Oo_X(\sum_{i=1}^n n_iD_i))^x$.
\end{center}
Notice that $\Mm_{[\Oo_{X_0}]}$ is a $\ZZ^n$-graded $\Oo_X$-subalgebra of $\Rr_X^{G}$. Because $k[r_1,...,r_n]$ is a $\ZZ^n$-graded $k$-subalgebra of $\Mm_{[\Oo_{X_0}]}(X)$ spanned by elements of linearly independent degrees, it immediately follows that it is a polynomial subalgebra of $\cox^{G}(X)$. Also, each $\Mm_{[\Gg]}(X)$ is a $\ZZ^n$-graded $k[r_1,...,r_n]$-module, and $\Mm_{[\Gg]}(X_0)$ is the localisation of $\Mm_{[\Gg]}(X)$ at the homogeneous element $r_1...r_n$. Indeed, up to removing a $G$-stable closed subvariety of codimension $\geq 2$, we can suppose that $X$ is smooth and $X_0=X\setminus \cup_{i=1}^n D_i$. As the support of the divisor of zeroes of $r_1...r_n$ on $X$ is  $\cup_{i=1}^n D_i$, we conclude by applying \cite[I.9.3.1]{EGA}. It follows that $\Rr_X^G(X_0)$ is the localization of $\cox^G(X)$ at the homogeneous element $r_1...r_n$. 

We use the notation of \ref{SecCoxRingEq} for the equivalence relations between sections over open subsets of $X_0$ of $G$-linearized divisorial sheaves on $X_0$. Given two homogeneous sections $s_1,s_2\in\Rr_X^G(X_0)$, we have $s_1\sim^G_{X_0} s_2$ if and only if both are $\ZZ^n$-homogeneous elements of some $\Mm_{[\Gg]}(X_0)$ and there exists a Laurent monomial $r_1^{\alpha_1}..r_n^{\alpha_n}$ such that $s_1=s_2r_1^{\alpha_1}..r_n^{\alpha_n}$. By Lemma \ref{UsualTrick}, imposing the relations $r_i=1$ in $\Rr_X^G(X_0)$ is the same as passing to the quotient modulo $\sim^G_{X_0}$. Thus, we obtain an isomorphism 
\begin{center}
$\cox^G(X)[\frac{1}{r_1..r_n}]/((r_i-1)_{1\leq i\leq n})\simeq \cox^G(X_0)$,
\end{center}
which implies the isomorphism in the statement above.

To prove that the sequence $((r_i-1)_{i=1,...,n})$ is regular in $\cox^{G}(X)$, we first prove that each $\Mm_{[\Gg]}(X_0)$ is a free $k[r_1^{\pm 1},...,r_n^{\pm 1}]$-module. An element $t\in\Mm_{[\Gg]}(X_0)$ decomposes uniquely as a sum of homogeneous elements so that we can suppose $t$ homogeneous, say of degree $(d_i)\in\ZZ^n$. The multiplication by $r_1^{d_1}...r_n^{d_n}$ defines an isomorphism of $k$-vector spaces 
\begin{center}
$\Gg^x(X_0)\rightarrow (\Gg\otimes \Oo_X(\sum_{i=1}^n d_iD_i))^x(X_0)$.
\end{center}
Let $(f_i)_i$ a $k$-basis of $\Gg^x(X_0)$, then $t$ decomposes uniquely as a linear combination $t=\sum_i \alpha_i r^{d_1}...r^{d_n}f_i$, whence the free module structure. Now, we claim that for $j=1,...,n$, there is an injective morphism of $k[r_1,...,r_n]$-modules
\begin{center}
$\Mm_{[\Gg]}(X)/((r_i-1)_{i=1,...,j})\xhookrightarrow{}\Mm_{[\Gg]}(X_0)/((r_i-1)_{i=1,...,j})$.
\end{center}
For this, it suffices to apply Lemma \ref{Lem_RegSeq} $j$ times, considering for the $k$th step ($1\leq k\leq j$) the $\ZZ$-grading induced by the projection 
\begin{center}
$\ZZ^n\rightarrow \ZZ$, $(d_1,...,d_n)\mapsto d_k$.
\end{center}
Now, the result follows from the claim, the fact that $\Mm_{[\Gg]}(X_0)$ is a free $k[r_1^{\pm 1},...,r_n^{\pm 1}]$-module, and that the sequence $((r_i-1)_{1\leq i\leq n})$ is regular in $k[r_1^{\pm 1},...,r_n^{\pm 1}]$.
\end{proof}

\begin{lem}\label{UsualTrick}
Let $M$ a graded module over an algebra of Laurent polynomials $k[t_1^{\pm 1},...,t_n^{\pm 1}]$ endowed with the standard $\ZZ^n$-grading. For all $x\in M$, the submodule generated by the family $(t_1^{d_1}...t_n^{d_n}.x-x)_{(d_i)\in \ZZ^n}$ equals the submodule generated by the family $(t_i.x-x)_{1\leq i\leq n}$.
\end{lem}
\begin{proof}
This comes from the identities
\begin{center}
$t_it_j.x-x=t_i.(t_j.x-x)+(t_i.x-x)$ and $t_i^{-1}.x-x=(-t_i^{-1}).(t_i.x-x)$.
\end{center}
\end{proof}

\begin{lem}\label{Lem_RegSeq}
Let $k[r]$ be a polynomial $k$-algebra with its standard $\ZZ$-grading, and $M$ a $\ZZ$-graded $k[r]$-module such that $r$ is a non-zero divisor in $M$. Then, $(r-1)M[1/r]\cap M=(r-1)M$.
\end{lem}
\begin{proof}
The localization morphism provides a graded injective morphism of $\ZZ$-graded $k[r]$-module
\begin{center}
$M=\bigoplus_{d\in\ZZ}M_d\xhookrightarrow{} M[1/r]=\bigoplus_{d\in\ZZ}(M[1/r])_d$.
\end{center}
It suffices to prove that for an element $m\in M[1/r]$ such that $(r-1)m\in M$, we have that $m\in M$. Consider the decomposition $m=m_d+...+m_{d+l}$, $l\geq 0$ as a sum of homogeneous elements. Then, the decomposition of $(r-1)m$ reads
\begin{center}
$(r-1)m=rm_{d+l}+(rm_{d+l-1}-m_{d+l})+...+(rm_{d}-m_{d+1})-m_d$.
\end{center}
Necessarily, each term of this sum lies in $M$. It follows that $m_{d+i}\in M$, $i=0,..,l$ by an immediate induction, whence the result.
\end{proof}

\begin{cor}\label{Cor_OPS_XSmoothProjective}
There exists a smooth complete $G$-variety $X'$ containing $X_{\rm sm}$ as a an open $G$-stable variety. As a consequence, we have
\begin{center}
$\cox^G(X)\simeq\cox^G(X')/((r_i-1)_i)$,
\end{center}
where the $r_i$ are the canonical sections in the $G$-linearized divisorial sheaves associated to the $G$-stable prime divisors lying in $X'\setminus X_{\rm sm}$.
\end{cor}
\begin{proof}
By a theorem of Sumihiro (\cite[Thm. 3]{Sumihiro}), $X$ can be embedded equivariantly as an open subvariety of a normal complete $G$-variety $\overline{X}$. Now, there is a $G$-equivariant resolution of singularities $\varphi:X'\rightarrow \overline{X}$, that is, $X'$ is a smooth $G$-variety and $\varphi$ is a projective birational equivariant morphism that induces an isomorphism over $\overline{X}_{\rm sm}$ (\cite[Proposition 3.9.1 and Theorem 3.36]{KollarSingularities}). As we have $\cox^G(X)\simeq \cox^G(X_{\rm sm})$, the statement follows from \ref{RelationCoxEqToGInvSubset}.
\end{proof}

\begin{prop}\label{PropCoxGFreeModuleCoxG}
Suppose that $X$ is almost homogeneous. Then the polynomial algebra $k[r_1,...,r_n]$ is the subalgebra of invariants $(\cox^G(X))^G$.
\end{prop}
\begin{proof}
Let $f\in(\cox^G(X))^G$ that we can suppose homogeneous because $G$ acts on homogeneous components. Then $\divi_X(f)$ is a $G$-invariant effective divisor on $X$, hence is a positive linear combination of the $D_i$. It follows that $f$ is the product of a monomial in the $r_i$ with a $G$-invariant homogeneous unit. However, a $G$-invariant homogeneous unit is necessarily of trivial $\clg^G(X)$-degree, otherwise the multiplication by such a section would provide an isomorphism of some non-trivial $G$-linearized divisorial sheaf with $\Oo_X$ endowed with the trivial $G$-linearization, a contradiction. We conclude, that $f$ is a non-zero scalar multiple of a monomial in the $r_i$. 
\end{proof}

Suppose that $X$ is almost homogeneous, let $X_0\simeq G/H$ be the dense orbit, and consider the equivariant characteristic space $q:\hat{X}^G\rightarrow X$. In the corollary below, we describe the open subvariety $q^{-1}(X_0)$. In order to state it, we need some preliminary observations. Clearly, $q^{-1}(X_0)$ is a homogeneous space for $G\times \Gamma_{\clg^G(X)}$ since $X_0$ is a homogeneous space for $G$ and $q$ induces is a $\Gamma_{\clg^G(X)}$-torsor over $X_0$. Also, the last proposition implies that we have a $\Gamma_{\clg^G(X)}$-equivariant good quotient $\pi:\tilde{X}^G\rightarrow \AAA_k^n$. Finally, the exact sequence (\ref{EqSEClGroupEquivariant}) yields by duality an exact sequence of algebraic groups
\begin{center}
$1\rightarrow\Gamma_{\hat{H}}\rightarrow\Gamma_{\clg^G(X)}\rightarrow\GG_m^n\rightarrow 1$,
\end{center}
whence a $\Gamma_{\hat{H}}$-torsor $\Gamma_{\clg^G(X)}\rightarrow\GG_m^n$. 

\begin{cor}\label{Cor_StructOpenOrbit}
The good quotient $\pi$ restricts to a $\Gamma_{\clg^G(X)}$-equivariant geometric quotient $q^{-1}(X_0)\rightarrow \GG_m^n$. Moreover, $q^{-1}(X_0)$ is isomorphic as a $G\times\Gamma_{\clg^G(X)}$-variety over $\GG_m^n$ to the associated fiber bundle
\begin{center}
$G/K\times^{\Gamma_{\hat{H}}}\Gamma_{\clg^G(X)}\rightarrow \GG_m^n$,
\end{center}
where $G/K$ is the equivariant characteristic space of $G/H$ (\ref{Ex_EqCox_HomSpace}).
\end{cor}
\begin{proof}
The morphism $\pi$ pulls back the standard coordinates of $\AAA^n_k$ to the $r_i$, whence a $\Gamma_{\clg^G(X)}$-equivariant restriction $q^{-1}(X_0)\rightarrow \GG_m^n$. By virtue of \ref{RelationCoxEqToGInvSubset} and \ref{Ex_EqCox_HomSpace}, the fiber at $(1,...,1)$ of this restriction is isomorphic to $G/K$ as a $G\times\Gamma_{\hat{H}}$-variety. By $\Gamma_{\clg^G(X)}$-equivariance of $\pi$ and surjectivity of the morphism $\Gamma_{\clg^G(X)}\rightarrow\GG_m^n$, the same holds for any fiber of this restriction, whence a geometric quotient. By standard properties of associated fiber bundles, $q^{-1}(X_0)$ is isomorphic to $G/K\times^{\Gamma_{\hat{H}}}\Gamma_{\clg^G(X)}$ as a $G\times\Gamma_{\clg^G(X)}$-variety over $\GG_m^n$.
\end{proof}

The inclusion $k[\hat{H}]\subset\Oo(H)$ yields a morphism of algebraic groups $H\rightarrow \Gamma_{\hat{H}}$. Composed with the injective morphism $\Gamma_{\hat{H}}\rightarrow\Gamma_{\clg^G(X)}$ described above, we obtain a morphism $\varphi:H\rightarrow\Gamma_{\clg^G(X)}$. In the next proposition, we view $H$ as a subgroup of $G\times \Gamma_{\clg^G(X)}$ by identifying it with the image of the morphism $h\mapsto (h,\varphi(h))$. 

\begin{prop}
Let $\hat{x}$ be a point in the fiber $q^{-1}(x)$, where $x\in X_0$ is the point identified with $H/H$. Then, $H$ is the stabilizer of $\hat{x}$ in $G\times \Gamma_{\clg^G(X)}$.
\end{prop}
\begin{proof}
Let $\Ff_1,...,\Ff_m$ be $G$-linearized divisorial sheaves on $X$ whose classes generate the abelian group $\clg^G(X)$. Via restriction of this sheaves to $X_0\simeq G/H$, we obtain $G$-linearized invertible sheaves $\Ll_{\chi_1},...,\Ll_{\chi_m}$ corresponding to certain characters $\chi_i\in\hat{H}\simeq\pic^G(G/H)$. Let $L_{\chi_i}$ be the $G$-linearized line bundle over $X_0$ corresponding to $\Ll_{\chi_i}$, and view $q^{-1}(X_0)$ as a subvariety of $L_{\chi_1}\times_{X_0}...\times_{X_0}L_{\chi_m}$. Obviously, the image of the projection to $G$ of $\stab_{G\times \Gamma_{\clg^G(X)}}(\hat{x})$ has to be included in $H$. Then, consider an element $(h,t)\in\stab_{G\times \Gamma_{\clg^G(X)}}(\hat{x})\subset H\times\Gamma_{\clg^G(X)}$. We have $h.\hat{x}=(\chi_1(h)\hat{x}_1,...,\chi_m(h)\hat{x}_m)=\varphi(h^{-1}).\hat{x}$, which forces $t=\varphi(h)$.
\end{proof}

\subsection{Relation between the equivariant Cox ring and the usual Cox ring}
\label{Sec_RelationCoxG_To_Cox}

Let $G$ be a connected algebraic group acting on a pointed normal $G$-variety $(X,x)$ such that $\Oo(X)^*\simeq k^*$, so that both equivariant and usual Cox sheaves are well defined. In view of Remark \ref{Rem_CanonicalMorphismFromEqCharSpace}, there exists a canonical morphism of graded $\Oo_X$-algebras
\begin{center}
$\theta:\Rr_X^{G,x}\rightarrow \Rr_X^x$,
\end{center}
whose associated morphism between grading groups is $\clg^G(X)\xrightarrow{\phi}\clg(X)$, this is the \textit{forgetful morphism}. 

On the other hand, $\hat{G}$ is a free $\ZZ$-module of finite rank which acts by automorphisms of graded $k$-vector space on $\Rr_X^{G}(U)$, for any open subset $U\subset X$. To see this, consider $[\Ff]\in \clg^G(X)$, $\lambda\in\hat{G}$, and notice that $(\Ff^x\otimes\Oo_X(\lambda))^x\simeq\Ff^x$ as divisorial sheaves. Hence, for a section $s\in\Ff^x(U)$, we define the section
\begin{center}
$\lambda.s\in(\Ff^x\otimes\Oo_X(\lambda))^x(U)$
\end{center}
to be the unique section that corresponds to $s$ modulo $\sim_U$. This assignment defines an action of $k[\hat{G}]$ on $\Rr^{G}_X(U)$ that respects multiplication, so that we get a $k[\hat{G}]$-algebra structure on $\Rr^{G}_X(U)$. In the next propositions we make this structure more explicit, and study the forgetful morphism.

\begin{prop}
The equivariant Cox ring is a free $k[\hat{G}]$-module. More precisely, we have
\begin{center}
$\cox^G(X)=\bigoplus_{[\Ff]\in \Im \phi}M_{[\Ff]}$,
\end{center}
where $M_{[\Ff]}:=\bigoplus_{\mu\in \hat{G}}(\Ff^x\otimes\Oo(\mu))^x(X)$ is $\hat{G}$-graded and free of rank $\dim_k(\Ff^x(X))$ over $k[\hat{G}]$. 
\end{prop}
\begin{proof}
The action of $k[\hat{G}]$ stabilizes the $k$-vector spaces $M_{[\Ff]}$. With this action, they inherit a structure of $\hat{G}$-graded $k[\hat{G}]$-module. Moreover the exact sequence (\ref{EqSEClGroupEquivariantOubli}) becomes
\begin{center}
$0\rightarrow\hat{G}\xrightarrow{\gamma} \clg^G(X)\xrightarrow{\phi} \clg(X)$,
\end{center}
which justifies the decomposition of $\cox^G(X)$ as the direct sum of the modules $M_{[\Ff]}$. In order to check that they are free, let's consider $[\Ff]\in\clg^G(X)$ and the $k$-vector space $\Ff^x(X)$. For any $\lambda\in\hat{G}$, the multiplication by $\lambda$ defines an isomorphism of $k$-vector spaces with $(\Ff\otimes\Oo_X(\lambda))^x(X)$, the inverse being the multiplication by $\lambda^{-1}$. By a simple verification as in the proof of \ref{RelationCoxEqToGInvSubset}, we obtain that $M_{[\Ff]}$ is freely generated over $k[\hat{G}]$ by any $k$-basis of $\Ff^x(X)$. 
\end{proof}

\begin{prop}\label{PropMorphismeCoxGToCox}
The forgetful morphism $\theta:\Rr_X^G\rightarrow \Rr_X$ is an integral extension of graded $\Oo_X$-algebras, whose kernel is the sheaf of ideals generated by the family of global sections $(1-\lambda_j.1)_{1\leq j\leq n}$, where $(\lambda_j)_{1\leq j\leq n}$ is an arbitrary $\ZZ$-basis of $\hat{G}$, and $1\in\Oo(X)$ has trivial degree. Moreover, this family defines a regular sequence in $\Rr_X^G(U)$ for any open subset $U\subset X$, and we have
\begin{center}
$\Im\theta(U)\simeq \Rr_X^G(U)/(1_{|U}-(\lambda_j.1)_{|U})_{1\leq j\leq n}$.
\end{center}
\end{prop}
\begin{proof}
Let $U\subset X$ be a non-empty open subset. Replacing $X$ by $U$ in the last proposition and considering the action of $k[\hat{G}]$ on $\Rr_X^G(U)$, we see that the forgetful morphism corresponds to the quotient of $\Rr_X^G(U)$ by the sub-$k[\hat{G}]$-module spanned by $(f_{\Ff,i}-\lambda.f_{\Ff,i})_{[\Ff]\in\Im\phi,i,\lambda\in\hat{G}}$, where $(f_{\Ff,i})_i$ is a $k$-basis of $\Ff^x(U)$ for all $[\Ff]\in \Im\phi$. Using Lemma \ref{UsualTrick}, this submodule is spanned by the elements of the family $(f_{\Ff,i}-\lambda_j.f_{\Ff,i})_{[\Ff]\in\Im\phi,i,1\leq j\leq n}$ where $(\lambda_j)_{1\leq j\leq n}$ is an arbitrary $\ZZ$-basis of $\hat{G}$. Finally, we observe that this sub-$k[\hat{G}]$-module is the ideal of $\Rr_X^{G}(U)$ generated by $(1_{|U}-(\lambda_j.1)_{|U})_{1\leq j\leq n}$ where $1\in \Oo(X)$ has trivial degree. 

The sequence $(1_{|U}-(\lambda_j.1)_{|U})_{1\leq j\leq n}$ is regular because the sequence $(1-\lambda_j)_{1\leq j\leq n}$ is regular in $k[\hat{G}]$ and $\Rr_X^{G,x}(U)$ is a free $k[\hat{G}]$-module. Now, we show that an element $f\in\Rr_X^x(U)$ is integral over $\Rr_X^{G,x}(U)$. We may assume this element homogeneous, say $f\in\Ff^x(U)$. Then we use the fact that a non-zero power of $\Ff^x$ is linearizable (\ref{RemFiniteIndexForgetGLin}). For the last assertion, we argue exactly as in \cite[1.4.3.5]{coxrings} to prove the isomorphism.
\end{proof}

\begin{cor}\label{CorCoxGToCoxIII}
Suppose that the Picard group of $G$ is trivial, then
\begin{center}
$\cox(X)\simeq \cox^G(X)/(1-\lambda_j.1)_{1\leq j\leq n}$,
\end{center}
where $(\lambda_j)_{1\leq j\leq n}$ is an arbitrary $\ZZ$-basis of $\hat{G}$, and $1\in\Oo(X)$ has trivial degree.
\end{cor}
\begin{proof}
With this assumption, $\clg^G(X)\xrightarrow{\phi}\clg(X)$ is surjective (\ref{RemFiniteIndexForgetGLin}). It follows that $\theta$ is also surjective. We conclude by applying the last proposition.
\end{proof}

\subsection{Pullback of equivariant characteristic spaces}

In this section, we study the pullback of an equivariant characteristic space along an equivariant almost principal bundle. Let $G,H$ be algebraic groups with $G$ connected, $(X,x)$ a pointed normal $G$-variety, $(Y,y)$ a pointed normal $G\times H$-variety, and $\pi_1:Y\rightarrow X$ a $G$-equivariant almost principal $H$-bundle. Suppose that $\Oo(X)^{*G}\simeq\Oo(Y)^{*G\times H}\simeq k^*$, and that $X$ has a $G$-equivariant Cox sheaf of finite type. Then, we claim that there is a cartesian square
\begin{center}
\begin{tikzcd}
		\hat{Y}^{G\times H} \arrow[r, "\pi_2"] \arrow[d, "q_2"] &  \hat{X}^G \arrow[d, "q_1"] \\
	    Y \arrow[r, "\pi_1"]  & X,
\end{tikzcd}
\end{center}
where $q_1$  is a $G$-equivariant characteristic space over $X$, $q_2$  is a $G\times H$-equivariant characteristic space over $Y$, and the horizontal arrows are $G$-equivariant almost principal $H$-bundles. Indeed, because of \ref{Prop_ClgG_CommutativeSquare} we can write
\begin{center}
$\Rr^{G\times H}_Y=\bigoplus_{[\Ff]\in \clg^G(X)}(\pi_1^*\Ff)^y$.
\end{center}
Thus, $\hat{Y}^{G\times H}\rightarrow Y$ is the pullback of $\hat{X}^G\rightarrow X$ along $\pi_1$. As $\pi_2$ induces a $G$-equivariant $H$-torsor over $q_1^{-1}(X_{\rm sm})$, it follows that $\pi_2$ is a $G$-equivariant almost principal $H$-bundle. In fact, $\pi_2$ corresponds to the graded $\Oo_X$-algebra morphism
\begin{center}
$\Rr^G_X=\bigoplus_{[\Ff]\in\clg^G(X)}\Ff^x\rightarrow\pi_{1*}\Rr^{G\times H}_Y$
\end{center}
defined by the natural morphisms  $\Ff^x\rightarrow\pi_{1*}\pi_1^*\Ff^x$ between homogeneous components. As $Y$ admits a $G\times H$-equivariant Cox sheaf of finite type, it also admits a $G$-equivariant Cox sheaf of finite type when $H$ is connected and $\Oo(Y)^{*G}\simeq k^*$ (\ref{PropMorphismeCoxGToCox}). When $H=\TT$ is a torus, we can say more (see also \cite[2.2]{LBraun}):

\begin{prop}\label{Prop_TorsorTorus_SameCoxSheaf}
Suppose that $H=\TT$ is a torus and that $\Oo(Y)^{*G}\simeq k^*$. Then $\hat{X}^G\xrightarrow{q_1} X$ factors through $Y$ via the morphism $\pi_1$, and this factorization realizes $\hat{X}^G$ as the $G$-equivariant characteristic space of $Y$.
\end{prop}
\begin{proof}
In this situation, the exact sequence (\ref{EqSEClGroupEquivariantOubli}) reads
\begin{center}
$0\rightarrow\hat{\TT}\rightarrow\clg^{G\times\TT}(Y)\xrightarrow{\phi}\clg^G(Y)\rightarrow 0$,
\end{center}
and we have an isomorphism $\pi_{1*}^\TT:\clg^{G\times\TT}(Y)\rightarrow\clg^G(X)$, $[\Ff]\mapsto [(\pi_{1*}\Ff)^\TT]$ (\ref{Prop_ClgG_CommutativeSquare}). This allows us to write
\begin{center}
$\Rr^G_X=q_{1*}\Oo_{\hat{X}^G}=\bigoplus_{[\Gg]\in\clg^G(Y)}\bigoplus_{\lambda\in\hat{\TT}}(\pi_{1*}(\Ff\otimes_{\Oo_Y}\Oo_Y(\lambda))^{\TT})^x$,
\end{center}
where for each $[\Gg]\in\clg^G(Y)$, we choose $[\Ff]\in\clg^{G\times\TT}(Y)$ such that $\phi([\Ff])=[\Gg]$. This yields a factorization of $q_1$ through the relative spectrum over $X$ of the degree zero part of the graded algebra $\Rr^G_X$ with respect to the induced $\clg^G(Y)$-grading. This degree zero part is
\begin{center}
$\bigoplus_{\lambda\in\hat{\TT}}\pi_{1*}(\Oo_Y\otimes_{\Oo_Y}\Oo_Y(\lambda))^{\TT}=\bigoplus_{\lambda\in\hat{\TT}}(\pi_{1*}\Oo_Y)^{(\TT)}_\lambda=\pi_{1*}\Oo_Y$.
\end{center}
Hence, the factorization is
\begin{center}
$\hat{X}\xrightarrow{\varphi}Y\xrightarrow{\pi_1}X$,
\end{center}
where $\varphi$ is a good quotient by $\Gamma_{\clg^G(Y)}$. In fact, because $q_1$ is a $G$-equivariant almost principal $\Gamma_{\clg^G(X)}$-bundle, the good quotient $\varphi$ is a $G$-equivariant almost principal $\Gamma_{\clg^G(Y)}$-bundle. For every open subset $V\subset X$, we have
\begin{align*}
  \varphi_*\Oo_{\hat{X}}(\pi_1^{-1}(V)) & = q_*\Oo_{\hat{X}}(V)\\
    	& = \bigoplus_{[\Gg]\in\clg^G(Y)}\bigoplus_{\lambda\in\hat{\TT}}\,(\Ff\otimes_{\Oo_Y}\Oo_Y(\lambda))^y(\pi_1^{-1}(V))^{\TT}\\
    	& = \bigoplus_{[\Gg]\in\clg^G(Y)}\bigoplus_{\lambda\in\hat{\TT}}\,\Ff^y(\pi_1^{-1}(V))^{(\TT)}_\lambda\\
    	& = \bigoplus_{[\Gg]\in\clg^G(Y)}\Ff^y(\pi_1^{-1}(V))\\
    	& \simeq \bigoplus_{[\Gg]\in\clg^G(Y)}\Gg^y(\pi_1^{-1}(V)).
\end{align*}
As $\varphi_*\Oo_{\hat{X}}$ is a quasi-coherent sheaf on $Y$, and $\pi_1$ is affine, this shows that the $G$-equivariant type of $\hat{X}^G$ viewed as a $G$-equivariant almost principal $\Gamma_{\clg^G(Y)}$-bundle over $Y$ is the identity of $\clg^G(Y)$.
\end{proof}

\section{Normal rational $G$-varieties of complexity one}

Let $G$ denote a connected reductive algebraic group, $B$ a Borel subgroup, $T$ a maximal torus in $B$, and $U$ the unipotent part of $B$. In this section, we begin by recalling basic facts on the geometry of normal $G$-varieties of complexity one. Then, we obtain a necessary and sufficient condition for the (equivariant) Cox ring to exist and be finitely generated. Then, we give a presentation by generators and relations of the subalgebra of $U$-invariants of the equivariant Cox ring.

\subsection{Generalities}
\label{Sec_GeneralitiesComplexityOne}

Let ($X,x)$ be a pointed normal $G$-variety. The \textit{complexity} of the $G$-action on $X$ is the minimal codimension of a $B$-orbit. As any two Borel subgroups of $G$ are conjugated, the complexity doesn't depend on the choice of $B$. Suppose that $X$ is of complexity one. By a theorem of Rosenlicht (\cite[5.1]{TimashevBook}) there exists a unique up to isomorphism smooth projective algebraic curve $C$ whose field of rational functions is $k(X)^B$. Let's consider the rational quotient
\begin{center}
$\pi:X\dashrightarrow C$
\end{center}
by $B$ given by the inclusion $k(X)^B\subset k(X)$. We show that this quotient admits a rational section. For this, we use the local structure theorem of Brion, Luna and Vust (\cite[Thm 1.4]{BLV}). It is a very useful tool for the study of varieties with group action. We present a variant of this theorem and a corollary, both due to Knop.

\begin{thm}\cite[1.2]{Knop93}\label{ThmLocalStruct1}
Let $G$ a connected reductive group, $X$ a $G$-variety equipped with an ample $G$-linearized invertible sheaf $\Ll$, and $s\in\Ll(X)^{(B)}$ a $B$-semi-invariant section. Let $P$ the parabolic subgroup stabilizer of $s$ in $\PP(\Ll(X))$, with Levi decomposition $P_u\rtimes L$, $T\subset L$. Then the open subvariety $X_s$ is $P$-stable, and there exists a closed $L$-stable subvariety $Z\subset X_s$ such that the morphism induced by the $P$-action
\begin{center}
$P_u\times Z\rightarrow X_s$
\end{center}
is a $P$-equivariant isomorphism.
\end{thm}

\begin{cor}\cite[§2]{Knop93}\label{ThmLocalStruct2}
Let $G$ a connected reductive group, $X$ a normal $G$-variety, and $P(X)$ the smallest stabilizer of a $B$-stable divisor in $X$, with Levi decomposition $P(X)_u\rtimes L(X)$, $T\subset L(X)$. Then, there exists an affine $L(X)$-stable subvariety $Z$ such that $X_0:=P(X)Z$ is an open affine subvariety of $X$. The natural morphism $P(X)_u\times Z\rightarrow X_0$ is an isomorphism and the derived subgroup of $L(X)$ acts trivially on $Z$.
\end{cor}
Let $X_0\subset X$ be as in the corollary, and $H$ be the kernel of the $L(X)$ action on $Z$. Then $\TT:=L(X)/H$ is a torus acting faithfully on $Z$. The local structure of $Z$ for such an action is well known: up to replacing $Z$ by a smaller $L(X)$-stable subvariety, we have an $L(X)$-equivariant isomorphism $Z\simeq \TT\times C_0$, where $L(X)$ acts on $\TT$ by left multiplication via the projection $L(X)\rightarrow \TT$, and $C_0$ is an algebraic curve. We end up with a $P(X)$-equivariant isomorphism
\begin{center}
$X_0\simeq P(X)_u\times \TT\times C_0$,
\end{center}
which identifies $\pi_{\mid X_0}$ with the projection to $C_0$, and yields the desired rational section.  

Recall that the (equivariant) class group being finitely generated is a necessary condition for the (equivariant) Cox ring to be finitely generated (\ref{RemFinitenessClg}). In fact, we show in the two next propositions that for complexity one normal varieties, this is also sufficient.

\begin{prop}\label{PropClassGroupFinitelyGeneratedC1Rational}
The (equivariant) class group of $X$ is finitely generated if and only if $X$ is a rational variety.
\end{prop}
\begin{proof}
By \ref{RemFiniteIndexForgetGLin}, the equivariant class group is finitely generated if and only if the class group is so. By the localization sequence (\ref{EqSEClGroup}), $\clg(X)$ is finitely generated if and only if  $\clg(X_0)$ is so. But because $P(X)_u\times \TT$ is a rational variety with trivial class group, thus $\clg(X_0)$ is isomorphic to $\clg(C_0)$ via the pullback associated with the projection to $C_0$ (\cite[5.1.2]{LinearizationGBrion}). By the localisation sequence again, we deduce that $\clg(X)$ is finitely generated if and only if $\clg(C)$ is. This last condition is equivalent to $C\simeq \PP^1$, which yields the statement.
\end{proof}

\begin{prop}\label{PropCoxGFinitelyGeneratedC1Rational}
Suppose that $X$ is a rational normal variety of complexity one satisfying $\Oo(X)^{*G}\simeq k^*$. Then $\cox^G(X)$ is finitely generated. If moreover  $\Oo(X)^{*}\simeq k^*$, then $\cox(X)$ is finitely generated.
\end{prop}
\begin{proof}
Replacing $X$ by its smooth locus doesn't change the equivariant Cox ring, hence we can suppose $X$ smooth. With this assumption, the Cox sheaf is of finite type and the equivariant characteristic space exists as a normal variety. Let's embed $\Gamma_{\pic^G(X)}$ in a torus $\TT$ by means of an injective morphism of diagonalizable groups $\varphi:\Gamma_{\pic^G(X)}\rightarrow \TT$ dually defined as follows: choose representatives $(\Ll_i)_{i=1,..,n}$ of elements in a generating set of $\pic^G(X)$ and consider the free $\ZZ$-module $\ZZ^n$ generated by these representatives. Then, define $\varphi^\sharp$ to be the natural surjective morphism
\begin{center}
$\varphi^\sharp:\ZZ^n\rightarrow \pic^G(X)$, $(d_1,...,d_n)\mapsto [\Ll_1^{\otimes d_1}\otimes_{\Oo_X}...\otimes_{\Oo_X}\Ll_n^{\otimes d_n}]$.
\end{center}
Let $\Gamma_{\pic^G(X)}$ act on $\TT$ through $\varphi$, and consider the cartesian square
\begin{center}
	\begin{tikzcd}
		\hat{X}^G\times \TT \arrow[r,"p_{\hat{X}^G}"] \arrow[d,  "/\Gamma_{\pic^G(X)}"] & \hat{X}^G \arrow[d, "/\Gamma_{\pic^G(X)}"]\\
		Y:=\hat{X}^G\times^{\Gamma_{\pic^G(X)}}\TT \arrow[r, "\psi", swap] & X\,,
	\end{tikzcd}
\end{center}
where $\psi:Y\rightarrow X$ is the $G$-equivariant $\TT$-torseur over $X$ whose $G$-equivariant type is $\varphi^\sharp$. By \ref{Rem_CanonicalMorphismFromEqCharSpace}, there is a morphism of graded $\Oo_X$-algebras
\begin{center}
$\theta:\Aa^G:=\bigoplus_{(d_i)\in\ZZ^n}(\Ll_1^x)^{\otimes d_1}\otimes_{\Oo_X}...\otimes_{\Oo_X}(\Ll_n^x)^{\otimes d_n}\rightarrow \Rr^{G,x}$
\end{center}
such that $\theta(U)$ is surjective for any open subvariety $U\subset X$. Hence, to prove that $\cox^G(X)$ is finitely generated, we are reduced to prove that the algebra of regular functions on $Y$ is so. Because $Y$ is a $G$-equivariant $\TT$-torsor over $X$, it is a smooth rational $G\times \TT$-variety of complexity one. Then, a result of Knop \cite[Satz 1]{KnopHibertVierzhen} says that $\Oo(Y)$ is a finitely generated $k$-algebra. The second assertion is then an application of \ref{PropMorphismeCoxGToCox}.
\end{proof}

\begin{rem}\label{RemAlmostHomogeneousC1AreRational}
The second assertion in the last statement is \cite[Thm 4.3.1.5]{coxrings}. Also, notice that almost homogeneous varieties of complexity one are automatically rational. Indeed, we have in this situation a natural injection $k(C)\rightarrow k(G)$ from which we see that $C$ is unirational because $G$ is a rational variety. By Luröth's theorem, one deduces that $C$ is rational and hence so is $X$ in view of the isomorphism after Corollary \ref{ThmLocalStruct2}. Moreover, we have $\Oo(X)^{*G}\simeq k^*$ because $G$ has an open orbit. As a consequence, the equivariant Cox ring of $X$ is well-defined and finitely generated.
\end{rem}

Suppose that $X$ is rational of complexity one and such that $\Oo(X)^*\simeq k^*$, so that $\cox(X)$ is well defined and finitely generated. Because $U$ has trivial character group and Picard group, every divisorial sheaf on $X$ has a unique $U$-linearization, whence a canonical isomorphism
\begin{center}
$\cox^U(X)\simeq\cox(X)$ \quad (\ref{Rem_CoxEqIsomCox_SemisimpleSimplyConnected}).
\end{center}
We consider $\cox(X)$ endowed with this canonical structure of $\clg(X)$-graded $U$-algebra.

\begin{prop}\label{prop_PresCoxU}
Suppose that $\pic(G)=0$. Then, the $k$-algebra $\cox(X)^U$ is a finitely generated normal domain, and there is an isomorphism of $\clg(X)$-graded $k$-algebra
\begin{center}
$\cox(X)^U\simeq\cox^G(X)^U/(1-\lambda_j.1)_{1\leq j\leq n}$,
\end{center}
where $(\lambda_j)_{1\leq j\leq n}$ is an arbitrary $\ZZ$-basis of $\hat{G}$. From this isomorphism, $\cox(X)^U$ is endowed with a structure of $\clg(X)\times\hat{T}$-graded algebra.
\end{prop}
\begin{proof}
The $k$-algebra $\cox(X)^U$ is a normal domain because $\cox(X)$ is so. Moreover, \ref{CorCoxGToCoxIII} yields the isomorphism of the statement, whence the finite generation of $\cox(X)^U$. The right hand side is naturally $\clg(X)\times\hat{T}$-graded, so is the left hand side by transport of structure.
\end{proof}

Corollary \ref{Cor_OPS_XSmoothProjective} can be useful for instance if one is looking for a description $\cox^G(X)$ by generators and relations. Indeed, reusing the notations of this result, $X'$ is a smooth complete rational $G$-variety of complexity one, and $\cox^G(X)$ is the quotient of (the finitely generated) $\cox^G(X')$ by a particularly simple regular sequence. So one may rather look for a description of $\cox^G(X')$. This task might be easier thanks to the following proposition (which is probably known but for which we could not find a reference).

\begin{prop}\label{Prop_SmoothCompleteRational_FreePic}
Suppose that $X$ is a smooth complete rational $G$-variety of complexity one. Then, the abelian groups $\pic(X)$ and $\pic^G(X)$ are free of finite rank.
\end{prop}
\begin{proof}
By \cite[Thm. 2]{Sumihiro} there is a $G$-equivariant projective birational morphism $\varphi:X'\rightarrow X$, where $X'$ is a (projective) normal $G$-variety. This implies that the pullback morphism $\varphi^*:\pic(X)\rightarrow\pic(X')$ is injective. Thus we can suppose that $X$ is projective. Also, using the exact sequence (\ref{EqSEClGroupEquivariantOubli}), we are reduce to prove that $\pic(X)$ is free of finite rank. We look for a non-empty open subvariety $U\subset X$ such that $\Oo(U)^*\simeq k^*$ and $\pic(U)=0$. By the exact sequence (\ref{EqSEClGroup}), finding such a subvariety $U$ yields the result.

Choose a maximal torus $T\subset G$ and a 1-parameter subgroup $\lambda$ of $T$ having the same fixed points in $X$ as $T$. By the next lemma, each $G$-orbit admit only finitely many fixed points. Because $X$ is of complexity one, it follows that the connected components $(C_i)_{0\leq i\leq r}$ of the fixed points variety $X^T=X^\lambda$ are points or curves. Also, the Bialynicki-Birula decomposition \cite[Thm. 13.47]{Milne} says that $X$ is the disjoint union of the subvarieties $X_i$, $i=0,...,r$ such that for all $x\in X_i$, the limit $\lim_{t\rightarrow 0}\lambda(t).x$ exists and lies in $C_i$. Moreover, the maps
\begin{center}
$f_i:X_i\rightarrow C_i, x\mapsto \lim_{t\rightarrow 0}\lambda(t).x,\quad i=0,...,r$
\end{center}
are good quotients for the $\GG_m$-action. We can assume that $X_0$ is an open subvariety. If $C_0$ is a curve, then it is isomorphic to $\PP^1_k$. Indeed, it is a projective normal curve which is rational because $k(C_0)\subset k(X)$. Consider $U:=\pi^{-1}(\AAA^1_k)$. Then $U$ is affine and its coordinate algebra is $\ZZ_{\geq 0}$-graded. Moreover, the homogeneous component of degree zero is $\Oo(\AAA^1_k)$, whence $\Oo(U)^*\simeq k^*$. Let $\Ll$ be a $\GG_m$-linearized invertible sheaf on $U$. By \cite[Cor. 6.4]{BassHaboush}, we have $f^*i^*\Ll\simeq\Ll$ as invertible sheaves on $U$, where $i$ denotes the closed immersion $\AAA^1_k\rightarrow U$ and $f$ is the restriction of $f_0$ to $U$. Thus, $\Ll$ is trivial as an invertible sheaf. This implies that $\pic(U)$ is trivial since the forgetful morphism $\pic^{\GG_m}(U)\rightarrow\pic(U)$ is surjective. If $C_0$ is a point, we let $U:=\pi^{-1}(C_0)$ and one proves similarly as above that $\Oo(U)^*\simeq k^*$ and $\pic(U)=0$.
\end{proof}

\begin{lem}
Consider an homogeneous space $G/H$, and a maximal torus $T\subset G$. Then there are only finitely many $T$-fixed points in $G/H$.
\end{lem}
\begin{proof}
We can suppose that $T$ admit at least one fix point $x$. Choosing $x$ as a base point, we have $T\subset H$. Consider a $T$-fixed point $gH$ in $G/H$. We have $g^{-1}Tg=h^{-1}Th$ for a certain $h\in H$, because any two maximal tori in $H$ are conjugate by an element of $H$. But we can translate $g$ by an element of $H$, and consequently suppose that $g\in N_G(T)$. Thus, $N_G(T)$ acts transitively on the set of fixed points. The stabilizer of $x$ is the centralizer $N_H(T)$, whence the finiteness of the set of fixed points.
\end{proof}

\subsection{The algebra $\cox^G(X)^U$}
\label{SecCoxU}

Let $(X,x)$ be a pointed normal rational $G$-variety of complexity one such that $\Oo(X)^{*G}\simeq k^*$. Suppose that $G$ has trivial Picard group. To study $\cox^G(X)$, it is classical in invariant theory to investigate the more accessible structure of the algebra of $U$-invariant $\cox^G(X)^U$. By \cite[D.5]{TimashevBook}, \ref{PropCoxGFinitelyGeneratedC1Rational}, and \ref{Prop_CoxG_NormalVar}, this subalgebra is a finitely generated normal domain. We give a description by generators and relations of $\cox^G(X)^U$ when $\Oo(X)^*\simeq k^*$. This is a generalization of a previous result by Ponomareva (\cite[Thm 4]{Ponomareva}).

Notice that $\cox^G(X)^U$ is naturally  $\clg^G(X)\times \hat{T}$-graded, the homogeneous elements being global sections of $G$-linearized divisorial sheaves whose divisors of zeroes are the $B$-stable effective divisors of $X$. To see this, consider a $B$-invariant effective divisor $D$, and the divisorial sheaf $\Oo_X(D)^x$ endowed with a $G$-linearization (such a linearization exists by assumption on $\pic(G)$, see \ref{RemFiniteIndexForgetGLin}). We have to verify that the canonical section associated to $D$ spans a $B$-stable line in $\cox^G(X)$. As $\Oo(X)^{*G}\simeq k^*$, every automorphism of $\Oo_X(D)^x$ preserving the $G$-linearization is given by the multiplication of sections by a non-zero constant. It follows that the canonical section associated with $D$ is, up to multiplication by a non-zero constant, the unique global section of this $G$-linearized divisorial sheaf whose divisor of zeroes is $D$. Hence, this section is a $B$-eigenvector.

Consider the maximal open subvariety $V\subset X$ on which the rational quotient $\pi:X\dashrightarrow \PP^1_k$ by $B$ defines a morphism. By \cite[II.7.3.6]{EGA}, the complement of $V$ in $X$ is of codimension $\geq2$. Hence, there exists up to isomorphism a unique divisorial sheaf $\Ff$ whose restriction to $V$ is isomorphic to $\pi^*_{|V}\Oo_{\PP^1_k}(1)$. By  hypothesis on the Picard group of $G$, we endow the sheaf $\Ff$ with a $G$-linearization, and consider the rigidified $G$-linearized divisorial sheaf $\Ff^x$ associated to $\Ff$. After choosing homogeneous coordinates $x,y$ on $\PP^1_k$, the sections $\pi^*_{|V}(x)$ and $\pi^*_{|V}(y)$ defining the morphism $\pi_{|V}$ uniquely extend to global sections $a,b\in\Ff^x(X)$ which are $B$-semi-invariant of the same weight. The pullback of Weil divisors on $\PP^1_k$ is well-defined, this is just the usual pullback of Cartier divisors
\begin{center}
$\pi^*:\wdiv(\PP_k^1)\rightarrow \wdiv(X), p=[\alpha:\beta]\mapsto\divi_X(\beta a-\alpha b)$.
\end{center}
As in the preceding section, consider a $P(X)$-stable open affine subvariety $X_0$ given by the local structure theorem. The complement of $X_0$ in $X$ is a finite union of prime $B$-stable divisors. On the other hand, $\pi_{|X_0}$ is a geometric quotient by $B$ whose schematic fibers are subvarieties of codimension one in $X_0$. We deduce that schematic fibers of $\pi_{|V}$ in general position define $B$-stable prime divisors of $X$ by taking closure. These divisors are nothing but the prime divisors in the image of $\pi^*$.

\begin{defn}\cite[3]{Ponomareva}
The prime divisors in the image of $\pi^*$ are the \textit{parametric divisors}. The finite set of $B$-stable prime divisors that are not parametric is the set of \textit{exceptional divisors}. For an exceptional divisor $E$, the image of $\pi_{|E}$ is either dense or a point in $\PP^1_k$. In this last case, the image point is called \textit{exceptional}.
\end{defn}

\begin{nota}
Let $(x_i)_{i\in I}$ be the finite family of exceptional points with respective homogeneous coordinates $[\alpha_i:\beta_i]$, $i\in I$. For all $i\in I$, let $(E^{x_i}_j)_j$ be the finite family of exceptional divisors sent to $x_i$ by $\pi$. Also, let $(E_k)_k$ be the finite family of exceptional divisors dominating $\PP^1_k$. Finally, equip $\Oo(E^{x_i}_j)$ (resp. $\Oo(E_k)$) with arbitrary $G$-linearizations, let $s_{ij}$ (resp. $s_k$) denote the canonical sections associated with $E^{x_i}_j$ (resp. $E_k$) in the $G$-linearized sheaves $\Oo(E^{x_i}_j)^x$ (resp. $\Oo(E_k)^x$), and let $n_{ij}$ denote the (integral) coefficient of $E^{x_i}_j$ in the divisor $\pi^*(x_i)$. 
\end{nota}

\begin{thm}\label{ThmUinvCoxG}
Suppose that $\Oo(X)^*\simeq k^*$. Then $\cox^G(X)^U$ is generated as a $k[\hat{G}]$-algebra by the elements $a,b,(s_{ij})_{ij}, (s_k)_k$. The ideal of relations contains the following identities
\begin{center}
$\beta_i a-\alpha_i b=\lambda_i\prod_j s_{ij}^{n_{ij}},$
\end{center}
where for all $i\in I$, $\lambda_i$ is a certain character of $G$. If moreover, the condition
\begin{itemize}[label=$(\star)$]
\item the common degree of the sections $a,b$ is $\ZZ$-torsion free in $\clg^G(X)\times\hat{T}$
\end{itemize}
is satisfied, then the above relations generate the whole ideal.
\end{thm}
\begin{proof}
We first prove that the elements $a,b,(s_{ij})_{ij},(s_k)_k$ generate $\cox^G(X)^U$ as a $k[\hat{G}]$-algebra. The divisor of zeroes of a $\clg^G(X)\times\hat{T}$-homogeneous element $s$ can be written as
\begin{center}
$\divi_X(s)=\sum_{ij}p_{ij}E^{x_i}_j+\sum_kp_kE_k+\sum_l q_lD_l$,
\end{center}
where the $D_l$ are $B$-stable parametric divisors. We infer that $s$ is associated, by mean of an homogeneous unit, with an element
\begin{center}
$\prod_{ij}s_{ij}^{p_{ij}}\prod_ks_k^{p_k}\prod_l(v_l a- u_l b)^{q_l}$,
\end{center}
where $[u_l:v_l]$ is a homogeneous coordinate of the point $\pi(D_l)\in\PP^1_k$. By virtue of \ref{PropUnitsCoxEq}, it follows that the elements  $a,b,(s_{ij})_{ij},(s_k)_k$ generate $\cox^G(X)^U$ as $k[\hat{G}]$-algebra. 

By the same arguments, we see that the relations of the statement  hold. Now suppose that the common degree of $a,b$ is $\ZZ$-torsion free. We show that these relations generate the whole ideal. For this, consider an arbitrary relation
\begin{center}
$\lambda\prod_{ij}s_{ij}^{p_{ij}}\prod_ks_k^{p_k}f(a,b)+\mu\prod_{ij}s_{ij}^{q_{ij}}\prod_ks_k^{q_k}g(a,b)+...=0$,
\end{center}
where $f,g...$ are $\ZZ$-homogeneous polynomials in two indeterminates (by assumption on torsion-freeness of the degree of $a,b$), and $\lambda, \mu,...$ are characters of $G$. We claim that by substituting relations from the statement, one can transform a sum
\begin{center}
$\lambda\prod_{ij}s_{ij}^{p_{ij}}\prod_ks_k^{p_k}f(a,b)+\mu\prod_{ij}s_{ij}^{q_{ij}}\prod_ks_k^{q_k}g(a,b)$,
\end{center}
in an expression of the form $\nu\prod_{ij}s_{ij}^{r_{ij}}\prod_ks_k^{r_k}h(a,b)$, where $h$ is a $\ZZ$-homogeneous polynomial in two indeterminates, and $\nu\in \hat{G}$. Iterating this process, we obtain the trivial relation in a finite number of steps which finishes the proof. We set $A=\lambda\prod_{ij}s_{ij}^{p_{ij}}\prod_ks_k^{p_k}f(a,b)$, and $B=\mu\prod_{ij}s_{ij}^{q_{ij}}\prod_ks_k^{q_k}g(a,b)$, then $A/B\in k(X)^B=k(a/b)$. Taking the divisor of zeroes of $A/B$, we obtain an equation of the form
\begin{center}
$\sum_{ij}p_{ij}E^{x_i}_j+\sum_kp_kE_k+\divi_X(f_A(a,b))=\sum_{ij}q_{ij}E^{x_i}_j+\sum_kq_kE_k+\divi_X(f_B(a,b))$
\end{center}
where $f_A$ and $f_B$ are $\ZZ$-homogeneous polynomials in two indeterminates. We see right away that $p_k=q_k$ holds necessarily for all $k$, so that we can put in factor $\prod_ks_k^{p_k}$ in $A+B$. On the other hand, for each exceptional point $x_i$ we get an equality
\begin{center}
$\sum_j(p_{ij}-q_{ij})E^{x_i}_j=\divi_X(\beta_i a- \alpha_i b)^{d_i}$, where $[\alpha_i:\beta_i]=x_i$.
\end{center}
We deduce that $p_{ij}-q_{ij}=d_i n_{ij}$ for all $i,j$. We can suppose  that $d_i\geq 0$, so that $p_{ij}\geq q_{ij}$ for all $j$. Hence by using relations of the statement we write
\begin{center}
$\prod_j s_{ij}^{p_{ij}}=\prod_js_{ij}^{q_{ij}}\prod_js_{ij}^{p_{ij}-q_{ij}}=\lambda_i^{-d_i}(\beta_i a-\alpha_i b)^{d_i}\prod_js_{ij}^{q_{ij}}$.
\end{center}
This allows us to factor $\prod_js_{ij}^{q_{ij}}$ in $A+B$. Doing this for each exceptional point, we see that the remaining factor is an homogeneous expression of the form
\begin{center}
$\lambda'f'(a,b)+\mu'g'(a,b)$,
\end{center}
where $\lambda',\mu'\in \hat{G}$ and $f',g'$ are $\ZZ$-homogeneous polynomials in two indeterminates. We necessarily have $\lambda'=\mu'$, and $\deg f'=\deg g'$, whence the sought reduction.
\end{proof}

\begin{prop}\label{RemUinvCoxG}
The condition $(\star)$ is satisfied in the following situations
\begin{itemize}
\item $X$ is smooth and complete,
\item $G=\TT$ is a torus and $\Oo(X)^{\TT}\simeq k$,
\item $X$ is almost homogeneous.
\end{itemize}
\end{prop}
\begin{proof}
In the first situation, $\pic^G(X)$ is free of finite rank (\ref{Prop_SmoothCompleteRational_FreePic}). Consequently, it suffices to check that the common $\pic^G(X)$-degree of $a,b$ is non-trivial. This is verified because $a,b$ are linearly independent and $\Oo(X)\simeq k$. 

In the second situation, suppose that there exists an integer $n\geq 1$ that kills the common degree of $a,b$. Then, consider a non-empty $\TT$-stable open subvariety $X_0$ equivariantly isomorphic to the product $\TT\times Z$, where $\TT$ acts on itself by left multiplication, and $Z$ is an open subvariety of $\PP^1_k$ on which $\TT$ acts trivially. Restricting the homogeneous sections $a,b$ to $X_0$, we have
$$
\begin{cases}
a,b\in\bigoplus_{\clg^\TT(X)}\Oo(Z)[\hat{\TT}]\\
a^n,b^n\in k, \textrm{ of } \clg^\TT(X)\textrm{-degree zero}\\
\end{cases}
$$
Hence, we necessarily have $a,b\in k$, in a certain $\clg^\TT(X)$-homogeneous component. This is a contradiction as these two sections are linearly independent.

In the third situation, suppose again that there exists an integer $n\geq 1$ that kills the common degree of $a,b$. Then we obtain $a^n, b^n\in\Oo(X)^B$, hence $a^n, b^n\in\Oo(X)^G$. As $G$ is connected, $\cox^G(X)^G$ is integrally closed in $\cox^G(X)$, whence $a,b\in\cox^G(X)^G$. But the homogeneous components of this last algebra are one dimensional $k$-vector spaces (\ref{PropCoxGFreeModuleCoxG}), and $a,b$ are linearly independent, a contradiction. 
\end{proof}

\begin{rem}
In \cite[Thm 4.4.1.6]{coxrings}, Hausen and Süss give a presentation by generators and relations of the Cox ring of a normal rational $\TT$-variety of complexity one satisfying $\Oo(X)\simeq k$. In view of the preceding proposition, Theorem \ref{ThmUinvCoxG} combined with \ref{CorCoxGToCoxIII} yields a similar presentation for the slightly more general case of a normal rational $\TT$-variety of complexity one satisfying $\Oo(X)^{\TT}\simeq k$.
\end{rem}

\subsection{$\cox(X)^U$ as the Cox ring of a $T$-variety of complexity one}
\label{Sex_CoxXU_II}
Let $(X,x)$ a pointed normal rational $G$-variety of complexity one such that $\Oo(X)^*\simeq k^*$. Suppose moreover that $G$ has trivial Picard group, and that the condition $(\star)$ of \ref{ThmUinvCoxG} is satisfied, e.g. $X$ is almost homogeneous (\ref{RemUinvCoxG}). 

In \cite[3.4.2]{coxrings}, are studied the finitely generated integral normal factorially graded $k$-algebras such that the grading is effective, pointed, and of complexity one. Geometrically, these are the coordinate algebras of normal affine varieties $Z$ endowed with an effective action of a diagonalizable group $\Gamma$ such that $Z$ is $\Gamma$-factorial, of complexity one under the action of $\Gamma^{\circ}$, and the invariant regular functions as well as the invertible homogeneous regular functions on $Z$ are constant.  By \cite[4.4.2.3]{coxrings}, an algebra satisfying these conditions is up to graded isomorphism obtained by the 

\begin{cons}\label{ConsRingFactoC1}\cite[3.4.2.1]{coxrings}
Fix integers $r\in\ZZ_{\geq 1}$, $m\in\ZZ_{\geq 0}$, a sequence of integers $n_0,...,n_r\in\ZZ_{\geq 1}$, and let $n:=n_0+...+n_r$. Consider as inputs

\begin{itemize}[label=$\bullet$]
\item A matrix $A:=[a_0,...,a_r]$ with pairwise linearly independent column vectors $a_0,...,a_r\in k^2$.
\item An $r\times (n+m)$ block matrix  $P_0=[L\,\,\, 0_{r,m}]$, where $L$ is an  $r\times n$ matrix  built from the $n_i$-tuples $l_i:=(l_{i1},...,l_{in_i})\in\ZZ^{n_i}_{\geq 1}$, $0\leq i\leq r$, called \textit{exponent vectors}, as below
\[
L = \begin{bmatrix} 
    -l_0& l_1 & \dots & 0 \\
    \vdots & \vdots & \ddots & \vdots \\
    -l_0 &  0  &  \dots   & l_r 
    \end{bmatrix}
\]
\end{itemize}
Now consider the polynomial algebra $k[T_{ij},S_k]$, where $0\leq i\leq r$, $1\leq j\leq n_i$, and $1\leq k\leq m$. For every $0\leq i\leq r$, define a monomial
\begin{center}
$T_i^{l_i}:=T_{i1}^{l_{i1}}...T_{in_i}^{l_{in_i}}$,
\end{center}
whence the name "exponent vector".
Denote $\mathscr{I}$ the set of triples $(i_1,i_2,i_3)$ with $0\leq i_1< i_2< i_3\leq r$, and for all $I\in\mathscr{I}$, consider the trinomial
\[g_I:=\det \begin{bmatrix} 
    T_{i_1}^{l_{i_1}}& T_{i_2}^{l_{i_2}} & T_{i_3}^{l_{i_3}} \\
    a_{i_1} &  a_{i_2}  &  a_{i_3} 
    \end{bmatrix}.
\]
We introduce a grading on $k[T_{ij},S_k]$ by the abelian group $K_0:=\ZZ^{n+m}/\Im(^tP_0)$, where $^tP_0$ is the transpose of  $P_0$. Let $Q_0:\ZZ^{n+m}\rightarrow K_0$ be the projection, and set
\begin{center}
$\deg T_{ij}:=Q_0(e_{ij})$, and $\deg S_k:=Q_0(e_k)$,
\end{center}
where $(e_{ij}, e_k)$ is the standard basis of $\ZZ^{n+m}$. Finally consider the $K_0$-graded $k$-algebra
\begin{center}
$R(A,P_0):=k[T_{ij},S_k]/(g_I)_{I\in\mathscr{I}}$.
\end{center}
\end{cons}

\begin{prop}\label{PropR_A_P_0}
The $k$-algebra $\cox(X)^U$ is isomorphic to an algebra $R(A,P_0)$ constructed as in \ref{ConsRingFactoC1}.
\end{prop}
\begin{proof}
Consider a presentation by generators and relations of $\cox(X)^U$ as provided by \ref{ThmUinvCoxG} and \ref{prop_PresCoxU}. Suppose first that $X$ admits at least two exceptional points. Consider the algebra $R(A,P_0)$ built from the following input data
\begin{itemize}[label=$\bullet$]
\item $m$ is the number of exceptional divisors in $X$ dominating $\PP^1_k$,
\item $a_0,...,a_r$ are homogemeous coordinates on $\PP^1_k$ of the exceptional points, and we can suppose that these coordinates are of the form $[1:0], [0:-1],[\alpha_2:\beta_2],...,[\alpha_r:\beta_r]$,
\item the matrix $A$ is \[
\begin{bmatrix} 
    1 & 0 & \alpha_2 &\dots & \alpha_r \\
    0 & -1 & \beta_2 & \dots & \beta_r 
    \end{bmatrix},
\]
\item for $0\leq i\leq r$, set 
\begin{center}
$l_i:=(n_{i,j})_j$ (notation of \ref{ThmUinvCoxG}).
\end{center}
\end{itemize}
On the other hand, using the relations $a=\prod_j s_{0j}^{n_{0j}}$, and $b=\prod_j s_{1j}^{n_{1j}}$ given by \ref{ThmUinvCoxG}, we can remove $a,b$ from our generating set of $\cox(X)^U$. We end up with the same number of generators as for the $k$-algebras $R(A,P_0)$. The ideal of relations is generated by the following ones
\begin{center}
$\alpha_i \prod_j s_{0j}^{n_{0,j}}-\beta_i \prod_j s_{1j}^{n_{1,j}}=\prod_j s_{ij}^{n_{i,j}}$, $i=2,...,r$.
\end{center}
With regard to the algebra $R(A,P_0)$, we have the identities
\begin{center}
$g_{(i,j,l)}=\alpha_ {jl}g_{(i,j,k)}-\alpha_{ij}g_{(j,k,l)}$\quad (\cite[3.4.2.4]{coxrings}),
\end{center}
where $0\leq i<j<k<l\leq r$, and $\alpha_{ij}:=\det(a_i,a_j)$. From this, we infer that the trinomials $g_{(0,1,i)}$, $2\leq i \leq r$ generate the ideal $((g_I)_{I\in\mathscr{I}})$. But these trinomials are, up to the change of variables $T_{ij}\mapsto s_ {ij}$, the trinomials appearing in relations listed above for $\cox(X)^U$. This finishes the proof in this case.

If $X$ admits at most two exceptional points, then $\cox(X)^U$ is a polynomial algebra. In this case we use almost the same input data for $R(A,P_0)$, except that we consider two column vectors $a_0,a_1\in k^2$ which are homogeneous coordinates of two distinct points in $\PP^1_k$ that contain all the exceptional points, and for the $a_i$ which are not exceptional, we set $l_i:=(1)$.
\end{proof}

\begin{cor}
The spectrum of $\cox(X)^U$ is a complete intersection.
\end{cor}
\begin{proof}
By \cite[3.4.2.3]{coxrings}, the spectrum of any algebra $R(A,P_0)$ constructed as in \ref{ConsRingFactoC1} is a complete intersection, whence the result.
\end{proof}

\begin{cor}
\label{Cor_TotalCoordSpaceC1_OnlyConstantInvertibleFunctions}
The total coordinate space $\tilde{X}$ has only constant invertible functions.
\end{cor}
\begin{proof}
An invertible function $f\in\cox(X)^*$ is $U$-invariant (\cite[4.1.6]{LinearizationGBrion}). Thus, it is constant because any algebra $R(A,P_0)$ has only constant invertible functions (\cite[3.4.2.3]{coxrings}).
\end{proof}

\begin{cor}\label{Cor_CoxXU_CoxRingTVarC1}
The $k$-algebra $\cox(X)^U$ is the Cox ring of a normal variety of complexity one under a torus action.
\end{cor}
\begin{proof}
Proceeding as in \cite[Prop. 5.1]{HausenIteration}, we construct a normal affine rational variety of complexity one under a torus action, whose Cox ring is $\cox(X)^U$.
\end{proof}

\begin{rem}\label{Rem_Inputdata}
Following the proof of \ref{PropR_A_P_0}, the input data to be used in Construction \ref{ConsRingFactoC1} in order to obtain $\cox(X)^U$ can interpreted geometrically:
\begin{itemize}[label=$\bullet$]
\item $m$ is the number of exceptional divisors in $X$ dominating $\PP^1_k$,
\item $a_0,...,a_r$ are homogemeous coordinates on $\PP^1_k$ of the exceptional points $x_0,...,x_r$.
\item the exponent vectors are the vectors formed by the multiplicities of the exceptional divisors in the pullbacks $\pi^*(x_i)$, $i=0,...,r$.
\end{itemize}
\end{rem}

\begin{rem}\label{Rem_Platonic}
Consider the spectrum $\tilde{Y}$ of an algebra $R(A,P_0)$ viewed as the total coordinate space of a normal rational variety $Y$ of complexity one under a torus action. The geometries of $\tilde{Y}$ and $Y$ highly depend on the exponent vectors $l_i$, $i=0,...,r,$ involved in the equations. In \cite{HausenIteration}, the study of singularities on $Y$ leads to the following notion that will be used later: a ring $R(A,P_0)$ is a \textit{Platonic ring} either if $r\leq 1$, or if every tuple $(l_{0i_0},...,l_{ri_r})$ is \textit{Platonic}, i.e. after ordering it decreasingly, the first triple is one of the \textit{Platonic triples}
\begin{center}
$(5,3,2),(4,3,2),(3,3,2),(x,2,2),(x,y,1)$, $x\geq y\geq 1$,
\end{center}
and the remaining integers of the tuple equal one.
\end{rem}

\subsection{Log terminality in the total coordinate space}
\label{Sec_LogTermCoxRing}
Let $G$ be a connected reductive group. In this section, we characterize the log terminality of singularities in the total coordinate spaces of almost homogeneous $G$-varieties of complexity $\leq 1$. One key ingredient is a theorem of Braun (\cite[Thm 3.1]{LBraun}) saying that a total coordinate space $\tilde{X}$ is always $1$-Gorenstein . Thus, applying \cite[6.2.12, 6.2.15]{Ishii} yields the equivalence
\begin{center}
$\tilde{X}$ has log terminal singularities $\iff$ $\tilde{X}$ has rational singularities.
\end{center}

\begin{prop}\label{Prop_SphericalCoxLogTerm}
Let $X$ be a spherical $G$-variety with only constant invertible regular functions. Then, the total coordinate space $\tilde{X}$ has log terminal singularities.
\end{prop}
\begin{proof}
The total coordinate space $\tilde{X}$ is a spherical variety under a connected reductive group $G_1\times\Gamma_{\clg(X)}^\circ$, where $G_1$ is a finite cover of $G$ (see \cite[Thm 4.2.3.2]{coxrings}). As spherical varieties have rational singularities (\cite[Thm. 15.20]{TimashevBook}), this completes the proof.
\end{proof}

\begin{cor}
Let $X$ be a $\QQ$-factorial projective spherical variety. Then $X$ is of Fano type.
\end{cor}
\begin{proof}
Apply \cite[Thm 1.1]{Gongyo} together with the last proposition.
\end{proof}

Now, we consider the total coordinate space of an almost homogeneous $G$-variety of complexity one, and prove the

\begin{thm}\label{Thm_MainTheorem}
Let $X$ be an almost homogeneous $G$-variety of complexity one with only constant invertible regular functions. Then
\begin{center}
$\tilde{X}$ has log terminal singularities $\iff$ $\cox(X)^U$ is a Platonic ring \quad (see \ref{Rem_Platonic}).
\end{center}
\end{thm}

The Cox ring of $X$ is well defined and finitely generated (\ref{RemAlmostHomogeneousC1AreRational}). Moreover, up to replacing $G$ by a finite cover, we can suppose that $\pic(G)=0$, and that $\cox(X)$ is endowed with a compatible structure of $\clg(X)$-graded $G$-algebra (see \cite[Thm. 4.2.3.2]{coxrings}). We choose a Borel subgroup $B$, a maximal torus $T$ in $B$, and denote by $U$ the unipotent part of $B$.

Proving that a spherical variety has rational singularities usually relies on a reduction to the toric case. Similarly we are going to reduce to the case of a normal variety of complexity one under a torus action, and use the following result of Arzhantsev, Braun, Hausen and Wrobel:

\begin{thm}\cite[Thm. 1]{HausenIteration}\label{ThmHausenLogTerminal}
Let $\tilde{Y}$ be a normal affine $\QQ$-Gorenstein rational $\TT$-variety of complexity one such that $\Oo(\tilde{Y})^\TT\simeq k$. Then $\tilde{Y}$ has log terminal singularities if and only if its Cox ring is Platonic.
\end{thm}

Consider the categorical quotient $\tilde{Y}:=\spec\cox(X)^U$ of $\tilde{X}$ by $U$. By \cite[Thm. D.5]{TimashevBook}, $\tilde{Y}$ is a normal variety. Moreover, it is of complexity one under the induced action of the torus $\TT:=T\times\Gamma_{\clg(X)}^\circ$. Indeed, $\tilde{X}$ is almost homogeneous of complexity one under $G\times\Gamma_{\clg(X)}^\circ$.
\begin{lem}
The variety $\tilde{Y}$ satisfies the assumptions of Theorem \ref{ThmHausenLogTerminal}. Moreover, it is $1$-Gorenstein.
\end{lem}
\begin{proof}
As $\tilde{X}$ is almost homogeneous of complexity one under $G\times\Gamma_{\clg(X)}^\circ$, it is a rational variety (\ref{RemAlmostHomogeneousC1AreRational}). Then, we have
\begin{center}
$k(\PP^1_k)\simeq k(\tilde{X})^{B\times\Gamma_{\clg(X)}^\circ}= k(\tilde{X})^{U.\TT}= \textrm{Frac}(\cox(X)^U)^{\TT}=k(\tilde{Y})^{\TT}$,
\end{center}
whence $\tilde{Y}$ is a rational variety (for a proof of the second equality, proceed similarly as in \cite[D.7]{TimashevBook}). By the result of Braun again, $\tilde{Y}$ is $1$-Gorenstein because it can be viewed as a total coordinate space (\ref{Cor_CoxXU_CoxRingTVarC1}). Finally, we verify that $\Oo(\tilde{Y})^{\TT}\simeq k$. We have identifications
\begin{center}
$\Oo(\tilde{Y})^{\TT}\simeq\Oo(\tilde{X})^{B\times\Gamma_{\clg(X)}^\circ}\simeq \Oo(\tilde{X})^{G\times\Gamma_{\clg(X)}^\circ}\simeq k[(r_i)_i]^{\Gamma_{\clg(X)}^\circ}$,
\end{center}
where $k[(r_i)_i]$ is a polynomial $k$-subalgebra in the canonical sections associated to the $G$-stable prime divisors in $X$. Indeed, because $X$ is almost homogeneous, a $G$-fixed homogeneous element in $\cox(X)$ is the product of a monomial in the $r_i$ with a homogeneous unit, and the homogeneous units are constant (\cite[1.5.2.5]{coxrings}). The algebra $k[(r_i)_i]$ is $\clg(X)$-graded, and any non-zero degree associated with a homogeneous element is torsion free in $\clg(X)$. It follows that 
\begin{center}
$k[(r_i)_i]^{\Gamma_{\clg(X)}^\circ}\simeq k[(r_i)_i]_0\simeq k$.
\end{center}
\end{proof}
By \cite[Thm D.5]{TimashevBook}, $\tilde{X}$ has rational singularities if and only if $\tilde{Y}$ do so. Using Theorem \ref{ThmHausenLogTerminal} together with the above lemma yields
\begin{align*}
  \tilde{X} \textrm{ has log terminal singularities} & \iff \tilde{Y} \textrm{ has log terminal singularities} \\
    	& \iff \cox(\tilde{Y}) \textrm{ is a Platonic ring}.
\end{align*}
To finish, we claim that
\begin{center}
$\cox(\tilde{Y})$ is a Platonic ring $\iff$ $\cox(X)^U$ is a Platonic ring.
\end{center}
The algebra $\cox(\tilde{Y})$ is again of the form given by Construction \ref{ConsRingFactoC1}, and the associated exponent vectors have an explicit description in term of the ones associated with $\cox(X)^U$ (see \cite[2.7]{Wrobel1}). Using this description, the claim is verified by means of an easy case by case check. This finishes the proof of the Theorem \ref{Thm_MainTheorem}.

\subsection{Iteration of (equivariant) Cox rings}
\label{Sec_IterationCoxRings}
In \cite{HausenIteration}, Arzhantsev, Braun, Hausen and Wrobel introduced the notion of \textit{iteration of Cox rings}: Let $Z$ be a normal variety with finitely generated Cox ring. If the total coordinate space $\tilde{Z}$ has a non-trivial class group and satisfies $\Oo(\tilde{Z})^*\simeq k^*$, then it has a non-trivial well-defined Cox ring. If the latter is finitely generated, we get a new total coordinate space $\tilde{Z}^{(2)}$, and so on. This iteration process yields a sequence of Cox rings which stops if only if one of the following cases occurs at some step:
\begin{itemize}
\item we obtain a total coordinate space whose Cox ring is not well defined (i.e. there exists $n\geq 0$ such that $\clg(\tilde{Z}^{(n)})$ has a non-trivial torsion subgroup, and $\Oo(\tilde{Z}^{(n)})^*\not\simeq k^*$).
\item we obtain a total coordinate space whose Cox ring is not finitely generated.
\item we obtain a factorial total coordinate space (i.e. with trivial class group).
\end{itemize}
If we never fall in one of the cases above, $Z$ is said to have \textit{infinite iteration of Cox rings}. Otherwise, $Z$ is said to have \textit{finite iteration of Cox rings}, and the last obtained Cox ring is the \textit{master Cox ring}. Toric varieties obviously admit finite iteration of Cox rings with a factorial master Cox ring in one step. More generally, spherical varieties admit finite iteration of Cox rings with a factorial (finitely generated) master Cox ring in at most two steps (\cite[Thm 1.1]{Gagliardi}). 
In this section, we prove the finite iteration of Cox rings in the complexity one case, assuming some technical condition (\ref{Thm_FiniteIterationCoxRings}). We use a similar approach as in \cite{LBraun}, namely we systematically decompose a diagonalizable torsor into a torsor by a torus followed by a torsor by a finite diagonalizable group, and use the fact that characteristic spaces are preserved by torsors under tori (\ref{Prop_TorsorTorus_SameCoxSheaf}).

The above definitions have obvious analogues in the equivariant setting. We then speak of \textit{iteration of equivariant Cox rings}. For a variety $X$ endowed with an action of an algebraic group $G$, we denote $\hat{X}^{G,(i)}$ the $i$th $G$-equivariant characteristic space (if it exists), and $\tilde{X}^{G,(i)}$ the ith $G$-equivariant total coordinate space. The next theorem establishes the finiteness of iteration of equivariant Cox rings for almost homogeneous varieties under a connected reductive algebraic group. This yields in particular the finiteness of iteration of Cox rings for almost homogeneous varieties under a semisimple algebraic group (\ref{Rem_CoxEqIsomCox_SemisimpleSimplyConnected}).

\begin{thm}\label{Thm_Itearation_Eq_CoxRing}
Let $G$ be a connected reductive group and $X$ be an almost homogeneous $G$-variety of complexity one. Then $X$ admits finite iteration of equivariant Cox rings. Moreover, each equivariant Cox ring occuring in the sequence is the Cox ring of a complexity one almost homogeneous $G$-variety. In particular, the equivariant master Cox ring is finitely generated and factorial.
\end{thm}
\begin{proof}
We can suppose that $X$ is a smooth $G$-variety, up to replacing it by its smooth locus. By \ref{PropCoxGFinitelyGeneratedC1Rational} and \ref{RemAlmostHomogeneousC1AreRational}, the $G$-equivariant characteristic space $\hat{X}^G$ exists and $q_1:\hat{X}^G\rightarrow X$ is a $G$-equivariant $\Gamma_{\pic^G(X)}$-torsor. Consider the exact sequence of abelian groups
\begin{center}
$0\rightarrow \pic^G(X)_{\rm tor}\rightarrow\pic^G(X)\rightarrow M\rightarrow 0$,
\end{center}
where $\pic^G(X)_{\rm tor}$ denotes the torsion part of $\pic^G(X)$. This defines a factorization
\begin{center}
$\hat{X}^G\xrightarrow{g_1}X'\xrightarrow{f_1} X$
\end{center}
of $q_1$ into $G$-equivariant diagonalizable torsors. In particular, $g_1$ is a torsor under a torus, and $f_1$ is a torsor under a finite diagonalizable group. Denote $H$ the stabilizer of a base point in the open $G$-orbit of $X$. Then $X$ is a normal $G/H$-embedding, and we denote $D_1,...,D_r$ the $G$-stable prime divisors lying in $X\setminus G/H$. Using the exact sequence (\ref{EqSEClGroupEquivariant}) applied to the restriction $\pic^G(X)\rightarrow\pic^G(G/H)$, we infer that the group $\pic^G(X)_{\rm tor}$ embeds in $\pic^G(G/H)\simeq\hat{H}$. Let $H_1$ be the intersection of the kernels of the characters of $H$ corresponding to elements of $\pic^G(X)_{\rm tor}$. This is a normal subgroup of $H$, and $\pic^G(X)_{\rm tor}\simeq \widehat{H/H_1}$ under the above embedding. Then $G/H_1\rightarrow G/H$ is a diagonalizable torsor under the group $H/H_1$, and $G/H_1$ identifies with the $G$-open orbit in $X'$. Because $X'$ is an almost homogeneous $G$-variety of complexity one, it has a well defined $
G$-equivariant characteristic space. Proposition \ref{PropCoxGFreeModuleCoxG} guarantees that $\Oo(\hat{X}^G)^{*G}\simeq k^*$, thus we can apply \ref{Prop_TorsorTorus_SameCoxSheaf} in order to obtain that $\hat{X}^G$ has also a well defined $G$-equivariant characteristic space which coincides with that of $X'$. This construction can be iterated so that we obtain a commutative diagram
\begin{center}
\begin{tikzcd}
... \arrow[d, ""]                       & ... \arrow[d, "f_3"]      & ... \arrow[ld, "q_3"] \arrow[d, ""]                                                    & ... \arrow[d, ""]                 \\
G/H_2 \arrow[r, "", hook] \arrow[d, ""] & X'^{(2)} \arrow[d, "f_2"] & {\hat{X}^{G,(2)}} \arrow[l, "g_2"] \arrow[ld, "q_2"] \arrow[d, ""] \arrow[r, "", hook] & {\tilde{X}^{G,(2)}} \arrow[d, ""] \\
G/H_1 \arrow[r, "", hook] \arrow[d, ""] & X' \arrow[d, "f_1"]       & \hat{X}^G \arrow[l, "g_1"] \arrow[ld, "q_1"] \arrow[r, "", hook]                       & \tilde{X}^G                       \\
G/H \arrow[r, "", hook]                 & X                         &                                                                                        &                                  
\end{tikzcd}
\end{center}
Moreover, as each $H_i$, $i\geq 1$ is necessarily of finite index in $H$, it follows that this construction gives rise to a descending sequence $(H_i)_{i\geq 1}$ of subgroups of $H$, all of which containing the neutral component $H^\circ$. Thus, this sequence is constant after a certain rank. As a consequence $X$ admits finite iteration of equivariant Cox rings with a finitely generated factorial equivariant master Cox ring.
\end{proof}

Let $G$ be a connected reductive algebraic group, and $X$ be a normal rational $G$-variety of complexity one satisfying the condition
\begin{center}
$(\Cc):$ $X$ and each of its iterated characteristic spaces have only constant invertible functions.
\end{center}
By adapting the argument in the last proof, we show in Theorem \ref{Thm_FiniteIterationCoxRings} that $X$ admits finite iteration of Cox rings. The condition $(\Cc)$ guarantees that $X$ has well defined iterated Cox rings, and is satisfied in the following situations
\begin{itemize}
\item $\Oo(X)\simeq k$ (e.g. $X$ is complete).
\item $X$ is a quasicone.
\item $\Oo(X)^*\simeq k^*$, and $X$ is almost homogeneous.
\item $G=\TT$ is a torus, $\Oo(X)^*\simeq k^*$, and $\Oo(X)^\TT\simeq k$.
\end{itemize}
Indeed, for the first two cases, apply \cite[4.4]{LBraun},  and for the remaining cases, use repeatedly  \ref{Cor_TotalCoordSpaceC1_OnlyConstantInvertibleFunctions}.

\begin{thm}\label{Thm_FiniteIterationCoxRings}
Let $G$ be a connected reductive group, and $X$ be a normal rational $G$-variety of complexity one satisfying condition $(\Cc)$. Then, $X$ admits finite iteration of Cox rings with well defined iterated Cox rings. If moreover $X$ is almost homogeneous, then the master Cox ring is factorial and finitely generated.
\end{thm}
\begin{proof}
Up to replacing $X$ by its smooth locus, we can suppose that $X$ is smooth as this doesn't affect the assumptions and the iteration of Cox rings. By virtue of the condition $(\Cc)$ and 
\ref{PropClassGroupFinitelyGeneratedC1Rational}, the iteration stops at the step $n\geq 0$ if and only if  
\begin{center}
either $\tilde{X}^{(n)}$ is not rational, or $\tilde{X}^{(n)}$ is factorial.
\end{center}
To begin with, let us examine the first iteration in detail. We can suppose that $\pic(X)$ has a non-trivial torsion subgroup, as otherwise $\hat{X}$ is factorial and iteration stops (\cite[1.4.1.5]{coxrings}). Consider the diagonalizable torsor $f_1:X'\rightarrow X$ associated to the inclusion morphism $\pic(X)_{\rm tor}\xhookrightarrow{}\pic(X)$ (\ref{Cor_PicG_RepresentH1}). Also, consider the exact sequence of abelian groups
\begin{center}
$0\rightarrow \pic(X)_{\rm tor}\rightarrow\pic(X)\rightarrow M\rightarrow 0$.
\end{center}
This defines a factorization
\begin{center}
$\hat{X}\xrightarrow{g_1}X'\xrightarrow{f_1} X$
\end{center}
of the characteristic space $q_1:\hat{X}\rightarrow X$ into diagonalizable torsors. In particular, $g_1$ is a torsor under the torus $\Gamma_M$, and $f_1$ is a torsor under the finite diagonalizable group $\Gamma_1$ with character group $\pic(X)_{\rm tor}$. In fact we can suppose, up to replacing $G$ by a connected finite cover, that $f_1$, $g_1$ are a $G$-equivariant diagonalizable torsors (\cite[4.2.3.2]{coxrings}). Also, we can suppose that $X'$ is rational, otherwise $\hat{X}$ is not rational and iteration stops. Then, there is a commutative diagram
\begin{center}
\begin{tikzcd}
X' \arrow[r, "\pi_1", dashed] \arrow[d, phantom] \arrow[d, "f_1"] & \PP^1_k \arrow[d, "\varphi_1"] \\
X \arrow[r, "\pi_0", dashed]                                       & \PP^1_k,           
\end{tikzcd}
\end{center}
where the $\pi_i$ are the respective rational quotients by $B$, and $\varphi_1$ is a geometric quotient by $\Gamma_1$. Notice that the $\pi_i$ are defined at every codimension $1$ point. Let $\Gamma_{1,\rm eff}$ be the quotient of $\Gamma_1$ by the kernel of the $\Gamma_1$-action on $\PP^1_k$, and $d:=\mid\Gamma_{1,\rm eff}\mid$. The locus where $\Gamma_{1,\rm eff}$ acts freely is non-empty open, $\varphi_1$ is étale at each point of this locus and ramified at each point of its complement. Let $x$ be a point 
in $\PP^1_k$, and consider the pullback
\begin{equation}\label{Eq_Formula_Pullbackpi}
\pi_0^*(x)=\sum_{i=1}^{N_x}n_{i,x}E^x_i,
\end{equation}
\begin{center}
\end{center}
where the $E_i^x$ are pairwise distinct $B$-stable prime divisors in $X$, and $n_{i,x}\in\ZZ_{>0}$, $i=1,...,N_x$. Inspired by Construction \ref{ConsRingFactoC1}, the vector $(n_{i,x})_{i=1,...,N_x}$ is called the \textit{exponent vector} associated with $x$. On the other hand, consider the set theoretic fiber $\varphi_1^{-1}(x)=\lbrace x'_1,...,x'_{l_x}\rbrace$. Then, we have
\begin{equation}\label{Eq_Formula_Pullbackphi}
\varphi_1^*(x)=\frac{d}{l_x}x'_1+...+\frac{d}{l_x}x'_{l_x}.
\end{equation}
Indeed, the ramification indices associated to the $x'_k$, $k=1,...,l_x$, are equal and their sum is $d$ since the finite extension of fields of rational functions defined by $\varphi_1$ is Galois of degree $d$. By definition, $\varphi_1$ is ramified over $x$ if and only if $d/l_x>1$. Now we compute the pullbacks $\pi_1^*(x'_k)$. Since $f_1$ is étale, there are no multiplicities for the prime divisors occuring in $f_1^*(E^x_i)$. Moreover, denoting $\Sigma_{x,i}$ the set of prime divisors occuring in the support of $f_1^*(E^{x}_i)$, the morphism $\pi_1$ induces an equivariant map of transitive $\Gamma_1$-sets
\begin{center}
$\Sigma_{x,i}\rightarrow \lbrace x'_1,...,x'_{l_x}\rbrace$.
\end{center}
It follows that we have $f_1^*(E^x_i)=\sum_{k=1}^{l_x}(E^{x'_k,1}_i+...+E^{x'_k,m_i}_i)$, where the $E^{x'_k,j}_i$ are pairwise distinct $B$-stable prime divisors in $X'$ which are respectively sent to $x'_k$ by $ \pi_1$, and $m_i$ is the cardinality of the orbit $\stab_{\Gamma_1}(x'_1).E^{x'_k,1}_i$. Using the commutativity of the above square and the formulas (\ref{Eq_Formula_Pullbackpi}) and (\ref{Eq_Formula_Pullbackphi}), we obtain
\begin{equation}\label{Eq_Formula_ExponentVector}
\pi_1^*(x'_k)=\sum_{i=1}^{N_x} \frac{n_{i,x}}{d/l_x}(E^{x'_k,1}_i+...+E^{x'_k,m_i}_i).
\end{equation}
It follows from this computation that if $\varphi_1$ is ramified over $x$, then it is an exceptional point, and the coordinates of the associated exponent vector have a non-trivial greatest common divisor. We call \textit{primitive} an exponent vector whose coordinates have a trivial greatest common divisor. By the above formula, the exponent vectors associated to the points $x'_k$ lying in the fiber $\varphi_1^{-1}(x)$ are equal to each other. Also, if the exponent vector associated with $x$ is primitive, then the exponent vectors associated with the points $x'_k$ are also primitive. The idea of the proof is roughly to observe that there is a finite number of non-primitive exponent vectors, and that this number decreases along the iteration process (up to identifying equal exponent vectors). Then, assuming infinite iteration Cox rings easily yields a contradiction. We now give the details

Thanks to condition $(\Cc)$ and Proposition \ref{Prop_TorsorTorus_SameCoxSheaf}, the above construction can be iterated together with the iteration of Cox rings. At the $n$th step of the iteration of Cox rings, we obtain  a diagram
\begin{center}
\begin{tikzcd}
X^{'(n)} \arrow[d, "\pi_n", dashed] \arrow[r, "f_n"] & ... \arrow[r, "f_3"]       & X^{'(2)} \arrow[d, "\pi_2", dashed] \arrow[r, "f_2"] & X' \arrow[d, "\pi_1", dashed] \arrow[r, "f_1"] & X \arrow[d, "\pi_0", dashed] \\
\PP^1_k \arrow[r, "\varphi_n"]                       & ... \arrow[r, "\varphi_3"] & \PP^1_k \arrow[r, "\varphi_2"]                       & \PP^1_k \arrow[r, "\varphi_1"]                 & \PP^1_k,                   
\end{tikzcd}
\end{center}
where $X^{'(i)}$ and $\hat{X}^{(i)}$ have the common characteristic space $\hat{X}^{(i+1)}$ ($i=1,...,n$), and the $f_i$ are $G$-equivariant $\Gamma_i$-torsors ($\Gamma_i:=\pic(X'^{(i-1)})_{\rm tor}$, $i=1,...,n-1$). Here, the $\Gamma_i$ must be non-trivial because otherwise iteration would stop before the $n$th step. Notice that for the torsors $f_i$ to be $G$-equivariant, we have possibly replaced $G$ by the choice of a finite cover, however the morphisms in the above diagram don't depend on this choice. By repeatedly applying Lemmas \ref{Lem_PhiUnramified_TorsionFreePicardGroup} and \ref{Lem_LiftGroupActions} below, we obtain respectively the following facts:
\begin{itemize}
\item The morphisms $\varphi_{i}$ are ramified for $i=1,...,n-1$. 
\item The composition $F_n:=f_1\circ...\circ f_{n-1}\circ f_n$ is a $G$-equivariant torsor under the action of a finite group $S_n$. Moreover, the actions of $S_n$ and $G$ commute.
\end{itemize}
In particular, $F_n$ is finite étale and $\Phi_n:=\varphi_1\circ...\circ\varphi_{n-1}\circ\varphi_n$ is a geometric quotient by $S_n$. Now, in view of the commutativity of the diagram
\begin{center}
\begin{tikzcd}
X^{'(n)} \arrow[r, "\pi_n", dashed] \arrow[d, "F_n"] & \PP^1_k \arrow[d, "\Phi_n"] \\
X \arrow[r, "\pi_0", dashed]                         & \PP^1_k,                    
\end{tikzcd}
\end{center}
the exponent vector associated with a point in a fiber $\Phi_n^{-1}(x)$ can be obtained by using the same computation as the one leading to formula (\ref{Eq_Formula_ExponentVector}). It follows that the exponent vectors associated to points in a common fiber $\Phi_n^{-1}(x)$ are equal to each other. We say that two exponent vectors are \textit{equivalent} if they are associated with points in a common fiber $\Phi_n^{-1}(x)$, this defines an equivalence relation. Consider the exceptional points $x_1,...,x_r\in\PP^1_k$ associated with $\pi_0$ and such that each associated exponent vector is not primitive. Let $u_n$ be the number of classes of non-primitive exponent vectors defined by pullbacks of points of $\PP_k^1$ by $\pi_n$. We have $u_0=r$ and $(u_n)_n$ is clearly decreasing with values in $\NN$. Suppose by contradiction that we have infinite iteration of Cox rings, which implies that $u_n$ is defined for all $n\in \NN$. Then, this sequence clearly admits a limit $l\geq 0$. We cannot have $l=0$ because there would exist an integer $n\geq 0$ such that $\varphi_n$ is unramified. Thus, there is an exceptional point $x_k$, $1\leq k\leq r$, for which the greatest common divisor of the coordinates of the associated exponent vector is unbounded, a contradiction. 

To finish the proof of the finiteness of the iteration of Cox rings, we now give the two lemmas mentionned above:

\begin{lem}\label{Lem_PhiUnramified_TorsionFreePicardGroup}
Suppose that $\varphi_1$ is not ramified. Then $\varphi_1$ is the identity of $\PP^1_k$. Moreover, $X'$ and $\hat{X}$ have torsion-free Picard groups.
\end{lem}
\begin{proof}
If $\varphi_1$ is not ramified, it is a finite surjective étale morphism, whence an isomorphism since $\PP^1_k$ is simply connected. This means that $\Gamma_1$ acts trivially on $\PP^1_k$, thus $\varphi_1=\Id_{\PP^1_k}$. For the second statement, we show first that the torsion freeness of $\pic(\hat{X})$ implies the torsion freeness of $\pic(X')$. Using condition $(\Cc)$, the exact sequence (\ref{EqSEClGroupEquivariantOubli}) applied to $\hat{X}$ viewed as a $\Gamma_M$-variety gives
\begin{center}
$0\rightarrow M\rightarrow \pic^{\Gamma_M}(\hat{X})\rightarrow \pic(\hat{X})\rightarrow 0$.
\end{center}
This exact sequence guarantees that the torsion subgroup of $\pic^{\Gamma_M}(\hat{X})$ embeds in that of $\pic(\hat{X})$. To finish, we have an isomorphism $\pic^{\Gamma_M}(\hat{X})\simeq\pic(X')$ (\ref{Prop_ClgG_CommutativeSquare}).

Now we prove that $\pic(\hat{X})$ is torsion free. By hypothesis, we have a commutative triangle
\begin{center}
\begin{tikzcd}
X' \arrow[rd, "\pi_1", dashed] \arrow[d, "f_1"] &         \\
X \arrow[r, "\pi_0", dashed]                  & \PP^1_k,
\end{tikzcd}
\end{center}
and we claim that all but a finite number of $B$-stable prime divisors $E$ in $X$ satisfy $f_1^*(E)=E'$ where $E'$ is a $B\times\Gamma_1$-stable prime divisor in $X'$. Indeed, it suffices to consider an affine $B$-stable open subvariety $X_0\subset X$ as in the paragraph after Corollary \ref{ThmLocalStruct2} and the geometric quotient $X_0\rightarrow C_0\subset\PP^1_k$ by $B$ induced by the restriction of $\pi_0$. Then $\pi_1$ also induces a geometric quotient $f_1^{-1}(X_0)\rightarrow C_0$ by $B$. For all $x\in C_0$, the pullbacks $\pi_0^*(x)=E^x$ and $\pi_1^*(x)=E^{'x}$ are $B$-stable prime divisors. It follows that the support of $f_1^*(E^x)$ is $E^{'x}$ and there is no multiplicity as $f_1$ is étale. As all but a finite number of $B$-stable prime divisors in $X$ are pullbacks of points of $C_0$, the claim follows. Now, because $g_1$ is a torsor under a torus, we conclude that for all but a finite number of $B$-stable prime divisors $E$ in $X$, we have the property
\begin{center}
$q_1^*(E)=\hat{E}$, where $\hat{E}$ is a $B\times\Gamma_M$-stable prime divisor in $\hat{X}$.
\end{center}
For any prime divisor $E$ in $X$, we have $q_1^*(E)=\divi(s_E)$, where we view the canonical section $s_E$ associated with $E$ as a regular function on $\hat{X}$. Hence, all but a finite number of $B\times\Gamma_M$-stable prime divisors in $\hat{X}$ are principal. Let $\lbrace E_1,...,E_r \rbrace$ be the finite set of $B$-stable prime divisors in $X$ not satisfying the above property. We have $\divi(s_{E_i})=q_1^*(E_i)=\hat{E}_i^1+...+\hat{E}_i^{l_i}$ where the $\hat{E}_i^k$ are $B\times\Gamma_M$-stable prime divisors in $\hat{X}$ ($1\leq k\leq l_i$). Consider the $B\times\Gamma_M$-stable open subvariety $\hat{X}_0\subset \hat{X}$ obtained by removing the supports of the $q_1^*(E_i)$, $i=1,...,r$. The exact sequence (\ref{EqSEClGroup}) reads
\begin{center}
$0\rightarrow \Oo(\hat{X}_0)^*\rightarrow \bigoplus_{i,k} \ZZ \hat{E}_i^k \rightarrow\pic(\hat{X})\rightarrow\pic(\hat{X}_0)=0.$
\end{center}
Indeed, all the $B\times\Gamma_M$-stable prime divisors in $\hat{X}_0$ are principal whence $\pic(\hat{X}_0)=0$.  Up to multiplication by a non-zero constant, the units of $\Oo(\hat{X}_0)$ are Laurent monomials in the canonical sections $s_{E_i}$, $i=1,...,r$. Hence, the cokernel of the morphism $\Oo(\hat{X}_0)^*\rightarrow \bigoplus_{i,k} \ZZ \hat{E}_i^k$ is a torsion-free abelian group isomorphic to $\pic(\hat{X})$.
\end{proof}

\begin{lem}\label{Lem_LiftGroupActions}
Consider as above the $G$-equivariant diagonalizable $\Gamma_1$-torsor $f_1:X'\rightarrow X$. Suppose that there is a $G$-equivariant torsor $f_0:X\rightarrow Y$ under a finite group $S_0$ such that the actions of $G$ and $S_0$ commute. Then $F_1:=f_0\circ f_1$ is a torsor under a finite group $S_
1$ which is an extension of $S_0$ by $\Gamma_1$. Moreover, $F_1$ is $G$-equivariant and the respective actions of $S_1$ and $G$ on $X'$ commute.
\end{lem}
\begin{proof}
We adapt the arguments of \cite[4.2.4.1]{coxrings}.
Let $N_{\Aut(X')}(\Gamma_1)$ be the normalizer of $\Gamma_1$ viewed as a subgroup of $\Aut(X')$, and $\Phi\in N_{\Aut(X')}(\Gamma_1)$. Then, $f_1\Phi:X'\rightarrow X$ is a $\Gamma_1$-invariant morphism, whence the existence of a morphism $\varphi$ such that following diagram commutes
\begin{center}
\begin{tikzcd}
X' \arrow[r, "\Phi"] \arrow[d, "f_1"] & X' \arrow[d, "f_1"] \\
X \arrow[r, "\varphi"]                & X.                  
\end{tikzcd}
\end{center}
It is readily checked that $\varphi\in\Aut(X)$, and that the map $\beta:N_{\Aut(X')}(\Gamma_1)\rightarrow\Aut(X), \Phi\mapsto\varphi$ is a group morphism. We show that this morphism is surjective. Consider an automorphism $\varphi\in\Aut(X)$, and let $\varphi^*X'$ denote the pullback of $X'$ along $\varphi$. This defines an isomorphism $u:\varphi^*X'\rightarrow X$. Moreover, the structural morphism $\varphi^*X'\rightarrow X$ is a $\Gamma_1$-torsor which has the same type (\ref{Def_TypeDiagTorsor}) as $f_1$ up to an automorphism of $\hat{\Gamma}_1$. Indeed, the pullback by $\varphi$ induces an automorphism of $\pic(X)_{\rm tor}=\hat{\Gamma}_1$. It follows from this last fact and \ref{Prop_FormGen_AlmostPrincII} that there exists an isomorphism $v:X'\rightarrow \varphi^*X'$ over $X$ compatible with the $\Gamma_1$-action (here we used that $\Oo(X)^*\simeq k^*$). Summarizing, we obtain a commutative diagram
\begin{center}
\begin{tikzcd}
X' \arrow[rd] \arrow[r, "v"] & \varphi^*X' \arrow[d] \arrow[r, "u"] & X' \arrow[d, "f_1"] \\
                             & X \arrow[r, "\varphi"]                  & X,    
\end{tikzcd}
\end{center}
where $uv\in N_{\Aut(X')}(\Gamma_1)$ and $\beta(uv)=\varphi$, which proves the surjectivity. The kernel of $\beta$ consists of the elements of $N_{\Aut(X')}(\Gamma_1)$ which are automorphisms of $X'$ viewed as a variety over $X$. Obviously, $\Ker\beta$ contains $\Gamma_1$. Conversely, any element $v\in\Ker\beta$ induces an automorphism of the field of rational functions $k(X')$ viewed as a $k(X)$-algebra. As the extension $k(X)\xhookrightarrow{} k(X')$ is Galois with Galois group $\Gamma_1$, we conclude that $\ker \beta=\Gamma_1$. We have obtained an exact sequence of groups (see also \cite[Proposition 2.1]{GebhardMartin})
\begin{center}
$1\rightarrow\Gamma_1\rightarrow N_{\Aut(X')}(\Gamma_1)\xrightarrow{\beta}\Aut(X)\rightarrow 1$.
\end{center}
Let $S_1$ be the preimage of $S_0\subset\Aut(X)$ by $\beta$. This is a finite group acting on $X'$ which is an extension of $S_0$ by $\Gamma_1$. Also, it is directly checked that the quotient of $X'$ by $S_1$ exists and is a torsor which identifies with the morphism $F_1$. 

It remains to prove that the actions of $G$ and $S_1$ commute. Let $s_1\in S_1(k)$, $g\in G(k)$, $x\in X(k)$. Then we have 
\begin{center}
$f_1s_1gs_1^{-1}g^{-1}(x)=s_0gs_0^{-1}g^{-1}f_1(x)=f_1(x)$,
\end{center}
where $s_0:=\beta(s_1)$. It follows that $s_1gs_1^{-1}g^{-1}\in\Gamma_1(k)$, whence a continuous map $G(k)\rightarrow\Gamma_1(k)$ taking an element to its commutator with $s_1$. Because $G$ is connected, this map is constant. Taking its value at $e_G$, we conclude that the actions of $G$ and $S_1$ commute.
\end{proof}
We finally have to prove the last claim of the theorem. By construction, the $G$-varieties $X^{'(i)}$ are of complexity one. They are moreover almost homogeneous if $X$ is so. The claim follows from this observation and \ref{PropCoxGFinitelyGeneratedC1Rational}, \ref{RemAlmostHomogeneousC1AreRational}.
\end{proof}

\begin{rem}
There exists a normal rational $\GG_m$-variety of complexity one with only constant regular functions, such that the master Cox ring is not finitely generated. Consider indeed the two normal affine rational $\GG_m$-surfaces defined by the $\QQ$-divisors
\begin{center}
$D_1:=\frac{1}{3}[0]+\frac{1}{3}[1]+\frac{1}{3}[\infty]$, and $D_2:=-\frac{1}{3}[0]-\frac{1}{3}[1]-\frac{1}{3}[\infty]$.
\end{center}
on $\PP^1_k$ (see \citep{FlennerZaidenberg}). It is easy to verify that these two surfaces glue to a $\GG_m$-surface $X$ with only constant regular functions. The exponent vectors defined by the rational quotient $X\dashrightarrow \PP^1_k$ by $\GG_m$ are $(3,3)$, $(3,3)$, $(3,3)$. Hence, $\tilde{X}$ is not rational in view of the criterion \cite[5.8]{HausenIteration}. This implies that $\cox(\tilde{X})$ is not finitely generated.
\end{rem}

\begin{rem}
Let $G$ be a connected reductive group, and $X$ be a normal rational $G$-variety of complexity one satisfying the condition
\begin{center}
$(\Cc^G):$ $\Oo(\hat{X}^{G,(i)})^{*G}\simeq k^*$, for $i\geq 0$.
\end{center}
For instance, a normal rational $G$-variety of complexity one satisfies $(\Cc^G)$ if it is complete, or a quasicone, or almost homogeneous. The Theorem \ref{Thm_FiniteIterationCoxRings} has an obvious equivariant analogue in this setting and its proof adapts as well. More precisely, we obtain that $X$ admits finite iteration of $G$-equivariant Cox rings, with well defined $G$-equivariant iterated Cox rings. The proof is in fact easier as the iterated equivariant characteristic spaces are canonically endowed with a $G$-action.
\end{rem}

\nocite{*}
\bibliographystyle{apa}
\bibliography{biblio/biblio}

\end{document}